\documentclass[a4paper]{siamltex}
\usepackage{amsfonts}
\usepackage[dvips]{epsfig,graphics, color}
\usepackage{amssymb}
\input{epsf}

\usepackage{mathbbold}
\usepackage{mathrsfs}
\usepackage{amsmath}
\usepackage{subfig}
\newtheorem{remark}{Remark}[section]
\newtheorem{algorithm}{Algorithm}[section]
\newtheorem{assumption}{Assumption}[section]

\begin{document}

\title{Adaptive Finite Element Approximations for Kohn-Sham Models
\thanks{{This work was partially
supported by  the Funds for Creative Research Groups of China under grant 11021101,
the National Basic Research Program of China under grant
2011CB309703, the National Science Foundation of China
under   grants 11101416 and 91330202, the National 863 Project of China
under grant 2012AA01A309, and the National Center for Mathematics and Interdisciplinary Sciences of Chinese Academy of Sciences.}}}

\author{Huajie Chen\thanks{Department of Mathematics, Technische Universit\"{a}t M\"{u}nchen, Germany ({chenh@ma.tum.de}).}
\and Xiaoying Dai\thanks{LSEC, Institute of Computational Mathematics and Scientific/Engineering Computing, Academy of
Mathematics and Systems Science, Chinese Academy of Sciences, China ({daixy@lsec.cc.ac.cn}).}
\and Xingao Gong\thanks{ Department of
Physics, Fudan University, Shanghai 200433, China
(xggong@fudan.edu.cn).}
\and  Lianhua He\thanks{LSEC, Institute of Computational Mathematics and Scientific/Engineering Computing, Academy of
Mathematics and Systems Science, Chinese Academy of Sciences, China ({helh@lsec.cc.ac.cn}).}
\and Aihui Zhou\thanks{LSEC, Institute of Computational Mathematics and Scientific/Engineering Computing,
Academy of Mathematics and Systems Science, Chinese Academy of Sciences, China ({azhou@lsec.cc.ac.cn}).}}
\date{}
\maketitle

\begin{abstract}
The Kohn-Sham model is a powerful,
widely used approach for computation of ground state electronic energies and densities
in chemistry, materials science, biology, and nanoscience. In this paper, we study  adaptive finite element
approximations for the Kohn-Sham  model.
Based on the residual type a posteriori error estimators proposed in this paper, we introduce
an adaptive finite element algorithm with
a quite general marking strategy and prove the convergence of the adaptive finite element approximations.
Using D{\" o}rfler's marking strategy, we then get
the convergence rate and quasi-optimal complexity.
We also carry out several typical numerical experiments that not only support our theory,
but also show the robustness and efficiency of the adaptive finite element computations in electronic structure calculations.
\end{abstract}\vskip 0.2cm

\vskip 0.2cm


\section{Introduction} \label{sec-introduction}
\setcounter{equation}{0}

The Kohn-Sham density functional model is a powerful,
widely used approach for computation  of ground state electronic energies and densities
in chemistry, materials science, biology, and nanosciences.
Consider a molecular system consisting of $M$ nuclei of charges
$\{Z_1,\cdots,Z_M\}$ located at the positions $\{{\bf R}_1,\cdots,{\bf R}_M\}$
and $N$ electrons in the non-relativistic and spin-unpolarized setting.
By density functional theorem (DFT)  \cite{hohenberg-kohn64,kohn-sham65}, the ground state solutions of the system may be obtained by solving the
 lowest $N$ eigenpairs of the following Kohn-Sham equation
\begin{eqnarray}\label{problem-eigen-intro}
\left\{ \begin{array}{rcl} \left(-\frac{1}{2}\Delta + V_{\rm ext}(x)+
\frac{1}{2} \int_{\mathbb{R}^3}
\frac{\rho(y)}{|x-y|}dy +  V_{\rm xc}(\rho)\right)\phi_i &=& \mu_i\phi_i\quad\mbox{in}~\mathbb{R}^3,
\quad i=1,2,\cdots, N,\\[1ex] \displaystyle
\int_{\mathbb{R}^3}\phi_i\phi_j &=& \delta_{ij},
\end{array} \right.
\end{eqnarray}
where $\displaystyle V_{\rm ext}(x)=-\sum_{k=1}^M\frac{Z_k}{|x-{\bf R}_k|}$ is
the electrostatic potential generated by the nuclei,
$\displaystyle \rho(x)=\sum_{i=1}^N|\phi_i(x)|^2$ is the electron density,
and $V_{\rm xc}(\rho)$ denotes the exchange-correlation potential.
%
%Consider a molecular system consisting of $M$ nuclei of charges
%$\{Z_1,\cdots,Z_M\}$ located at the positions $\{{\bf R}_1,\cdots,{\bf R}_M\}$
%and $N$ electrons in the non-relativistic and spin-unpolarized setting.
%The ground state solutions of the system can be obtained by minimizing the Kohn-Sham energy functional
%\begin{eqnarray}\label{eq-energy-intro-all}
%E(\{\phi_i\}_{i=1}^N)& =& \frac{1}{2}\sum_{i=1}^{N} \int_{\mathbb{R}^3}|\nabla\phi_i(x)|^2dx +
%\int_{\mathbb{R}^3}V_{\rm ext}(x)\rho(x)dx \nonumber \\
%&& + \frac{1}{2}\int_{\mathbb{R}^3}\int_{\mathbb{R}^3}
%\frac{\rho(x)\rho(y)}{|x-y|}dxdy + \int_{\mathbb{R}^3}e_{\rm xc}(\rho(x))dx \quad\quad
%\end{eqnarray}
%with respect to the Kohn-Sham orbitals $\{\phi_i\}_{i=1}^{N}$ under the orthogonality constraints
%\begin{equation*}
%\int_{\mathbb{R}^3}\phi_i\phi_j=\delta_{ij}, \quad 1\leq i,j\leq N,
%\end{equation*}
%where $\displaystyle V_{\rm ext}(x)=-\sum_{k=1}^M\frac{Z_k}{|x-{\bf R}_k|}$ is the electrostatic potential generated by the nuclei,
%$\displaystyle \rho(x)=\sum_{i=1}^N|\phi_i(x)|^2$ is the electron density,
%and $e_{\rm xc}(\rho)$ denotes the exchange-correlation energy per unit volume in an electron gas with density $\rho$.

Since the core electrons do not participate in the chemical binding and remain almost unchanged, a pseudopotential approximation is usually resorted to in practical computations of the Kohn-Sham equation, which is to replace the   Coulomb potential of the nucleus and the effects of the core electrons  by  an effective ionic potential acting on the valence electrons.
  Therefore, under the pseudopotential framework,  only valence electrons are involved.
The
pseudopotential consists of two terms: a local component $V_{\rm loc}$ (whose associated operator is the multiplication
by the function $V_{\rm loc}$) and a nonlocal component $V_{\rm nl}$
(an operator whose expression is given in Section \ref{sec-preliminaries}).
The resulted equation is still \eqref{problem-eigen-intro} but $V_{\rm ext} (x) = V_{\rm loc}(x) + V_{\rm nl}(x)$, $N$ now being  the number of valence electrons,  and  $\{\phi_i\}^N_{i=1}$  being  the set of the pseudo-orbitals of the valence electrons.

We understand that the Kohn-Sham approach achieves so far the best balance between accuracy and efficiency
among all the different formalisms of electronic structure theory,
and simulations of large-scale material systems with Kohn-Sham DFT are still computationally very demanding
(say, thousands of electrons or more). As a result,
efficient numerical algorithms that can be scalable on parallel computing platforms
are desirable to enable DFT calculations at larger scale and for more complex systems.
We see that real-space techniques and methods for electronic structure calculations have been
derived much attention from scientific and engineering computing communities and remarkably developed
during the last two decades, among which the finite element method possesses several significant advantages \cite{beck00,fang-gao-zhou12,pask01,pask-sterne05,torsti06,tsuchida-tsukada95}.
Although the finite element method employs more degrees of freedom than that of traditional methods like plane waves and Gaussians, it results in sparse algebraic eigenvalue problems and thus it is scalable on parallel computing platforms due to the strictly local basis functions, it is variational, and it is  friendly  to implement
 adaptive refinement approaches. Consequently, the computational accuracy and efficiency  of the finite element approximations can be well controlled.
%
%Although the finite element method uses more degrees of freedom than that of traditional methods like plane waves and Gaussians,
%strictly local basis functions produce well structured sparse Hamiltonian matrices; arbitrary boundary conditions can
%be easily incorporated; it is scalable on parallel computing platforms;
%more importantly, it is  variational and relatively straightforward to implement adaptive refinement
%techniques for describing regions around nuclei or chemical bonds where the electron
%density varies rapidly, while treating the other zones with a coarser description, by
%which computational accuracy and efficiency can be well controlled.

We observe  that even in the pseudopotential setting,
the eigenfunctions of \eqref{problem-eigen-intro} still vary rapidly around nuclei or chemical bonds
\cite{beck00,dai-gong-yang-zhang-zhou11,gong-shen-zhang-zhou08}.
% even though the wave function of any Coulombic many-electron system has good regularity (c.f.\cite{fournais-hoffmann-ostenhof09}).
Hence it is also natural to apply adaptive finite element (AFE) approaches
to improve the approximation accuracy and reduce the computational cost.
Indeed, we see that AFE computations have been quite successfully used in solving
Kohn-Sham equations and electronic structure calculations.
Tsuchida and Tsukada  combined the finite element method with the adaptive curvilinear coordinate approach for electronic structure calculation of some molecules \cite{tsuchida-tsukada96,tsuchida-tsukada98};
Shen and Zhang introduced some adaptive tetrahedral finite element disretizations in their theses
\cite{shen-thesis05,zhang-thesis07} and calculated several typical molecular systems efficiently
\cite{gong-shen-zhang-zhou08,shen-zhou06,zhang-shen-zhou-gong08,zhang08};
Bylaska et.al used adaptive piecewise linear finite element method on completely unstructured simplex meshes to resolve
the rapid variation electronic wave functions around atomic nuclei \cite{bylaska09};
Dai et.al designed some parallel adaptive and localization based finite element algorithms for
typical quantum chemistry and nanometer material computations containing more than one thousand atoms
using tens of hundreds of processors on computer cluster \cite{dai-thesis08, dai-gong-yang-zhang-zhou11,dai-shen-zhou08,dai-zhou08};
Gavini et.al constructed a finite element mesh using unstructured coarse-graining technique
and computed materials systems \cite{motamarri12,suryanarayana-gavini-etal10};
Yang successfully scaled their AFE simulations to over 6000 CPU cores on the Tianhe-1A supercomputer in his thesis
\cite{yzhang-thesis11}. The AFE simulations carried out in this paper also show the robustness and efficiency
of the AFE computations in electronic structure calculations. We may refer to \cite{fattebert-hornung-wissink07,torsti06} and references cited therein for other interesting discussions on
 adaptive finite element method (AFEM).

We see that it is significant to understand the mechanism of AFE computations, analyze the AFE approximations of Kohn-Sham equations, and give a mathematical justification of the AFE algorithm.
We note that the AFE computations are based on some a posteriori error estimators and  there are
a little work concerning analysis of the a posteriori error estimators and convergence of AFE approximations for DFT.
In \cite{chen-he-zhou09,chen-he-zhou10}, the authors of this paper considered the nonlinear eigenvalue problems
derived from the orbital-free DFT and obtained the convergence and optimal complexity of the AFE algorithm.
We understand that the orbital-free DFT is viewed as a simplification of the Kohn-Sham DFT, in which only one eigenpair is involved.
In this paper, we shall propose and analyze two AFE algorithms for Kohn-Sham DFT calculations and study the associated convergence and quasi-optimal complexity.
%Indeed, this is one of a series of work of
%our group on  AFE computations for
%electronic structures.

Let us now give an   informal description of the main results of this paper. We propose and analyze two  AFE algorithms:
Algorithm \ref{algorithm-AFEM} and Algorithm \ref{algorithm-AFEM-con-rate}, which are based on the residual type
a posteriori error estimators.  We show the a posteriori error estimates (see Theorem \ref{theorem-posteriori}) and prove that
\begin{itemize}
\item
Under some reasonable assumptions, all limit points of the AFE approximations of the ground state solutions are ground state solutions
(see Theorem \ref{theorem-convergence-adaptive}).
\item
Under other reasonable assumptions, some eigenpairs (in particular, ground state solutions)
can be well approximated by AFE approximations with some convergence rate (see
Theorem \ref{theorem-convergence-rate}).
\end{itemize}
In addition, we also study quasi-optimal complexity of AFE approximations (see Theorem \ref{theorem-optimal-complexity}).

We mention that Algorithm \ref{algorithm-AFEM} and Algorithm \ref{algorithm-AFEM-con-rate} may be viewed as some
extensions of associated existing algorithms for linear elliptic partial differential
equations of second order and have been in fact used for years,  for  instance, in package RealSPACES
(Real Space Parallel Adaptive Calculation of Electronic Structure) of the State Key Laboratory of Scientific and Engineering Computing,
Chinese Academy of Sciences. As we see, the numerical analysis for AFE approximation has been also derived much attention from the mathematical community.
Since Babu{\v s}ka and Vogelius \cite{babuska-vogelius84} gave an analysis of an AFEM for linear
symmetric elliptic problems in one dimension, there has been much investigation on the convergence and complexity of AFEMs in
literature (see, e.g., \cite{binev-dehmen-devore04,cascon-kreuzer-nochetto-siebert08,dai-xu-zhou08,dorfler96,
garau-morin-zuppa09,stevenson07} and the references
cited therein).
In the context of the finite element approximations of linear eigenvalue problems,
in particular, we see that there are  a number of works concerning a posteriori error estimates  \cite{ becker-rannacher01, dai-he-zhou12, duran-padra-rodriguez03,heuveline-rannacher01,larson00,mao-shen-zhou06,verfurth96}, AFEM
convergence \cite{dai-xu-zhou08,garau-morin11,garau-morin-zuppa09,giani-graham09,he-zhou11} and complexity
\cite{dai-he-zhou12, dai-xu-zhou08, garau-morin11,he-zhou11}.

However, there are several crucial difficulties in numerical analysis of the Kohn-Sham equation: it is  a nonlinear eigenvalue problem whose eigenvalues may be degenerate, and a number of eigenpairs must be involved; the associated energy functional is  nonconvex  with respect to density $\rho$,  as a result, there is no uniqueness result for
the ground state solutions; the energy functional
is invariance under unitary transforms, which  also induces redundancy of the ground state solutions.
To handle these difficulties arising from the Kohn-Sham equations,
we shall present some sophisticated arguments and consider the convergence under the distance between solution sets;
investigate the convergence rate and optimal complexity under certain inf-sup assumption;
and exploit the relationship between the finite element nonlinear eigenvalue
approximations and the associated finite element boundary value approximations.
 Thanks to our previous works \cite{chen-gong-he-yang-zhou10, chen-he-zhou09, chen-he-zhou10, dai-he-zhou12, dai-xu-zhou08, he-zhou11, zhou04, zhou07}  where the perturbation argument was
introduced for analyzing  AFEM of eigenvalue problems and
the compact approach
was specialized for handling  the nonlinear effects, combining the  crucial technical results proposed also in this paper, we are then able to analyze our adaptive finite element
algorithms for Kohn-Sham equations, prove the convergence and get the complexity.

The rest of this paper is organized as follows. In Section \ref{sec-preliminaries},
we provide some preliminaries for Kohn-Sham DFT problem setting and residual type a posteriori error estimator based AFE methods.
We prove the convergence of AFE approximations in Section \ref{sec-adaptive} and
analyze the convergence rate and optimal complexity of an
AFE algorithm in Section \ref{sec-optimal}. In Section \ref{sec-numerical}, we present some numerical
experiments that support the theory. Finally, we give some concluding remarks.

\section{Preliminaries} \label{sec-preliminaries}
\setcounter{equation}{0}
Physically, the Kohn-Sham model is set in $\mathbb{R}^3$. However,
due to the exponential decay
of the ground state wavefunction of the Schr\"{o}dinger equation  (c.f., e.g., \cite{agmon81, yesrentant10}) and the fact that Kohn-Sham model is an approximation of Schr\"{o}dinger equation, $\mathbb{R}^3$ is
usually replaced by some polyhedral  domain $\Omega\subset\mathbb{R}^3$  in practical computations for Kohn-Sham equation.

For $\kappa\in \mathbb{R}^{N\times N}$, we denote its Frobenius norm by $|\kappa|$.
For $p\geq 1$ and $s\geq 0$, we denote by $W^{s,p}(\Omega)$ the standard Sobolev spaces
with the induced norm $\|\cdot\|_{s,p,\Omega}$ (see, e.g. \cite{adams75,ciarlet78}).
For $p=2$, we denote by $H^s(\Omega)=W^{s,2}(\Omega)$ with the norm $\|\cdot\|_{s,\Omega}=\|\cdot\|_{s,2,\Omega}$, and
$H^1_0(\Omega)=\{v\in H^1(\Omega): v\mid_{\partial\Omega}=0\}$, where $v\mid_{\partial\Omega}=0$ is understood in the sense of trace.
The space $H^{-1}(\Omega)$, the dual of $H^1_0(\Omega)$, will also be used.
Let $\mathcal{H}=(H_0^1(\Omega))^N$ be the Hilbert space with $H_1$ inner product
$$
(\Phi,\Psi)=\sum_{i=1}^{N}\int_{\Omega}\phi_i\psi_i\quad{\rm for}~\Phi=(\phi_1,\cdots,\phi_N),\Psi=(\psi_1,\cdots,\psi_N)\in\mathcal{H}.
$$
Let $\mathbb{Q}$ be a subspace with orthonormality constraints:
\begin{eqnarray*}\label{eq-class}
\mathbb{Q}=\{\Phi\in\mathcal{H}:\Phi^T\Phi=I^{N\times N}\},
\end{eqnarray*}
where $\displaystyle\Phi^T\Psi=\left( \int_{\Omega}\phi_i\psi_j \right)_{ij}\in\mathbb{R}^{N\times N}$.
For $\Phi\in\mathcal{H}$ and a subdomain $\omega\subset\Omega$, we
shall denote by $\displaystyle\rho_{\Phi}=\sum_{i=1}^N|\phi_i|^2$ and (sometimes abuse the notation for simplicity) by
\begin{eqnarray*}\label{def-norm-abuse}
\|\Phi\|_{s,\omega}=\left(\sum_{i=1}^{N}\|\phi_i\|_{s,\omega}^2\right)^{1/2},s=0,1;\quad
\|\Phi\|_{0,p,\omega}=\left(\sum_{i=1}^{N}\|\phi_i\|_{0,p,\omega}^p\right)^{1/p}, 1\le p\le 6.
\end{eqnarray*}
In our discussions, we shall use the following sets:
$$\mathcal{S}^{N\times N}=\{M\in\mathbb{R}^{N\times N}: M^T=M\},~~
\mathcal{A}^{N\times N}=\{M\in\mathbb{R}^{N\times N}: M^T=-M\}.
$$
For any $\Phi\in \mathbb{Q}$, we may decompose $\mathcal{H}$ into a direct sum of three subspaces (see, e.g., \cite{edelman98}):
\begin{eqnarray*}
\mathcal{H}=\mathcal{S}_\Phi\oplus\mathcal{A}_\Phi\oplus\mathcal{T}_{\Phi},
\end{eqnarray*}
 where $\mathcal{S}_{\Phi}=\Phi\mathcal{S}^{N\times N}$,
$\mathcal{A}_{\Phi}=\Phi\mathcal{A}^{N\times N}$, and $
\mathcal{T}_{\Phi}=\left\{\Psi\in\mathcal{H}:\Psi^T\Phi=0\in\mathbb{R}^{N\times N}\right\}. $

For convenience, the symbol $\lesssim$ will be used throughout this paper, and $A\lesssim B$ means that
$A\leq C B$ for some constant $C$ that is independent of mesh parameters.
We use $\mathscr{P}(p,(c_1,c_2))$ to denote a class of functions satisfying some growth conditions:
\begin{align*}\label{pol-notation}
\mathscr{P}(p,(c_1,c_2))=~\big\{ f:~ \exists ~a_1,a_2\in \mathbb{R} \mbox{ such that }
~c_1 t^{p}+a_1 \leq f(t) \leq c_2 t^{p}+a_2 \quad \forall t\ge 0 \big\}
\end{align*}
with $c_1\in \mathbb{R}$ and $c_2, p\in [0,\infty)$.

\subsection{Problem setting}\label{problem-setting}
Consider the following general form of Kohn-Sham energy functional
\begin{eqnarray}\label{eq-energy}
E(\Phi) &=& \int_{\Omega} \left(\frac{1}{2}\sum_{i=1}^{N}|\nabla\phi_i|^2 + V_{\rm loc} \rho_{\Phi} +
\sum_{i=1}^N\phi_iV_{\rm nl}\phi_i + e_{\rm xc}(\rho_{\Phi})\right) \nonumber\\
 && +\frac{1}{2} D(\rho_{\Phi},\rho_{\Phi})
\end{eqnarray}
for $\Phi=(\phi_1,\phi_2,\cdots,\phi_N)\in\mathcal{H}$, which includes the cases of  Coulomb  potentials and pseudopotential approximations. For the Coulomb potential setting, $V_{\rm loc} = -\sum_{k=1}^M\frac{Z_k}{|x-{\bf R}_k|}$ and $V_{\rm nl} = 0$. While for the pseudopotential approximations, $V_{\rm loc}$ is the local part of pseudopotential and  $V_{\rm nl}$ is a
nonlocal pseudopotential operator (see, e.g., \cite{martin04}) given
by
$$
V_{\rm nl}\phi=\sum_{j=1}^n(\phi,\zeta_j)\zeta_j
$$
with $\zeta_j\in L^2(\Omega) (j=1,2,\cdots,n)$, $n\in\mathbb{N}$. $D(\rho_{\Phi},\rho_{\Phi})$ is the electron-electron Coulomb energy defined by
$$
D(f,g)=\int_{\Omega}f(g*r^{-1}) = \int_{\Omega}\int_{\Omega}f(x)g(y)\frac{1}{|x-y|}dxdy,
$$
and $e_{\rm xc}(t)$ is some real function over $[0,\infty)$.
%
%We may assume that $V_{\rm loc}\in L^2(\Omega)$, which is a very weak condition.
In our  analysis, we require $V_{\rm loc}$  belongs to $L^2(\Omega)$.
We  point out that $V_{\rm loc}\in L^2(\Omega)$ is a very mild condition, which is satisfied by both the Coulomb potential
$V_{\rm ext}(x)=-\sum_{k=1}^M\frac{Z_k}{|x-{\bf R}_k|}$ and the local part of pseudopotential.
Since $e_{\rm xc}:[0,\infty)\rightarrow\mathbb{R}$ does not have a simple analytical expression,
we shall use some approximations and assume throughout this paper that
\begin{eqnarray}\label{assumption-a0}
e_{\rm xc}(t)\in \mathscr{P}(3,(c_1,c_2))~{\rm with}~c_1\geq 0
\quad{\rm or}\quad e_{\rm xc}(t)\in \mathscr{P}(4/3,(c_1,c_2)),
\end{eqnarray}
which is satisfied by almost all the LDAs.

The ground state of the system is obtained by solving the minimization problem
\begin{eqnarray}\label{problem-min}
\inf\left\{ E(\Phi):\Phi\in \mathbb{Q} \right\},
\end{eqnarray}
and we refer to \cite{anantharaman09,cancesM10,chen-gong-he-yang-zhou10} for the discussion of existence of a minimizer.
Note that the energy functional \eqref{eq-energy} is invariant with respect to any unitary transform, i.e.
\begin{eqnarray}\label{energy-invariance}
E(\Phi)=E(\Phi U)=E\big((\sum_{j=1}^Nu_{ij}\phi_j)_{i=1}^N\big)\quad\forall~U=(u_{ij})_{i,j=1}^N\in\mathcal{O}^{N\times N},
\end{eqnarray}
where $\mathcal{O}^{N\times N}$ is the set of orthogonal matrices.
It follows from \eqref{energy-invariance} that if $\Phi$ is a minimizer of \eqref{problem-min},
then $\Phi U$ is also a minimizer for any orthogonal matrix $U$.
%The redundancy of the ground state solutions due to unitary transforms shall be got rid of in our analysis.
%
For any $\Psi\in \mathcal{H}$, we define the equivalence class
\begin{eqnarray*}
[\Psi]=\{\Psi U,~\forall~U\in \mathcal{O}^{N\times N}\}.
\end{eqnarray*}

%{\bf If not accounting the invariance of unitary transform, we assume in the convergence rate and complexity
%analysis that \eqref{problem-min} has $m$ minimizers,
%which are denoted as $[\Phi^{(l)}](l=1,\cdots,m)$ (daixy: This sentence is not proper to put here)}.

We see that any minimizer $\Phi=(\phi_1,\cdots,\phi_N)$ of \eqref{problem-min} satisfies the following weak form
(i.e. the Euler-Lagrange equation associated with the minimization problem):
\begin{eqnarray}\label{problem-eigen-compact-L}
\left\{ \begin{array}{rcl} (H_{\Phi}\phi_{i},v) &=&
\displaystyle \big( \sum_{j=1}^N \lambda_{ij}\phi_{j},v \big)
\quad\forall~ v\in H_0^1(\Omega), \quad i=1,2,\cdots,N,\\[1ex]
\displaystyle \int_{\Omega}\phi_{i}\phi_{j} &=& \delta_{ij},
\end{array} \right.
\end{eqnarray}
where $H_{\Phi}$ is the Kohn-Sham Hamiltonian operator as
\begin{eqnarray}\label{eq-operator-A}
H_{\Phi} = -\frac{1}{2}\Delta + V_{\rm loc} + V_{\rm nl}
+ \int_{\Omega}\frac{\rho_{\Phi}(y)}{|\cdot-y|}dy
+ e_{\rm xc}'(\rho_{\Phi})
\end{eqnarray}
and
\begin{eqnarray}\label{eq-Lambda}
\Lambda=(\lambda_{ij})_{i,j=1}^N =\left( \int_{\Omega}\phi_jH_{\Phi}\phi_i \right)_{i,j=1}^N
\end{eqnarray}
is the Lagrange multiplier.
Since the uniqueness of the ground state solution is unknown even up to a unitary transform,
we define the set of ground states by
\begin{eqnarray}\label{eq-theta}
\Theta=\left\{(\Lambda,\Phi)\in\mathbb{R}^{N\times N}\times\mathbb{Q}:
E(\Phi)=\min_{\Psi\in\mathbb{Q}}E(\Psi)~\mbox{and}~ (\Lambda,\Phi)\mbox{ solves
\eqref{problem-eigen-compact-L}}\right\}.
\end{eqnarray}
Note that the electron density $\rho_{\Phi}$ and the operator $H_{\Phi}$ are also invariant under any unitary transform,
we may diagonalize the matrix of Lagrange multipliers $\Lambda$. More precisely,
there exists a $U\in \mathcal{O}^{N\times N}$, such that the Lagrange multiplier is diagonal
for $\Psi =\Phi U=(\psi_1,\cdots,\psi_N)$, i.e.,
$$
\int_{\Omega}\psi_j H_{\Psi}\psi_i =\mu_i \delta_{ij}.
$$
Consequently, instead of \eqref{problem-eigen-compact-L}, we may consider a form with diagonal multiplier as follows:
\begin{eqnarray}\label{problem-eigen-compact}
\left\{ \begin{array}{rcl} (H_{\Psi}\psi_i,v) &=&
(\mu_i \psi_i,v) \quad\forall~ v\in H_0^1(\Omega),
\quad i=1,2,\cdots,N,\\[1ex] \displaystyle
\int_{\Omega}\psi_i \psi_j &=& \delta_{ij}, \end{array} \right.
\end{eqnarray}
which is the standard Kohn-Sham equation.

Note that any solution of \eqref{problem-eigen-compact-L} can be obtained
from a unitary transform of some solution of \eqref{problem-eigen-compact}.
 %That is, if $\Psi$ is solution of \eqref{problem-eigen-compact}, then any $\Phi \in [\Psi]$ is a solution of \eqref{problem-eigen-compact-L}.
That  is, once we get all solution of \eqref{problem-eigen-compact}, we then obtain all solution of \eqref{problem-eigen-compact-L}.
Consequently, we also call \eqref{problem-eigen-compact-L}  Kohn-Sham equation.

It is well known that the ground state has one electron in each of the $N$ orbitals with the lowest $N$ eigenvalues \cite{martin04}.  Therefore, the ground state solutions in \eqref{eq-theta}
 can be obtained by solving the lowest $N$ eigenpairs of \eqref{problem-eigen-compact}.

For convenience, define $\mathcal{F}:\mathbb{R}^{N\times N}\times \mathcal{H}\rightarrow \mathcal{H}^*$ by
\begin{eqnarray*}
\langle \mathcal{F}(\Lambda,\Phi),\Gamma\rangle =\sum_{i=1}^N \big(
H_{\Phi}\phi_i- \sum_{j=1}^N\lambda_{ij}\phi_{j},\gamma_i\big)\quad
\forall~\Gamma=(\gamma_i)_{i=1}^N\in \mathcal{H}.
\end{eqnarray*}
The Fr\'{e}chet derivative of $\mathcal{F}$ with respect to $\Phi$ at $(\Lambda,\Phi)$ is denoted by
$\mathcal{F}_{\Phi}'(\Lambda,\Phi):\mathcal{H}\rightarrow
\mathcal{H}^{\ast}$ as follows
\begin{eqnarray*}\label{derivative-F}
&&\langle \mathcal{F}_{\Phi}'(\Lambda,\Phi)\Psi,\Gamma \rangle
=\frac{1}{4}E''(\Phi)(\Psi,
\Gamma)-\sum_{i,j=1}^N(\lambda_{ij}\psi_j,\gamma_i)\nonumber\\
&=& \sum_{i=1}^N \big(
H_{\Phi}\psi_i- \sum_{j=1}^N\lambda_{ij}\psi_{j},\gamma_i
\big)+4\sum_{i,j=1}^N\big(e_{\rm xc}''(\rho_{\Phi})
\phi_i\psi_i,\phi_j\gamma_j\big) + \sum_{i,j=1}^N 4D(\phi_i\psi_i,\phi_j
\gamma_j).
\end{eqnarray*}

To study the convergence and complexity, we need the following assumptions \cite{chen-gong-he-yang-zhou10}
\vskip 0.1cm
\begin{itemize}
\item[{\bf A1}]
$|e_{\rm xc}'(t)|+|te_{\rm xc}''(t)|\in
\mathscr{P}(p_1,(c_1,c_2))$ for some $p_1\in [0,2]$.

\item[{\bf A2}]
There exists a constant $\alpha\in(0,1]$ such that
$|e_{\rm xc}''(t)| +|te_{\rm xc}'''(t)|\lesssim 1+t^{\alpha-1} \quad\forall~t>0$.

\item[{\bf A3}]
$(\Lambda,\Phi)$ is a solution of \eqref{problem-eigen-compact-L} and there exists a constant $\beta>0$
depending on $(\Lambda,\Phi)$ such that
\begin{eqnarray}\label{assumption-a3}
\inf_{\Gamma\in\mathcal{T}_{\Phi}}\sup_{\Psi\in\mathcal{T}_{\Phi}}\frac{\langle
\mathcal{F}_{\Phi}'(\Lambda,\Phi)\Psi,\Gamma \rangle}{\|\Psi\|_{1,\Omega}\|\Gamma\|_{1,\Omega}}\geq \beta.
\end{eqnarray}
\end{itemize}

\begin{remark}
%Assumption {\bf A1} ensures some lower semicontinuity property of the energy functional.
%Assumption {\bf A2} is used to make sure the continuity of $\mathcal{F}_{\Phi}'(\Lambda,\Phi)$.
We see that Assumption {\bf A2} implies Assumption {\bf A1} and
the commonly used $X_{\alpha}$ and LDA exchange-correlation energy functionals satisfy Assumption {\bf A2}.

Assumption {\bf A3} is equivalent to  that $\mathcal{F}_{\Phi}'(\Lambda,\Phi)$ is
an isomorphism from $\mathcal{T}_{\Phi}$ to $\mathcal{T}_{\Phi}$.
We observe that if Assumption {\bf A3} is satisfied for $\Phi\in\mathbb{Q}$,
then Assumption {\bf A3} is satisfied for any $\tilde{\Phi}\in[\Phi]$ with the same constant $\beta$, too.
% is used to get rid of the redundancy of ground state solutions up to unitary transforms,
% which can be viewed as imposing the inf-sup condition on the tangent space $\mathcal{T}_{\Phi}$.
We see that a stronger condition than \eqref{assumption-a3} that
\begin{eqnarray*}\label{coervicty}
\langle\mathcal{F}'_{\Phi}(\Lambda,\Phi)\Gamma,\Gamma\rangle
\geq\gamma\|\Gamma\|^2_{1,\Omega} \quad \forall~\Gamma\in\mathcal{T}_{\Phi}
\end{eqnarray*}
is used in \cite{cancesM10,schneider09}, which is satisfied for a linear self-adjoint operator
when there is a gap between the lowest $N$th eigenvalue and $(N+1)$th eigenvalue \cite{schneider09}.
% of the linear operator $\mathcal{F}'_{\Phi}(\Lambda,\Phi)$.
\end{remark}

\subsection{Adaptive finite element approximations} \label{subsec-fem}
Let $d_{_\Omega}$ be the diameter of $\Omega$ and $\{\mathcal{T}_h\}$ be a shape regular family of nested
conforming meshes over $\Omega$ with size $h\in (0,d_{_\Omega})$:
there exists a constant $\gamma^{\ast}$ such that
\begin{eqnarray}\label{shape-regularity}
\frac{h_{\tau}}{\rho_{\tau}} \leq \gamma^{\ast} \quad\forall~\tau \in \mathcal{T}_h,
\end{eqnarray}
where $h_{\tau}$ is the diameter of $\tau$ for each $\tau\in \mathcal{T}_h$, $\rho_{\tau}$ is the diameter of the biggest ball contained in $\tau$,
and $h=\max\{h_{\tau}: \tau\in \mathcal{T}_h\}$. Let $\mathcal{E}_h$ denote the set of interior faces (edges or sides) of $\mathcal{T}_h$.

Let $S^{h,k}(\Omega)$ be a subspace of continuous functions on $\Omega$ such that
\begin{eqnarray*}
S^{h,k}(\Omega)=\{v\in C(\bar{\Omega}): ~v|_{\tau}\in P^k_{\tau}
\quad\forall~\tau \in \mathcal{T}_h\},
\end{eqnarray*}
where $P^k_{\tau}$ is the space of polynomials of degree no greater than $k$ over $\tau$. Let $S^{h,k}_0(\Omega)=S^{h,k}(\Omega)\cap H^1_0(\Omega)$.
We shall denote $S^{h,k}_0(\Omega)$ by $S^{h}_0(\Omega)$ for simplification of notation afterwards and let $V_h=(S^h_0(\Omega))^N$.

We consider the following finite element approximations of \eqref{problem-min}:
\begin{eqnarray}\label{problem-min-dis}
\inf\{ E(\Phi_h):\Phi_h\in V_h\cap\mathbb{Q}\}.
\end{eqnarray}
We see from \cite{anantharaman09,chen-gong-he-yang-zhou10}  that the minimizer of \eqref{problem-min-dis} exists under condition (\ref{assumption-a0})
%Assumption  {\bf A2}.
%
Note that any minimizer $\Phi_h=(\phi_{1,h},\cdots,\phi_{N,h})$ of \eqref{problem-min-dis} solves the Euler-Lagrange  equation
\begin{eqnarray}\label{problem-eigen-compact-dis}
\left\{ \begin{array}{rcl} (H_{\Phi_h}\phi_{i,h},v) &=& \displaystyle \big( \sum_{j=1}^N \lambda_{ij,h}\phi_{j,h},v \big)
\quad\forall ~v\in S^h_0(\Omega), \quad i=1,2,\cdots,N,\\[1ex]
\displaystyle \int_{\Omega}\phi_{i,h}\phi_{j,h} &=& \delta_{ij}
\end{array} \right.
\end{eqnarray}
with the Lagrange multiplier
\begin{eqnarray*}\label{eq-Lambda-dis}
\Lambda_h=(\lambda_{ij,h})_{i,j=1}^N =\left( \int_{\Omega}\phi_{j,h}H_{\Phi_h}\phi_{i,h} \right)_{i,j=1}^N.
\end{eqnarray*}
Define the set of finite dimensional ground state solutions:
\begin{eqnarray*}\label{set-ground-dis}
\Theta_h=\left\{(\Lambda_h,\Phi_h)\in\mathbb{R}^{N\times N}\times (\mathbb{Q}\cap V_h):
E(\Phi_h)=\min_{\Psi\in\mathbb{Q}\cap V_h}E(\Psi) \mbox{ and
}(\Lambda_h,\Phi_h)\mbox{ solves }\eqref{problem-eigen-compact-dis} \right\}.
\end{eqnarray*}
We have from \cite{chen-gong-he-yang-zhou10} that the finite dimensional
approximations are uniformly bounded, i.e., there exists a constant $C$ such that
\begin{eqnarray}\label{eq-bounded}
\sup_{(\Lambda_h,\Phi_h)\in \Theta_h, h\in (0,d_{_\Omega})}(\|\Phi_h\|_{1,\Omega}+|\Lambda_h|)<C.
% + \sup_{\Phi_n\in \Theta_n, n\geq 1}|E(\Phi_n)|
\end{eqnarray}

Using a unitary transform, we can diagonalize $\Lambda_h$ and obtain a discrete Kohn-Sham equation
\begin{eqnarray}\label{problem-eigen-dis}
\left\{ \begin{array}{rcl} (H_{\Psi_h}\psi_{i,h},v) &=&
(\mu_{i,h}\psi_{i,h},v) \quad\forall~ v\in S^h_0(\Omega),
\quad i=1,2,\cdots,N,\\[1ex] \displaystyle
\int_{\Omega}\psi_{i,h}\psi_{j,h} &=& \delta_{ij} \end{array}\right.
\end{eqnarray}
with $\mu_{i,h}=(H_{\Psi_h}\psi_{i,h},\psi_{i,h})$.

%Similar to the continuous case, we   have that if $\Psi_h$ is a solution of \eqref{problem-eigen-dis}, then any $\Phi_h \in [\Psi_h]$ is solution of \eqref{problem-eigen-compact-dis}.
Similar to the continuous case, we have that any solution of \eqref{problem-eigen-compact-dis} can be obtained from a unitary transform of some solution of  \eqref{problem-eigen-dis}.
That is,
 \begin{eqnarray*}
  \Theta_h=\left\{(\Lambda_h,\Phi_h)\in\mathbb{R}^{N\times N}\times (\mathbb{Q}\cap V_h): \Phi_h \in [\Psi_h] ~\mbox{and}~\Lambda_h = \Phi_h^T H_{\Phi_h} \Phi_h,  \forall \Psi_h ~\mbox{with}~ (\bbmu_h, \Psi_h) \in \Xi_h\right\},
 \end{eqnarray*}
 where
 \begin{eqnarray*}
 \Xi_h= \left\{(\bbmu_h, \Psi_h)\in \mathbb{R}^{N\times N} \times (\mathbb{Q}\cap V_h): E(\Psi_h)=\min_{\Psi\in\mathbb{Q}\cap V_h}E(\Psi) \mbox{ and
}(\bbmu_h,\Psi_h)\mbox{ solves } (\ref{problem-eigen-dis}) \right\}.
 \end{eqnarray*}
Since (\ref{problem-eigen-dis}) is solvable, to get $\Theta_h$, we always resort to solving \eqref{problem-eigen-dis} in practice.

An adaptive mesh-refining algorithm usually consists of the following loop \cite{cascon-kreuzer-nochetto-siebert08, dai-xu-zhou08}:
$$
\mbox{\bf Solve}~\rightarrow~\mbox{\bf Estimate}~\rightarrow~\mbox{\bf Mark}~\rightarrow~\mbox{\bf Refine}.
$$

{\bf Solve.}
This step computes the piecewise polynomial finite element approximation with respect to a given mesh.
To simplify the analysis and do as the most work on numerical study of convergence of AFE approximations, we shall assume throughout this paper that we have the exact solutions of
discretized problems\footnote{
Similar conclusion can be expected for  the case where the errors of numerical
integrations and nonlinear algebraic solvers are included (see Section \ref{sec-conclude}). And we
understand that the assumption is indeed a very important practical issue.
}.

{\bf Estimate.} Given a partition $\mathcal{T}_h$ and the corresponding output $(\Lambda_h,\Phi_h)$ from the
``Solve'' step, ``Estimate'' computes the a posteriori error estimator
$\{\eta_h(\Phi_h,\tau)\}_{\tau\in\mathcal{T}_h}$, which is defined as follows.
Define the element residual $\mathcal{R}_\tau(\Phi_h)$ and the jump $J_e(\Phi_h)$ by
\begin{gather*}
\mathcal{R}_{\tau}(\Phi_h) = \big( H_{\Phi_h}\phi_{i,h}-\sum_{j=1}^N\lambda_{ij,h}\phi_{j,h} \big)_{i=1}^N
\quad \mbox{in}~\tau\in \mathcal{T}_h,\\[1ex]
J_e(\Phi_h) =  \Big(  j_e(\phi_{i,h})\Big)_{i=1}^N, ~~j_e(\phi_{i,h}) = \frac{1}{2}\nabla \phi_{i,h} |_{\tau_1}\cdot\overrightarrow{n_1} + \frac{1}{2}\nabla \phi_{i,h}
|_{\tau_2}\cdot\overrightarrow{n_2},
% = \left(\frac{1}{2} [[ \nabla \phi_{i,h}]]_e \cdot\overrightarrow{n_1} \right)_{i=1}^N \quad \mbox{on}~ e\in \mathcal{E}_h,
\end{gather*}
where  $e$ is the common face of elements $\tau_1$ and $\tau_2$ with
unit outward normals $\overrightarrow{n_1}$ and $\overrightarrow{n_2}$, respectively.
Let $\omega_h(e)$ be the union of elements that share the face $e$,
and $\omega_h(\tau)$ be the union of elements that share an edge with $\tau$.
For $\tau\in \mathcal{T}_h$, we define  local error indicator $\eta_h(\Phi_h, \tau)$ and the oscillation
${\rm osc}_h(\Phi_h,\tau)$ by
\begin{eqnarray*}\label{Gerror-indicator-1}
\eta^2_h(\Phi_h, \tau) = h_\tau^2\|\mathcal{R}_{\tau}(\Phi_h)\|_{0,\tau}^2 + \sum_{e\in
\mathcal{E}_h,e\subset\partial \tau} h_e \|J_e(\Phi_h)\|_{0,e}^2,
\end{eqnarray*}
\vskip -0.2cm
\begin{eqnarray*}\label{Glocal-oscillation}
{\rm osc}_h(\Phi_h,\tau) = h_\tau\|\mathcal{R}_{\tau}(\Phi_h) -\overline{\mathcal{R}_{\tau}(\Phi_h)}\|_{0,\tau},
\end{eqnarray*}
where $\overline{w}$ is the $L^2$-projection of $w\in L^2(\Omega)$
to polynomials of some degree  on $\tau$ or $e$.
Given a subset $\omega \subset \Omega$, we define the error estimator $\eta_h(\Phi_h, \omega)$ and
the oscillation ${\rm osc}_h(\Phi_h,\omega)$  by
\begin{eqnarray*}\label{Gerror-estimator-1}
\eta^2_h(\Phi_h, \omega) = \sum_{\tau\in \mathcal{T}_h, \tau\subset \omega} \eta^2_h(\Phi_h, \tau)\quad
\mbox{and}\quad
{\rm osc}^2_h(\Phi_h, \omega) = \sum_{\tau\in \mathcal{T}_h, \tau \subset \omega} {\rm osc}^2_h(\Phi_h, \tau).
\end{eqnarray*}

{\bf Mark.}  We shall replace the subscript $h$ (or $h_k$) by an iteration
counter $k$  whenever convenient afterwards.
Based on the a posteriori error indicators
$\{\eta_k(\Phi_k,\tau)\}_{\tau\in\mathcal{T}_k}$, ``Mark'' gives a strategy to
choose a subset of elements $\mathcal{M}_k$ of $\mathcal{T}_k$ for refinement.
One of the most widely used marking strategy to enforce error reduction is the so-called D\"{o}rfler strategy.

\vskip0.2cm
{D\"{o}rfler Strategy.}\quad Given a parameter $0<\theta <1$ :
\begin{enumerate}
\item Construct a subset $\mathcal{M}_k$ of
$\mathcal{T}_k$ by selecting some elements in $\mathcal{T}_k$ such
that
\begin{eqnarray}\label{mark-strategy}
\sum_{\tau\in \mathcal{M}_k}\eta^2_k({\Phi}_k, \tau)  \geq \theta
\sum_{\tau\in \mathcal{T}_k} \eta^2_k({\Phi}_k,\tau).
\end{eqnarray}
\item Mark all the elements in $\mathcal{M}_k$.
\end{enumerate}
A weaker strategy, which is called ``Maximum Strategy",
only requires that the set of marked elements $\mathcal{M}_k$ contains at least one element of $\mathcal{T}_k$ holding
the largest value estimator \cite{garau-morin11,garau-morin-zuppa09}. Namely, there exists at least one element
$\tau^{\max}_k\in \mathcal{M}_k$ such that
\begin{eqnarray}\label{mark-pri}
\eta_k(\Phi_k,\tau^{\max}_k)=\max_{\tau\in \mathcal{T}_k}\eta_k(\Phi_k,\tau).
\end{eqnarray}
It is easy to check that the most commonly used marking strategies, e.g., D\"{o}rfler's strategy
and Equidistribution strategy, fulfill this condition.

{\bf Refine.} Given the  partition $\mathcal{T}_k$ and the set of marked elements $\mathcal{M}_k$,
``Refine'' produces a new partition $\mathcal{T}_{k+1}$ by refining all elements
in $\mathcal{M}_k$ at least one time. We restrict ourself to a shape-regular bisection for the refinement.
Define
\begin{eqnarray*}
\mathcal{R}_{\mathcal{T}_k\rightarrow
\mathcal{T}_{k+1}}=\mathcal{T}_k\backslash(\mathcal{T}_k\cap
\mathcal{T}_{k+1})
\end{eqnarray*}
as the set of refined elements, we have $\mathcal{M}_k\subset
\mathcal{R}_{\mathcal{T}_k\rightarrow \mathcal{T}_{k+1}}$. Note that usually
more than the marked elements in $\mathcal{M}_k$ are refined in
order to keep the mesh conforming.

\section{Convergence of adaptive finite element approximations}\label{sec-adaptive}
\setcounter{equation}{0}

%Following \cite{chen-he-zhou09,garau-morin-zuppa09},
In this section, we propose and investigate
 an AFE algorithm with Maximum Strategy for Kohn-Sham equations as follows:
\vskip 0.1cm
\begin{algorithm}\label{algorithm-AFEM} {\bf AFE algorithm with Maximum Strategy}
\begin{enumerate}
\item Pick an initial mesh $\mathcal{T}_0$, and let $k=0$.
\item Solve \eqref{problem-eigen-dis} on $\mathcal{T}_k$ to get discrete
solutions $(\mu_{i,k},\psi_{i,k})(i=1,\cdots,N)$ and then $\Theta_k$.
\item Compute local error indictors
$\eta_k(\Psi_k,\tau)$ for all $\tau\in \mathcal{T}_k$.
\item Construct $\mathcal{M}_k \subset \mathcal{T}_k$ by Maximum Strategy.
\item Refine $\mathcal{T}_k$ to get a new conforming mesh $\mathcal{T}_{k+1}$.
\item Let $k=k+1$ and go to 2.
\end{enumerate}
\end{algorithm}
\vskip 0.1cm

We shall prove that all the limit points of the AFE approximations generated by
Algorithm \ref{algorithm-AFEM} are ground state solutions of \eqref{problem-eigen-compact-L},
for which we shall use the similar arguments in  \cite{chen-he-zhou09, garau-morin-zuppa09, zhou04, zhou07}.
Given an initial mesh $\mathcal{T}_0$, Algorithm \ref{algorithm-AFEM} generates a sequence of meshes
$\mathcal{T}_1,\mathcal{T}_2,\cdots$, and associated discrete subspaces
\begin{eqnarray*}
S_0^{h_0}(\Omega)\subsetneq S_0^{h_1}(\Omega)\subsetneq \cdots \subsetneq S_0^{h_n}(\Omega)\subsetneq S_0^{h_{n+1}}(\Omega)
\subsetneq \cdots \subsetneq S_{\infty}(\Omega) \subseteq H_0^1(\Omega),
\end{eqnarray*}
where $\displaystyle S_{\infty}(\Omega)=\overline{\cup_{k=1}^{\infty} S_0^{h_k}(\Omega)}^{H_0^1(\Omega)}$.
Similar to the definition for $V_h$, we set $V_{\infty}=(S_{\infty}(\Omega))^N$.
We have that $V_{\infty}$ is a Hilbert space with the inner product inherited from $\mathcal{H}$ and
\begin{eqnarray}\label{space-appro}
\lim_{k\rightarrow\infty}\inf_{\Psi_k\in V_{h_k}} \|\Psi_k-\Psi_{\infty}\|_{1,\Omega}=0 \quad\forall~\Psi_{\infty}\in
V_{\infty}.
\end{eqnarray}
%We set $V_{\infty}=(S_{\infty}(\Omega))^N \cap \mathbb{Q}$.
Using a direct calculation (see \cite{chen-gong-he-yang-zhou10}), we derive that
$$
\inf_{\tilde{\Psi}_k\in V_{h_k} \cap\mathbb{Q}} \|\tilde{\Psi}_k-\Psi_{\infty}\|_{1,\Omega}
\lesssim \inf_{\Psi_k\in V_{h_k}} \|\Psi_k-\Psi_{\infty}\|_{1,\Omega}
\quad\forall~\Psi_{\infty}\in V_{\infty}\cap\mathbb{Q}
$$
for any $k\in\mathbb{N}$, and hence
\begin{eqnarray}\label{proof-add-3}
\lim_{k\rightarrow\infty} \inf_{\tilde{\Psi}_k\in V_{h_k} \cap\mathbb{Q}} \|\tilde{\Psi}_k-\Psi_{\infty}\|_{1,\Omega}=0
\quad\forall~\Psi_{\infty}\in V_{\infty}\cap\mathbb{Q}.
\end{eqnarray}

From \cite{anantharaman09, chen-gong-he-yang-zhou10}, we know that if Assumption {\bf A2} is satisfied, then the minimizer of energy functional  \eqref{eq-energy}  in $V_{\infty} \cap \mathbb{Q}$ exists.

%Similar to \eqref{problem-min}, we know the existence of a minimizer of energy functional \eqref{eq-energy} in $V_{\infty}\cap \mathbb{Q}$.
We see that any minimizer $\Phi_{\infty}=(\phi_{1,\infty},\cdots,\phi_{N,\infty})\in V_{\infty}\cap \mathbb{Q}$ solves
 the following Euler-Lagrange equation
\begin{eqnarray}\label{problem-eigen-infty}
\left\{ \begin{array}{rcl} (H_{\Phi_{\infty}}\phi_{i,\infty},v) &=& \displaystyle
\big( \sum_{j=1}^N \lambda_{ij,\infty}\phi_{j,\infty},v \big)
\quad\forall ~v\in S_{\infty}(\Omega), \quad i=1,2,\cdots,N,\\[1ex]
\displaystyle \int_{\Omega}\phi_{i,\infty}\phi_{j,\infty} &=& \delta_{ij}\end{array} \right.
\end{eqnarray}
with the Lagrange multiplier
\begin{eqnarray}\label{eq-Lambda-infty}
\Lambda_{\infty}=(\lambda_{ij,\infty})_{i,j=1}^N =\left( \int_{\Omega}\phi_{j,\infty}
H_{\Phi_{\infty}}\phi_{i,\infty} \right)_{i,j=1}^N.
\end{eqnarray}
Define
\begin{eqnarray*}
\Theta_{\infty}=\{(\Lambda_{\infty},\Phi_{\infty})\in\mathbb{R}^{N\times N}\times (V_{\infty} \cap \mathbb{Q}):
E(\Phi_{\infty})=\min_{\Psi\in V_{\infty} \cap \mathbb{Q}} E(\Psi) \\
 \mbox{
and } (\Lambda_{\infty},\Phi_{\infty})\mbox{ solves } \eqref{problem-eigen-infty}\}.
\end{eqnarray*}
Using similar arguments to those in the proof of Theorem 4.1 in \cite{chen-he-zhou09},
we can prove that the AFE approximations for the Kohn-Shan equation converge to some limiting pair in $\Theta_{\infty}$.
\begin{lemma}\label{theorem-con-infty}
Let $\{\Theta_k\}_{k\in\mathbb{N}}$ be the sequence obtained by Algorithm \ref{algorithm-AFEM}.
We have
\begin{gather*}\label{con-u-infty}
\lim_{k\to\infty}E_k=\min_{\Psi\in V_{\infty}\cap \mathbb{Q}}E(\Psi),\\
\lim_{k\to\infty}d_{\mathcal{H}}(\Theta_k,\Theta_{\infty})=0,
\end{gather*}
where $E_k=E(\Phi)((\Lambda,\Phi)\in\Theta_k)$ and the distance between sets
$X,Y\subset\mathbb{R}^{N\times N}\times\mathcal{H}$ is defined by
\begin{eqnarray*}\label{dis-H}
d_{\mathcal{H}}(X,Y)=\sup_{(\Lambda,\Phi)\in X}\inf_{(\bbmu,\Psi)\in Y} (|\Lambda-\bbmu|+\|\Phi-\Psi\|_{1,\Omega}).
\end{eqnarray*}

\end{lemma}
\begin{proof}
Let $(\Lambda_{k},\Phi_{k}) \in \Theta_k$ for $k=1,2,\cdots$, and
$\{(\Lambda_{k_m},\Phi_{k_m})\}_{m\in \mathbb{N}}$ be any subsequence of
$\{(\Lambda_k,\Phi_k)\}_{k\in \mathbb{N}}$ with $1\leq k_1<k_2<\cdots <k_m<\cdots$.

First, following \cite{zhou04, zhou07} (see also \cite{chen-he-zhou09}), we have from (\ref{eq-bounded}) and the Eberlein-Smulian Theorem  that there exists a
weakly convergent subsequence $\{\Phi_{k_{m_j}}\}_{j\in \mathbb{N}}$ and $\Phi_{\infty}\in V_{\infty}$ satisfying
\begin{eqnarray}\label{con-weak-proof}
\Phi_{k_{m_j}}\rightharpoonup \Phi_{\infty} \quad \mbox{in}~\mathcal{H},
\end{eqnarray}
thus it is sufficient to prove
\begin{eqnarray}\label{temp1}
E(\Phi_{\infty})=\min_{\Psi\in V_{\infty} \cap \mathbb{Q}}E(\Psi),
\end{eqnarray}
\begin{eqnarray}\label{temp2}
\lim_{j\to\infty}\big(\|\Phi_{k_{m_j}}-\Phi_{\infty}\|_{1,\Omega}+|\Lambda_{k_{m_j}}-\Lambda_{\infty}|\big)=0.
\end{eqnarray}
Since $H_0^1(\Omega)$ is compactly imbedded in $L^p(\Omega)$ for $p\in[2,6)$,  we have that $\Phi_{k_{m_j}}\to \Phi_{\infty}$ strongly in $(L^p(\Omega))^N$ as $j\rightarrow \infty$.
Hence, we obtain % from \eqref{assumption-a0}
that
\begin{gather*}
\lim_{j\to\infty}\int_{\Omega} V_{\rm loc}(x)\rho_{\Phi_{k_{m_j}}}  = \int_{\Omega} V_{\rm loc}(x)\rho_{\Phi_{\infty}} , \\[1ex]
\lim_{j\to\infty}\int_{\Omega} \sum_{i=1}^N \phi_{i, k_{m_j}}V_{\rm nl}\phi_{i, k_{m_j}} = \int_{\Omega} \sum_{i=1}^N\phi_{i, \infty} V_{\rm nl}\phi_{i, \infty},\\[1ex]
\lim_{j\to\infty}\int_{\Omega}e_{\rm xc}(\rho_{\Phi_{k_{m_j}}})=\int_{\Omega}e_{\rm xc}(\rho_{\Phi_{\infty}}), \\[1ex]
\lim_{j\to\infty}D(\rho_{\Phi_{k_{m_j}}}, \rho_{\Phi_{k_{m_j}}})=D(\rho_{\Phi_{\infty}},\rho_{\Phi_{\infty}}),
\end{gather*}
where \eqref{assumption-a0}  is used for the third equality.
Besides, from (\ref{con-weak-proof}) we have
\begin{eqnarray*}
\liminf_{j\to\infty}   \| \nabla\Phi_{k_{m_j}}\|_{0, \Omega}  \geq \| \nabla\Phi_{\infty}\|_{0, \Omega}.
\end{eqnarray*}
Thus,
\begin{eqnarray}\label{semi-lower}
\liminf_{j\to\infty}   E(\Phi_{k_{m_j}})\geq E(\Phi_{\infty}).
\end{eqnarray}

Let $\Psi_{\infty}$ be a minimizer of the energy functional in $V_{\infty}\cap\mathbb{Q}$.
\eqref{proof-add-3} implies that there exists a sequence $\{\Psi_j\}_{j\in\mathbb{N}}$ such that
$\Psi_j\in V_{k_{m_j}}\cap\mathbb{Q}$ and $\Psi_j\rightarrow \Psi_{\infty}$ in $\mathcal{H}$. Therefore,
\begin{eqnarray}\label{proof-add-2}
E(\Psi_{\infty})=\lim_{j\rightarrow\infty}E(\Psi_j).
\end{eqnarray}
Note that $\{\Phi_{k_{m_j}}\}$ converge to $\Phi_{\infty}$
strongly in $(L^2(\Omega))^N$ leads to $\Phi_{\infty}\in V_{\infty}\cap\mathbb{Q}$,
we have
\begin{eqnarray}\label{proof-add-1}
E(\Phi_{\infty})\geq E(\Psi_{\infty}).
\end{eqnarray}
Since $\Phi_{k_{m_j}}$ is a minimizer of the energy functional in $V_{k_{m_j}}\cap\mathbb{Q}$, we obtain
$$
E(\Psi_j)\geq E(\Phi_{k_{m_j}}),
$$
which together with \eqref{semi-lower}, \eqref{proof-add-2} and \eqref{proof-add-1} leads to
\begin{eqnarray*}\label{engery-subconv}
\liminf_{j\to\infty}E(\Phi_{k_{m_j}})\geq E(\Phi_{\infty})\geq
E(\Psi_{\infty})=\lim_{j\rightarrow\infty}E(\Psi_j) \geq \liminf_{j\to\infty}E(\Phi_{k_{m_j}}).
\end{eqnarray*}
This implies
$$
\lim_{j\to\infty}E(\Phi_{k_{m_j}})=E(\Phi_{\infty})=\min_{\Psi\in V_{\infty}\cap\mathbb{Q}}E(\Psi)
$$
and  thus $(\Lambda_{\infty},\Phi_{\infty})\in\Theta_{\infty}$.

%Note that \eqref{space-appro} implies that $\{\Phi_{k_{m_j}}\}$ is a minimizing sequence for
%the energy functional in $V_{\infty} \cap \mathbb{Q}$, which
%together with (\ref{semi-lower}) and the fact that $\{\Phi_{k_{m_j}}\}$ converge to $\Phi_{\infty}$
%strongly in $(L^2(\Omega))^N$ leads to $(\Lambda_{\infty},\Phi_{\infty})\in\Theta_{\infty}$,
%i.e.,
%\begin{eqnarray*}\label{engery-subconv}
%\lim_{j\to\infty}E(\Phi_{k_{m_j}})=E(\Phi_{\infty})=\min_{\Psi\in V_{\infty} \cap \mathbb{Q}}E(\Psi).
%\end{eqnarray*}
Therefore, we get   that each term of $E(\Phi)$ converges and in particular
\begin{eqnarray}\label{conv-grad-norm}
\lim_{j\to\infty}\|\nabla \Phi_{k_{m_j}}\|_{0,\Omega} = \|\nabla \Phi_{\infty}\|_{0,\Omega}.
\end{eqnarray}
Since $(H_0^1(\Omega))^N$ is a Hilbert space under
norm $\|\nabla\cdot\|_{0,\Omega}$, we conclude from  (\ref{con-weak-proof}) and (\ref{conv-grad-norm}) that
\begin{eqnarray*}
\lim_{j\to\infty}\|\nabla(\Phi_{k_{m_j}}-\Phi_{\infty})\|_{0,\Omega}=0,
\end{eqnarray*}
which together with  (\ref{eq-Lambda}), (\ref{eq-Lambda-infty}) and (\ref{temp1}) implies (\ref{temp2}).
This completes the proof.
\end{proof}

To show that the limit in $V_{\infty} \cap \mathbb{Q} $ is indeed a ground state solution,
we turn to the convergence of the a posteriori error estimators.
% and the weak convergence of residual $\mathbf{R}_k(\Phi_{k})$,
Following the ideas in \cite{chen-he-zhou09, garau-morin11, garau-morin-zuppa09, morin-siebrt-veeser08}, we split the partition $\mathcal{T}_k$ into two sets $\mathcal{T}^{+}_k$ and $\mathcal{T}^{0}_k$, where
\begin{eqnarray*}
\mathcal{T}^{+}_k=\{\tau\in \mathcal{T}_k:\tau\in \mathcal{T}_l,~ \forall~ l\geq k\}\quad \textnormal{and}
\quad \mathcal{T}^{0}_k= \mathcal{T}_k\setminus\mathcal{T}^{+}_k.
\end{eqnarray*}
Actually, $\mathcal{T}^+_k$ is the set of elements that are not refined any more,
and $\mathcal{T}^0_k$ consists of those elements that will eventually be refined.
We denote by
\begin{eqnarray*}
\Omega^+_k=\cup_{\tau\in \mathcal{T}^+_k}\omega_k(\tau)\quad \textnormal{and} \quad
\Omega^0_k=\cup_{\tau\in \mathcal{T}^0_k}\omega_k(\tau).
\end{eqnarray*}
Since the  mesh size function $h_k\equiv h_k(x)$ associated with $\mathcal{T}_k$ is
monotonically decreasing and bounded from below by 0, we have that
\begin{eqnarray*}
h_{\infty}(x)=\lim_{k\rightarrow \infty}h_k(x)
\end{eqnarray*}
is well-defined for almost all $x\in \Omega$ and hence defines a function in $L^{\infty}(\Omega)$.
Moreover, the convergence is uniform (see \cite{morin-siebrt-veeser08}),
more precisely, if $\{h_k\}_{k\in\mathbb{N}}$ is the sequence of mesh size functions generated by Algorithm \ref{algorithm-AFEM},
then
\begin{eqnarray}\label{con-h}
\lim_{k\rightarrow \infty}\|h_k-h_{\infty}\|_{0,\infty,\Omega}=0
\end{eqnarray}
and
\begin{eqnarray}\label{con-h-0}
\lim_{k\rightarrow \infty}\|h_k\chi_{\Omega^0_k}\|_{0,\infty,\Omega}=0,
\end{eqnarray}
where  $\chi_{\Omega^0_k}$ is the characteristic function of $\Omega^0_k$.

\begin{lemma}\label{estimator-stability}
Let $(\Lambda_h, \Phi_h)\in \Theta_h$. If Assumption {\bf A1} is satisfied,
then there exists a constant $C_{\eta}>0$ depending only on the mesh regularity,
such that $\eta_h(\Phi_h,\Omega)\leq C_{\eta}$ and
\begin{eqnarray*}
\eta_h(\Phi_h,\tau)\lesssim \|\Phi_h\|_{0,6,\omega_h(\tau)}+\|\Phi_h\|_{1,\omega_h(\tau)}
\quad\forall ~\tau\in \mathcal{T}_h.
\end{eqnarray*}
\end{lemma}

\begin{proof}
Using \eqref{eq-bounded}, the inverse inequality, the H\"{o}lder inequality, the trace inequality and Assumption {\bf A1}, we have
\begin{eqnarray*}
h_\tau\|\mathcal{R}_\tau(\Phi_h)\|_{0,\tau} &=&h_\tau \Big(\sum_{i=1}^N
\|-\sum_{j=1}^N\lambda_{ij,h}\phi_{j,h}- \frac{1}{2}\Delta \phi_{i,h} + V_{\rm loc}\phi_{i,h} + V_{\rm nl}\phi_{i,h} \\
&& + e_{\rm xc}'(\rho_{\Phi_h})\phi_{i,h} + (r^{-1}*\rho_{\Phi_h})\phi_{i,h}\|^2_{0,\tau} \Big)^{1/2} \nonumber\\
&\lesssim & \sum_{i=1}^N h_\tau \Big( \|\phi_{i,h}\|_{0,\tau}+\|\Delta \phi_{i,h}\|_{0,\tau} + \|V_{\rm loc}\phi_{i,h}\|_{0,\tau} +
\sum_{j=1}^n\|\zeta_j\|^2_{0,\tau}\|\phi_{i,h}\|_{0,\tau}\nonumber\\
&&+\|e_{\rm xc}'(\rho_{\Phi_h})\phi_{i,h}\|_{0,\tau}+\|(r^{-1}*\rho_{\Phi_h}) \phi_{i,h}\|_{0,\tau} \Big)\nonumber\\ &\lesssim
&\|\Phi_h\|_{0,6,\omega_h(\tau)}+\|\Phi_h\|_{1,\omega_h(\tau)}
\end{eqnarray*}
and
\begin{eqnarray*}
h_e^{1/2}\|J_e(\Phi_h)\|_{0,e} &=&  h_e^{1/2}\left(\sum_{i=1}^N \|\frac{1}{2}\nabla \phi_{i,h} \big|_{\tau_1}\cdot\overrightarrow{n_1} +
\frac{1}{2}\nabla
\phi_{i,h} \big|_{\tau_2}\cdot\overrightarrow{n_2}\|_{0,e}^2\right)^{1/2} \nonumber\\
&\lesssim&  h_e^{1/2} \left(\sum_{i=1}^N\big(\|\nabla
\phi_{i,h}|_{\tau_1}\|_{0,e}^2 + \|\nabla \phi_{i,h}|_{\tau_2}\|_{0,e}^2\big)\right)^{1/2}\nonumber\\
&\lesssim&h_e^{1/2}\left(h_e^{-1}\sum_{i=1}^N\|\nabla \phi_{i,h}\|_{0,\omega_h(\tau)}^2\right)^{1/2}\\
&\lesssim & \|\Phi_h\|_{1,\omega_h(\tau)}.
\end{eqnarray*}
Hence we obtain
\begin{eqnarray*}\label{eta-stability}
\eta_h(\Phi_h,\tau)\lesssim \|\Phi_h\|_{0,6,\omega_h(\tau)}+\|\Phi_h\|_{1,\omega_h(\tau)} \quad\forall~ \tau\in \mathcal{T}_h,
\end{eqnarray*}
which together with the Sobolev inequality implies $\eta_h(\Phi_h,\Omega)\leq C_{\eta}$,
where the constant $C_{\eta}>0$ depends only on the data and the mesh regularity. This completes the proof.
\end{proof}

Using similar procedure as in \cite{chen-he-zhou09,garau-morin-zuppa09}, we can prove that the maximal error indicator
$\max_{\tau\in \mathcal{M}_{k}} \eta_{k}(\Phi_{k},\tau)$ tends to zero.
\begin{lemma}\label{marking-limit}
Let $\{\Phi_k\}_{k\in \mathbb{N}}$ be the sequence produced by Algorithm \ref{algorithm-AFEM}.
If Assumption {\bf A1} is satisfied, then
\begin{eqnarray*}
\lim_{k\rightarrow \infty}\max_{\tau\in \mathcal{M}_{k}} \eta_{k}(\Phi_{k},\tau)=0.
\end{eqnarray*}
\end{lemma}

\begin{proof}
%Using similar proof to that in Theorem \ref{theo-convergence},
We see from Lemma \ref{theorem-con-infty} that for any subsequence $\{\Phi_{k_m}\}$ of $\{\Phi_k\}$,
there exist a convergent subsequence
$\{\Phi_{k_{m_j}}\}$ and $\Phi_{\infty}$ satisfying $(\Lambda_{\infty}, \Phi_{\infty})\in\Theta_{\infty}$ such that
\begin{eqnarray}\label{proof-8}
\Phi_{k_{m_j}} \rightarrow \Phi_{\infty}\quad\mbox{in}~~\mathcal{H}.
\end{eqnarray}
Hence it is only necessary for us to prove that
\begin{eqnarray*}
\lim_{j\rightarrow \infty}\max_{\tau\in \mathcal{M}_{k_{m_j}}} \eta_{k_{m_j}}(\Phi_{k_{m_j}},\tau)=0.
\end{eqnarray*}
For simplicity, we denote the subsequence $\{\Phi_{k_{m_j}}\}_{j\in\mathbb{N}}$
by $\{\Phi_k\}_{k\in\mathbb{N}}$, and $\{\mathcal{T}_{k_{m_j}}\}_{j\in\mathbb{N}}$ by $\{\mathcal{T}_k\}_{k\in\mathbb{N}}$.
We obtain from Lemma \ref{estimator-stability} that
\begin{eqnarray}\label{eta-ineq-1}\nonumber
\eta_k(\Phi_k,\tau_k) &\lesssim& \|\Phi_k\|_{0,6,\omega_k(\tau_k)} +\|\Phi_k\|_{1,\omega_k(\tau_k)} \\ &\lesssim&
\|\Phi_k-\Phi_{\infty}\|_{1,\Omega}+\|\Phi_{\infty}\|_{0,6,\omega_k(\tau_k)} +\|\Phi_{k}\|_{1,\omega_k(\tau_k)},
\end{eqnarray}
where $\tau_k\in \mathcal{M}_k$ be such that
\begin{eqnarray*}
\eta_k(\Phi_k,\tau_k)=\max_{\tau\in \mathcal{M}_k}\eta_k(\Phi_k,\tau).
\end{eqnarray*}
Note that \eqref{proof-8} implies that the first term on the right-hand side of (\ref{eta-ineq-1}) goes to zero.
Since $\tau_k\in \mathcal{M}_k\subset \mathcal{T}_k^0$, we have from \eqref{con-h-0} that
\begin{eqnarray*}
|\omega_k(\tau_k)|\lesssim h_{\tau_k}^3\leq\|h_k\chi_{\Omega^0_k}\|^3_{0,\infty,\Omega}\rightarrow 0
\quad\textnormal{as}~~ k\rightarrow\infty,
\end{eqnarray*}
which implies that the other two terms on the right-hand side of (\ref{eta-ineq-1}) go to zero, too.
This completes the proof.
\end{proof}

Define a global residual $\mathbf{R}_h(\Phi_h)\in \mathcal{H}^{*}$ by
\begin{eqnarray}\label{def-R}
\langle\mathbf{R}_h(\Phi_h),\Gamma\rangle = \sum_{i=1}^N \big( H_{\Phi_h}\phi_{i,h}-\sum_{j=1}^N\lambda_{ij,h}\phi_{j,h},
\gamma_i \big) \quad \forall~ \Gamma=(\gamma_i)_{i=1}^N\in \mathcal{H}.
\end{eqnarray}
We see that
\begin{eqnarray}\label{eq-R-estimator}
\langle\mathbf{R}_h(\Phi_h),\Gamma\rangle=\sum_{\tau\in \mathcal{T}_h}\left( \big(\mathcal{R}_\tau(\Phi_h),\Gamma\big)_\tau
+ \sum_{e\in \mathcal{E}_h, e\subset\partial \tau} \big( J_e(\Phi_h),\Gamma \big)_e \right)\quad \forall~ \Gamma\in \mathcal{H}.
\end{eqnarray}
Thus
\begin{eqnarray}\label{proof-31}
|\langle\mathbf{R}_h(\Phi_h),\Gamma\rangle| \lesssim
\sum_{\tau\in\mathcal{T}_h}\eta_h(\Phi_h,\tau)\|\Gamma\|_{1,\omega_h(\tau)}\quad \forall~\Gamma\in\mathcal{H}.
\end{eqnarray}
Thanks to  Lemma \ref{estimator-stability} and Lemma \ref{marking-limit}, by carrying out  the similar procedure as the  proof for Lemma 4.3 of \cite{chen-he-zhou09}, we can obtain a weak convergence of $\mathbf{R}_{k}(\Phi_{k})$ as follows.
\begin{lemma}\label{R-limit}
Let $\{\Phi_k\}_{k\in \mathbb{N}}$ be the sequence produced by Algorithm \ref{algorithm-AFEM}.
If Assumption {\bf A1} is satisfied, then
\begin{eqnarray}\label{proof-45}
\lim_{k\rightarrow \infty}\langle \mathbf{R}_{k}(\Phi_{k}),\Gamma \rangle=0\quad\forall~ \Gamma\in \mathcal{H}.
\end{eqnarray}
\end{lemma}

Now we turn to prove the main result of this section, that is, the limit of
the AFE approximations for the Kohn-Shan equation is a  ground state solution.
\begin{theorem}\label{theorem-convergence-adaptive}
{\em (convergence)}
Let $\{\Theta_k\}_{k\in\mathbb{N}}$ be the sequence generated by Algorithm \ref{algorithm-AFEM}.
If the initial mesh $\mathcal{T}_0$ is sufficiently fine and Assumption {\bf A1} is satisfied, then
\begin{eqnarray}\label{con-energy-adaptive}
\lim_{k\to\infty}E_k=\min_{\Psi\in\mathbb{Q}}E(\Psi),
\end{eqnarray}
\begin{eqnarray}\label{conv-phi-adaptive}
\lim_{k\to\infty} d_{\mathcal{H}}(\Theta_k, \Theta)=0.
\end{eqnarray}
\end{theorem}

\begin{proof}
Let $\{(\Lambda_k,\Phi_k)\}_{k\in\mathbb{N}}$ be the sequence generated by  Algorithm \ref{algorithm-AFEM}.
We know from Lemma \ref{theorem-con-infty} that for any subsequence $\{(\Lambda_{k_m},\Phi_{k_m})\}_{m\in\mathbb{N}}$,
there exists a convergent subsequence $\{(\Lambda_{k_{m_j}},\Phi_{k_{m_j}})\}_{j\in\mathbb{N}}$
and $(\Lambda_{\infty},\Phi_{\infty})\in\Theta_{\infty}$  such that
\begin{gather*}
\Phi_{k_{m_j}} \rightarrow \Phi_{\infty}\quad\mbox{in}~~\mathcal{H},\\[1ex]
\Lambda_{k_{m_j}}\rightarrow\Lambda_{\infty}\quad\mbox{in}~~\mathbb{R}^{N\times N}.
\end{gather*}
Consequently, it is only necessary for us to prove $(\Lambda_{\infty},\Phi_{\infty})\in\Theta$,
which implies \eqref{con-energy-adaptive} and \eqref{conv-phi-adaptive} directly.
For simplicity, we denote by $\{(\Lambda_k,\Phi_k)\}_{k\in\mathbb{N}}$ the convergent subsequence
$\{(\Lambda_{k_{m_j}},\Phi_{k_{m_j}})\}_{j\in\mathbb{N}}$, and by $\{\mathcal{T}_k\}_{k\in\mathbb{N}}$
the corresponding subsequence $\{\mathcal{T}_{k_{m_j}}\}_{j\in\mathbb{N}}$.

We first show that the limiting eigenpair $(\Lambda_{\infty},\Phi_{\infty})$ is also an eigenpair of
\eqref{problem-eigen-compact-L}. We have from \eqref{def-R} that for any $\Gamma\in \mathcal{H}$
\begin{eqnarray}\label{proof-32-neq1}
(H_{\Phi_{\infty}}\Phi_{\infty} - \Lambda_{\infty}\Phi_{\infty},\Gamma) &=& (H_{\Phi_{\infty}}\Phi_{\infty}
- \Lambda_{\infty}\Phi_{\infty},\Gamma) - \langle\mathbf{R}_k(\Phi_k),\Gamma\rangle + \langle\mathbf{R}_k(\Phi_k),\Gamma\rangle \nonumber \\
&=& (H_{\Phi_{\infty}}\Phi_{\infty} - H_{\Phi_k}\Phi_k,\Gamma) - (\Lambda_{\infty}\Phi_{\infty} - \Lambda_k\Phi_k,\Gamma)\nonumber\\
&& + \langle\mathbf{R}_k(\Phi_k),\Gamma\rangle.
\end{eqnarray}
By a direct calculation using Assumption {\bf A1}, we get
\begin{eqnarray*}
(H_{\Phi_{\infty}}\Phi_{\infty} - H_{\Phi_k}\Phi_k,\Gamma)\lesssim \|\Phi_{\infty}-\Phi_k\|_{1,\Omega}\|\Gamma\|_{1,\Omega},
\end{eqnarray*}
which together with \eqref{proof-32-neq1} leads to
\begin{eqnarray}\label{proof-32}
(H_{\Phi_{\infty}}\Phi_{\infty} - \Lambda_{\infty}\Phi_{\infty},\Gamma)\lesssim(\|\Phi_{\infty}-\Phi_k\|_{1,\Omega}
+ |\Lambda_{\infty}-\Lambda_k|) \|\Gamma\|_{1,\Omega} + \langle\mathbf{R}_k(\Phi_k),\Gamma\rangle.~~~~~~
\end{eqnarray}
We get from $\Lambda_k\rightarrow\Lambda_{\infty}$ and $\Phi_k\rightarrow \Phi_{\infty}$ in $\mathcal{H}$ that
the first term on the right-hand side of \eqref{proof-32} goes to zero when $k$ goes to infinity.
We obtain from Lemma \ref{R-limit} that the other term on the right-hand side of \eqref{proof-32} goes to zero, and hence
\begin{eqnarray*}
(H_{\Phi_{\infty}}\Phi_{\infty},\Gamma)=(\Lambda_{\infty}\Phi_{\infty},\Gamma) \quad \forall~ \Gamma\in \mathcal{H}.
\end{eqnarray*}

Then we shall show that for a sufficiently fine initial mesh, the limiting eigenpair $(\Lambda_{\infty},\Phi_{\infty})$
is a ground state solution in $\Theta$.
Similar to \cite{chen-he-zhou09}, we set
\begin{eqnarray*}
\mathcal{W}=\{(\Lambda,\Phi)\in\mathbb{R}^{N\times N}\times\mathcal{H}:(\Lambda,\Phi)~{\rm solves~\eqref{problem-eigen-compact-L}}\}.
\end{eqnarray*}
Note that $\Theta\subsetneq\mathcal{W}$.
% the ground state solutions in $\Theta$ minimize energy functional \eqref{eq-energy}, which is continuous over $\mathcal{H}$
Using the fact
\begin{eqnarray*}
\lim_{h\rightarrow 0}\inf_{\Psi\in V_h} \|\Psi-\Phi\|_{1,\Omega}=0 \quad \forall~\Phi\in \mathcal{H},
\end{eqnarray*}
we can choose an initial mesh $\mathcal{T}_0$ such that
\begin{eqnarray*}
E_0\equiv \min_{\Phi_{h_0}\in V_{h_0}\cap\mathbb{Q}}E(\Phi_{h_0}) < \min_{(M,\Psi)\in \mathcal{W}\setminus\Theta} E(\Psi),
\end{eqnarray*}
Due to $\mathcal {T}_0\subset\mathcal {T}_k$, we have $E_k\le E_0$ and hence $(\Lambda_{\infty},\Phi_{\infty})\in\Theta$.
This completes the proof.
\end{proof}

\section{Quasi-optimality of adaptive finite element methods}\label{sec-optimal}
\setcounter{equation}{0}
In this section we propose and  analyze the following AFE algorithm using  D\"{o}rfler's marking strategy.

 \vskip 0.1cm
\begin{algorithm}\label{algorithm-AFEM-con-rate}
{\bf AFE algorithm with D\"{o}rfler Strategy}
\begin{enumerate}
\item Pick a given mesh $\mathcal{T}_0$, and let $k=0$.
\item Solve \eqref{problem-eigen-dis} on $\mathcal{T}_k$ to get discrete
solutions $(\mu_{i,k},\psi_{i,k})(i=1,\cdots,N)$, and then $\Theta_{k}$.
\item Compute local error indictors
$\eta_k(\Psi_k,\tau)$ for all $\tau\in \mathcal{T}_k$.
\item Construct $\mathcal{M}_k \subset \mathcal{T}_k$ by D\"{o}rfler Strategy and  parameter
 $\theta$.
\item Refine $\mathcal{T}_k$ to get a new conforming mesh $\mathcal{T}_{k+1}$.
\item Let $k=k+1$ and go to 2.
\end{enumerate}
\end{algorithm}
\vskip 0.1cm
%Based on the relevant results for linear boundary value problems (see Appendix),
We shall study  the convergence rate and quasi-optimal complexity of Algorithm \ref{algorithm-AFEM-con-rate}, for which we shall apply the perturbation arguments (c.f., e.g., \cite{chen-he-zhou10, dai-xu-zhou08,he-zhou11}) and  certain relationship
between nonlinear problem \eqref{problem-eigen-compact-L} and its associated linear boundary value problem (see \eqref{model-problem}).

To establish  the relationship, we define
\begin{eqnarray*}
a(\Phi,\Gamma)=\sum_{i=1}^N\frac{1}{2}(\nabla \phi_i,\nabla \gamma_i)
\quad \forall~\Phi=(\phi_i)_{i=1}^N,\Gamma=(\gamma_i)_{i=1}^N\in
\mathcal{H}.
\end{eqnarray*}
One sees that there exists a constant $c_a>0$ such that
\begin{eqnarray}\label{coercive-constant}
a(\Gamma,\Gamma)\geq c_a\|\Gamma\|^2_{1,\Omega} \quad\forall~ \Gamma\in \mathcal{H}.
\end{eqnarray}
%The energy norm $\|\cdot\|_{1,\Omega}$, which is equivalent to $\|\cdot\|_{1,\Omega}$, is defined by
%$\|\Phi\|_{1,\Omega}=\sqrt{a(\Phi,\Phi)}$.
% It is  known that \eqref{model-weak} is uniquely solvable for any $f\in (H^{-1}(\Omega))^N$.

Let $\mathcal{L}:\mathcal{H}\rightarrow \mathcal{H}^*$ be the operator defined by
$$
\langle\mathcal{L}(\Phi),\Gamma\rangle=a(\Phi,\Gamma)\quad\forall~\Gamma\in\mathcal{H},
$$
and $K:\mathcal{H}^*\rightarrow \mathcal{H}$  be the inverse operator of $\mathcal{L}$ such that
\begin{eqnarray*}\label{def-operator-k}
a(K\Phi, \Gamma) = (\Phi, \Gamma) \quad\forall~\Gamma \in\mathcal{H}.
\end{eqnarray*}
Note that \eqref{coercive-constant} implies that $K$ is well defined and there holds
\begin{eqnarray}\label{K-property}
\|K\Phi\|_{1,\Omega}\lesssim \|\Phi\|_{-1,\Omega}\quad
\forall~\Phi\in \mathcal{H}^*.
\end{eqnarray}
Let $P_h:\mathcal{H} \rightarrow V_h$ be the $H^1$-projection defined by
\begin{eqnarray}\label{projection}
a(\Phi-P_h \Phi, \Gamma) =0 \quad\forall~\Phi\in \mathcal{H},~\Gamma\in V_h.
\end{eqnarray}
For any $\Phi\in \mathcal{H}$, there hold
\begin{eqnarray}\label{stable1}
\|P_h \Phi\|_{1,\Omega}\lesssim \|\Phi\|_{1,\Omega} \quad \textnormal{and} \quad
\lim_{h\rightarrow 0}\|\Phi-P_h\Phi\|_{1,\Omega}=0.
\end{eqnarray}

\subsection{Basic estimate}
First we recall an a priori error estimate, whose proof is referred to \cite{chen-gong-he-yang-zhou10}.
Define
\begin{gather*}
X_{\Phi,h}=\mathcal{S}^{N\times N}\times(V_h\cap (\mathcal{S}_{\Phi}\oplus\mathcal{T}_{\Phi})).
\end{gather*}
\begin{theorem}\label{priori-theorem}
Let $(\Lambda,\Phi)$ be a solution of \eqref{problem-eigen-compact-L}.
If Assumptions {\bf A2} and {\bf A3} are satisfied, then there exists
$\delta>0$ such that for sufficiently small $h$, \eqref{problem-eigen-compact-dis}
has a unique local solution $(\Lambda_h,\Phi_h)\in X_{\Phi,h}\cap
B_{\delta}((\Lambda,\Phi))$. Moreover, there hold
\begin{eqnarray}\label{H1-norm}
\|\Phi-\Phi_h\|_{1,\Omega}\lesssim \inf_{\Psi\in V_h}\|\Phi-\Psi\|_{1,\Omega},
\end{eqnarray}
\vskip -0.5cm
\begin{eqnarray}\label{prop-conc}
|\Lambda_h-\Lambda|\lesssim \|\Phi_h-\Phi\|^2_{1,\Omega}+\|\Phi_h-\Phi\|_{0,\Omega},
\end{eqnarray}\vskip -0.5cm
\begin{eqnarray}\label{L2-norm}
\|\Phi-\Phi_h\|_{0,\Omega}\lesssim r(h) \|\Phi-\Phi_h\|_{1,\Omega}
\end{eqnarray}
with $r(h)\to 0$ as $h\rightarrow 0$.
\end{theorem}

Using Theorem \ref{priori-theorem}, we can denote afterwards by $(\Lambda_h,\Phi_h)\in X_{\Phi,h}\cap
B_{\delta}((\Lambda,\Phi))$ the unique local discrete approximation of $(\Lambda,\Phi)\in\Theta$.
%We say an equivalence class $[\Psi_h]$ approximates the equivalence class $[\Phi]$
%if there exists an orthogonal matrix $U_h$ such that $\Phi_h=\Psi_hU_h$.
%

%For $(\Lambda,\Phi)\in\Theta$ and $\Phi_h\in V_h$,
%we say the equivalence class $[\Phi_h]$ approximate the equivalence class $[\Phi]$ if
%$$
%d_{\rm ec}([\Phi_h],[\Phi]) <  d_{\rm ec}([\Phi_h],[\tilde{\Phi}]), ~~~~\forall (\tilde{\Lambda},\tilde{\Phi})\in\Theta ~~\mbox{and}~ [\Phi] \neq [\tilde{\Phi}],
%$$
%where the distance $d_{\rm ec}$ between two equivalence classes is defined by
%$$
%d_{\rm ec}([\Psi],[\Phi]) = \min_{U\in\mathcal{O}^{N\times N}}\|\Psi U-\Phi\|_{1,\Omega}.
%$$

For simplicity, we denote by $V=V_{\rm loc}+V_{\rm nl}$
and $\displaystyle\mathcal{N}(\rho_\Phi)=\int_{\Omega}\frac{\rho_{\Phi}(y)}{|\cdot-y|}dy+e_{xc}'(\rho_\Phi)$.

\begin{lemma}\label{V}
Let  $(\Lambda,\Phi)$ be a solution of \eqref{problem-eigen-compact-L} and $h_0\in (0,1)$ be the  mesh size of the initial mesh $\mathcal{T}_0$.
If Assumptions {\bf A2} and {\bf A3} are satisfied, then
there exists $\hat{\kappa}(h)$ such that $\hat{\kappa}(h)\to 0$ as $h\to 0$ and
\begin{eqnarray}\label{V-N-property}
\|V(\Phi_h-\Phi)\|_{-1,\Omega}+\|\mathcal{N}(\rho_{\Phi_h})\Phi_h-\mathcal{N}(\rho_{\Phi})\Phi\|_{-1,\Omega}
\lesssim \hat{\kappa}(h)\|\Phi-\Phi_h\|_{1,\Omega}.
\end{eqnarray}
\end{lemma}
\begin{proof}
For any $\Psi \in \mathcal{H}$, by using the  H\"{o}lder inequality and the Young's inequality, we have that for any $\varepsilon>0$, there holds
\begin{eqnarray*}\label{temp-1}
\|\Psi\|_{0,3,\Omega}&\leq& \|\Psi\|_{0, \Omega}^{1/3} \|\Psi\|_{0, 4, \Omega}^{2/3}  =  (\varepsilon^{-2/3}\|\Psi\|_{0, \Omega}^{1/3}) (\varepsilon^{2/3} \|\Psi\|_{0, 4, \Omega}^{2/3}) \nonumber\\
&\lesssim & \frac{\varepsilon^{-2}}{3} \|\Psi\|_{0, \Omega} + \frac{2\varepsilon}{3} \|\Psi\|_{1, \Omega},
\end{eqnarray*}
which together with \eqref{L2-norm} implies that there exists a positive constant $C$ independent of $h$ and $\varepsilon$ such that
\begin{eqnarray*}
\|\Phi-\Phi_h\|_{0,3,\Omega}\le C \left( \varepsilon^{-2} r(h)+ \varepsilon \right)\|\Phi-\Phi_h\|_{1,\Omega}
\quad\forall~h\in (0,h_0].
\end{eqnarray*}
% Taking $\varepsilon = r(h)^{1/3}$ and setting  $\hat{\kappa}(h) =  r(h)^{1/3}$, we have that $\hat{\kappa}(h) \to 0$ as $h \to 0$ and
%\begin{eqnarray*}\label{proof-ee}
%\|\Phi-\Phi_h\|_{0,3,\Omega} \lesssim \hat{\kappa}(h)\|\Phi-\Phi_h\|_{1,\Omega},
%\end{eqnarray*}
%which together with the H\"{o}lder inequality leads to
Therefore, by the H\"{o}lder inequality, we get
\begin{eqnarray}\label{proof-e00}
\|V_{\rm loc}(\Phi-\Phi_h)\|_{-1,\Omega}& =& \sup_{\Gamma\in \mathcal{H}}
\frac{\big(V_{\rm loc}(\Phi_h-\Phi),\Gamma\big)}{\|\Gamma\|_{1,\Omega}}
\leq \|V_{\rm loc}\|_{0,\Omega}\|\Phi-\Phi_h\|_{0,3,\Omega} \nonumber\\
& \lesssim &  \left( \varepsilon^{-2} r(h)+ \varepsilon \right) \|\Phi-\Phi_h\|_{1,\Omega}.
\end{eqnarray}

For the nonlocal pseudopotential operator, we derive
\begin{eqnarray*}\label{proof-44-}
\big(V_{\rm nl}(\Phi_h-\Phi),\Gamma\big) \lesssim\|\Phi-\Phi_h\|_{0,\Omega}\|\Gamma\|_{0,\Omega}
\quad\forall~\Gamma\in\mathcal{H}
\end{eqnarray*}
from the fact that
\begin{eqnarray*}
\big(\sum_{j=1}^n(\zeta_j,\phi_{i, h}-\phi_i)\zeta_j,v\big)
\lesssim\|\phi_{i, h}-\phi_i\|_{0,\Omega}\|v\|_{0,\Omega}
\quad\forall~v\in H_0^1(\Omega),\quad i=1,\cdots,N.
\end{eqnarray*}
Therefore, we have
\begin{eqnarray}\label{proof-e1}
\|V_{\rm nl}(\Phi_h-\Phi)\|_{-1,\Omega} =\sup_{\Gamma\in \mathcal{H}}
\frac{\big(V_{\rm nl}(\Phi_h-\Phi),\Gamma\big)}{\|\Gamma\|_{1,\Omega}}
\lesssim
\|\Phi-\Phi_h\|_{0,\Omega}\lesssim r(h)\|\Phi-\Phi_h\|_{1,\Omega}.~~~~~~~~
\end{eqnarray}

For the exchange-correlation part, we have that there exists $\xi=(\xi_1,\cdots,\xi_N)$
with $\xi_i=\delta_i\phi_{i,h}+(1-\delta_i)\phi_i$ and $\delta_i\in [0,1]~(i=1,\cdots,N)$, such that
\begin{eqnarray*}
(e_{\rm xc}'(\rho_{\Phi_h})\Phi_h-e_{\rm xc}'(\rho_{\Phi})\Phi,\Gamma)
%&=&\sum_{i=1}^N\int_{\Omega} (e_{\rm xc}'(\rho_{\Phi_h})\phi_{i,h}-e_{\rm xc}'(\rho_{\Phi})\phi_i)\gamma_i\nonumber\\
=\sum_{i=1}^N\int_{\Omega}(e_{\rm xc}'(\rho_{\xi})+2\xi_i^2e_{\rm xc}''(\rho_{\xi}))(\phi_{i,h}-\phi_i)\gamma_i.
\end{eqnarray*}
This together with Assumption {\bf A2} leads to
\begin{eqnarray}\label{proof-e2}\nonumber
&&(e_{\rm xc}'(\rho_{\Phi_h})\Phi_h-e_{\rm xc}'(\rho_{\Phi})\Phi,\Gamma) ~\lesssim~ \sum_{i=1}^N\int_{\Omega}(\rho_{\xi}+\rho_{\xi}^{\alpha})|\phi_{i,h}-\phi_i|\cdot|\gamma_i| \\\nonumber
&\lesssim& \sum_{i=1}^N\big(\|\rho_{\xi}^{\alpha}\|_{0,3/{\alpha},\Omega}
\|\phi_{i,h}-\phi_{i}\|_{0,\Omega} \|\gamma_i\|_{0,6/(3-2\alpha),\Omega}+\|\rho_{\xi}\|_{0,3,\Omega}
\|\phi_{i,h}-\phi_{i}\|_{0,\Omega} \|\gamma_i\|_{0,6,\Omega}\big) \\
&\lesssim&  \|\Phi_h-\Phi\|_{0,\Omega}\|\Gamma\|_{1,\Omega}\quad\quad \forall~\Gamma\in \mathcal{H},
\end{eqnarray}
where the H\"{o}lder inequality and the fact
\begin{eqnarray*}
\|\rho_{\xi}\|_{0,3,\Omega}\leq \|\xi\|^2_{0,6,\Omega}\leq \|\Phi\|^2_{0,6,\Omega}+\|\Phi_h\|^2_{0,6,\Omega}\leq \bar{C}
\end{eqnarray*}
are used.
For the Coulomb potential, we obtain from the Young's inequality and the Uncertainty Principle \cite{reed-simon75} that
\begin{eqnarray*}
\|r^{-1}*(\rho_{\Phi}-\rho_{\Phi_h})\|_{0,\infty,\Omega} \lesssim
\sum_{i=1}^N\|\nabla(\phi_i+\phi_{i,h})\|_{0,\Omega}\|\phi_i-\phi_{i,h}\|_{0,\Omega}\lesssim \|\Phi-\Phi_h\|_{0,\Omega}.
\end{eqnarray*}
Therefore, we have that for any $v\in H_0^1(\Omega)$ and $1\leq i\leq N$, there holds
\begin{eqnarray*}\label{proof-43}\nonumber
&& \int_{\Omega} \big( (r^{-1}*\rho_{\Phi_h})\phi_{i,h} - (r^{-1}*\rho_{\Phi})\phi_i \big) v\\ \nonumber &=& \int_{\Omega}
(r^{-1}*\rho_{\Phi_h})(\phi_{i,h}-\phi_i)v + \int_{\Omega} r^{-1}*(\rho_{\Phi_h}-\rho_{\Phi})\phi_i v \\ \nonumber &\lesssim&
\|r^{-1}*\rho_{\Phi_h}\|_{0,\infty,\Omega} \|\phi_{i,h}-\phi_i\|_{0,\Omega} \|v\|_{0,\Omega}
+ \|r^{-1}*(\rho_{\Phi_h}-\rho_{\Phi})\|_{0,\infty,\Omega}\|\phi_i\|_{0,\Omega} \|v\|_{0,\Omega}
\\ &\lesssim& \|\phi_i-\phi_{i,h}\|_{0,\Omega}\|v\|_{0,\Omega} + \|\Phi-\Phi_h\|_{0,\Omega}\|v\|_{0,\Omega},
\end{eqnarray*}
which implies
\begin{eqnarray}\label{proof-46}
((r^{-1}*\rho_{\Phi_h})\Phi_h - (r^{-1}*\rho_{\Phi})\Phi,\Gamma) \lesssim
\|\Phi-\Phi_h\|_{0,\Omega}\|\Gamma\|_{0,\Omega}\quad\forall~\Gamma\in \mathcal{H}.
\end{eqnarray}
Consequently, we obtain from \eqref{proof-e2}, \eqref{proof-46} and the definition of $\mathcal{N}$ that
\begin{eqnarray}\label{proof-e5}
\|\mathcal{N}(\rho_{\Phi_h})\Phi_h-\mathcal{N}(\rho_{\Phi})\Phi\|_{-1,\Omega}=\sup_{\Gamma\in
\mathcal{H}}\frac{(\mathcal{N}(\rho_{\Phi_h})\Phi_h-\mathcal{N}
(\rho_{\Phi})\Phi,\Gamma)}{\|\Gamma\|_{1,\Omega}}\lesssim \|\Phi-\Phi_h\|_{0,\Omega}.~~~~~~~~
\end{eqnarray}

Taking,  $\varepsilon = r(h)^{1/3}$ and setting  $\hat{\kappa}(h) =  r(h)^{1/3}$,
 we have that $\hat{\kappa}(h) \to 0$ as $h \to 0$.
Combining \eqref{L2-norm}, \eqref{proof-e00}, \eqref{proof-e1} and \eqref{proof-e5},
we complete the proof of \eqref{V-N-property}.
\end{proof}

We now exploit the relationship between the nonlinear eigenvalue problem and its associated linear boundary
value problem, which will be employed in our analysis.
We rewrite \eqref{problem-eigen-compact-L} and \eqref{problem-eigen-compact-dis}  as
\begin{eqnarray*}\label{operator-K-1}
\Phi=K(\Phi\Lambda -V \Phi-\mathcal{N}(\rho_{\Phi})\Phi),
\end{eqnarray*}
\vskip -0.6cm
\begin{eqnarray}\label{u-w}
\Phi_h=P_hK(\Phi_h\Lambda_h-V
\Phi_h-\mathcal{N}(\rho_{\Phi_h})\Phi_h),
\end{eqnarray}
respectively.
Set $W^h=K(\Phi_h \Lambda_h-V \Phi_h-\mathcal{N}(\rho_{\Phi_h})\Phi_h)$, we have $\Phi_h=P_h W^h$.

\begin{theorem}\label{thm-nonlinear-linear}
Let $(\Lambda,\Phi)$ be a solution of \eqref{problem-eigen-compact-L}. If Assumptions {\bf A2} and {\bf A3} are satisfied, then
there exists $\kappa(h)\in (0,1)$ such that $\kappa(h)\rightarrow 0$ as $h\rightarrow 0$ and
\begin{eqnarray}\label{nonlinear-linear-neq}
\|\Phi-\Phi_h\|_{1,\Omega}= \|W^h - P_h W^h\|_{1,\Omega} +\mathcal {O}(\kappa(h))\|\Phi-\Phi_h\|_{1,\Omega}.
\end{eqnarray}
\end{theorem}

\begin{proof}
By the definition of $W^h$, we have
\begin{eqnarray}\label{total-term}
\Phi-W^h= K(\Phi\Lambda-\Phi_h \Lambda_h) + KV (\Phi_h-\Phi)
+K(\mathcal{N}(\rho_{\Phi_h})\Phi_h-\mathcal{N}(\rho_{\Phi})\Phi).
\end{eqnarray}
For the first term on the right-hand side of \eqref{total-term}, we obtain from \eqref{K-property} and \eqref{L2-norm} that
\begin{eqnarray}\label{proof-ee1}
\|K(\Phi\Lambda-\Phi_h\Lambda_h)\|_{1,\Omega}
&\leq&\|\Phi\Lambda-\Phi_h \Lambda_h\|_{0,\Omega}
\lesssim\|(\Phi-\Phi_h)\Lambda\|_{0,\Omega}+\|\Phi_h(\Lambda-\Lambda_h)\|_{0,\Omega} \nonumber\\
&\lesssim &\|\Phi-\Phi_h\|_{0,\Omega}|\Lambda|+|\Lambda-\Lambda_h|
\lesssim r(h)\|\Phi-\Phi_h\|_{1,\Omega}. \quad
\end{eqnarray}
Using Lemma \ref{V}, we can estimate the second term on the right-hand side of \eqref{total-term} as follows
\begin{eqnarray}\label{proof-e0}
\|KV(\Phi-\Phi_h)\|_{1,\Omega} ~\lesssim~ \|V(\Phi-\Phi_h)\|_{-1,\Omega} ~\lesssim~
\hat{\kappa}(h) \|\Phi-\Phi_h\|_{1,\Omega}.
\end{eqnarray}
Using \eqref{K-property}, \eqref{L2-norm} and \eqref{V-N-property}, we obtain for the last term of \eqref{total-term} that
\begin{eqnarray}\label{proof-ee3}\nonumber
\|K(\mathcal{N}(\rho_{\Phi_h})\Phi_h-\mathcal{N}(\rho_{\Phi})\Phi)\|_{1,\Omega} \lesssim
\|\mathcal{N}(\rho_{\Phi_h})\Phi_h-\mathcal{N}(\rho_{\Phi})\Phi\|_{-1,\Omega}\lesssim r(h)\|\Phi-\Phi_h\|_{1,\Omega}.
\end{eqnarray}
Set $\kappa(h)=r(h)+\hat{\kappa}(h) $, we derive from \eqref{total-term}, \eqref{proof-ee1}, and \eqref{proof-e0}  that
\begin{eqnarray}\label{thm1-neq1}
\|\Phi-W^h\|_{1,\Omega}\leq \hat{C} \kappa(h)\|\Phi-\Phi_h\|_{1,\Omega},
\end{eqnarray}
with $\hat{C}$ being some constant.
Note that (\ref{u-w}) implies
\begin{eqnarray*}
\Phi - \Phi_h = W^h-P_h W^h +\Phi-W^h,
\end{eqnarray*}
which together with \eqref{thm1-neq1} leads to (\ref{nonlinear-linear-neq}). This completes the proof.
\end{proof}

\subsection{A posteriori error estimates}
Define
\begin{eqnarray}\label{def-kappa-}
\tilde{\kappa}(h_0)=\sup_{h\in (0
,h_0]}\kappa(h)
\end{eqnarray}
and note that $\tilde{\kappa}(h_0) \ll 1$ if $h_0\ll 1$.
Based on the relevant results for linear boundary value problems (see Appendix), we have the following
% a posteriori error
 estimates for AFE approximations.

\begin{theorem}
\label{theorem-posteriori}%{\em (a posteriori error estimate)}
Let $(\Lambda,\Phi)$ be a solution of \eqref{problem-eigen-compact-L}, $h_0 \ll 1$ and $h \in (0, h_0]$. If Assumptions {\bf A2} and {\bf A3} are satisfied,
then there exist positive constants $C_1, C_2$ and $C_3$ depending on the coercivity constant
$c_a$ (in \eqref{coercive-constant}) and the shape regularity constant $\gamma^{\ast}$ (in \eqref{shape-regularity}),
such that
\begin{eqnarray}\label{upper-bound}
\|\Phi-\Phi_h\|^2_{1,\Omega} \leq C_1 \eta^2_h(\Phi_h, \Omega),
\end{eqnarray}
\vskip -0.6cm
\begin{eqnarray}\label{lower-bound}
C_2 \eta^2_h(\Phi_h, \Omega) \le \|\Phi-\Phi_h\|_{1,\Omega}^2+ C_3{\rm osc}_h^2(\Phi_h, \Omega).
\end{eqnarray}
\end{theorem}

\begin{proof}
Due to $\mathcal{L} W^h = \Phi_h \Lambda_h-V\Phi_h-\mathcal{N}(\rho_{\Phi_h})\Phi_h$,
we obtain from (\ref{boundary-upper}) and (\ref{boundary-lower}) that
\begin{eqnarray}\label{auxiliary-boundary-problem-upper}
\|W^h - P_h W^h \|^2_{1,\Omega} \leq \tilde{C}_1 \tilde{\eta}^2_h (P_h W^h, \Omega),
\end{eqnarray}
\vskip -0.6cm
\begin{eqnarray}\label{auxiliary-boundary-problem-lower}
\tilde{C}_2 \tilde{\eta}^2_h (P_h W^h, \Omega) \le \|W^h-P_h W^h\|_{1,\Omega}^2
+\tilde{C}_3\widetilde{{\rm osc}}^2_h(P_h W^h, \Omega),
\end{eqnarray}
where the constants $\tilde C_1$, $\tilde C_2$ and $\tilde C_3$ are given in Theorem \ref{upper-bound-theorem},
$\tilde{\eta}^2_h (P_h W^h, \Omega)$ and $\widetilde{{\rm osc}}^2_h(P_h W^h, \Omega)$ are defined by
\eqref{error-indicator} and \eqref{local-oscillation} with $\Gamma$ being replaced by $P_h W^h$.
 It is easy to see that $\tilde{\eta}_h(P_h W^h, \Omega) = \eta_h(\Phi_h, \Omega) $ and $\widetilde{{\rm osc}}_h(P_h W^h, \Omega) =  {\rm osc}_h(\Phi_h, \Omega)$ from their definitions and the fact that $\Phi_h = P_h W^h$.

We have from \eqref{nonlinear-linear-neq} and \eqref{def-kappa-} that
\begin{eqnarray*}\label{temp-ineq-1}
\|\Phi-\Phi_h \|_{1,\Omega}\leq (1+\hat{C} \tilde{\kappa}(h_0))\|W^h-P_hW^h\|_{1,\Omega},
\end{eqnarray*}
which together with \eqref{auxiliary-boundary-problem-upper} leads to \eqref{upper-bound} by taking the constant
\begin{eqnarray}\label{C1}
C_1 = \tilde{C}_1 (1+ \hat{C} \tilde{\kappa}(h_0))^2.
\end{eqnarray}

Similarly, we get (\ref{lower-bound}) from (\ref{u-w}), (\ref{nonlinear-linear-neq})
and (\ref{auxiliary-boundary-problem-lower}). In particular, we may choose $C_2$ and $C_3$ by
\begin{eqnarray}\label{coef-eigen-bound}
C_2=\tilde{C}_2 (1 - \hat{C} \tilde{\kappa}(h_0))^2, ~~C_3=\tilde{C}_3 (1 - \hat{C}\tilde{\kappa}(h_0))^2.
\end{eqnarray}
%(\ref{upper-bound-lambda}) and (\ref{lower-bound-lambda}) can be easily obtained  from Theorem \ref{priori-theorem}.
This completes the proof.
\end{proof}

We shall now present the following property  that will be used  in
our analysis.
\begin{lemma}\label{lemma-eta-etah-T}
% Let $\Phi_k\in \mathbb{Q}\cap V_k$.
Let $(\Lambda_h, \Psi_h)$ be solution of (\ref{problem-eigen-compact-dis}). For any $\Psi_h' =  \Psi_h U$ with $U$ being some orthogonal matrix, there hold
\begin{eqnarray}\label{lemma-eta-etah-T-conc1}
\frac{1}{N} \eta_h^2(\Psi_h', \tau) \leq    \eta_h^2(\Psi_h, \tau) \leq N \eta_h^2(\Psi_h', \tau), ~~~\forall \tau \in \mathcal{T}_h,
\end{eqnarray}
and
\begin{eqnarray}\label{lemma-eta-etah-T-conc2}
 \frac{1}{N} osc_h^2(\Psi_h', \tau) \leq  osc_h^2(\Psi_h, \tau) \leq N osc_h^2(\Psi_h', \tau), ~~~\forall \tau \in \mathcal{T}_h.
\end{eqnarray}
%\begin{eqnarray}\label{lemma-eta-etah-T-conc1}
% \eta_h^2(\Psi_h', \tau) \lesssim  \eta_h^2(\Psi_h, \tau) \lesssim \eta_h^2(\Psi_h', \tau), ~~~\forall \tau \in \mathcal{T}_h,
%\end{eqnarray}
%and
%\begin{eqnarray}\label{lemma-eta-etah-T-conc2}
% osc_h^2(\Psi_h', \tau) \lesssim  osc_h^2(\Psi_h, \tau) \lesssim osc_h^2(\Psi_h', \tau), ~~~\forall \tau \in \mathcal{T}_h.
%\end{eqnarray}
\end{lemma}

\begin{proof}
We write $U = (\alpha_{i,j})_{i,j=1}^N$. Since $U$ is orthogonal, we have $\sum_{l=1}^{N} \alpha_{i, l} 　\alpha_{j,l}   = \sum_{l=1}^{N} \alpha_{l, i} 　\alpha_{l, j} 　= \delta_{ij}$ 　for $i,j=1, \cdots, N$.

On the one hand, we obtain from $\Psi_h' =\Psi_h  U $ that
\begin{eqnarray*}
\psi_{i, h}' = \sum_{j=1}^N  \alpha_{j,i}   \psi_{j,h}, ~~i = 1, \cdots, N.
\end{eqnarray*}　
Denote the Lagrange multiplier corresponding to $\Psi_h'$  by $\Lambda_h'$. Since $\Psi_h' =\Psi_h  U $ implies  $H_{\Psi_h'} = H_{\Psi_h}$,  we get
\begin{eqnarray*}
\Lambda_h' = (\Psi_h')^T H_{\Psi_h'} \Psi_h' = (\Psi_h  U)^T H_{\Psi_h} \Psi_h  U = U^T \Psi_h^T H_{\Psi_h} \Psi_h U  = U^T \Lambda_h U.
\end{eqnarray*}
Therefore,
\begin{eqnarray*}
 \Psi_h' \Lambda_h' = \Psi_h U U^T \Lambda_h U = \Psi_h \Lambda_h U,
\end{eqnarray*}
that is,
\begin{eqnarray*}
\sum_{j=1}^{N}\lambda_{ij,h}' \psi_{j, h}' = \sum_{l,j =1}^N　\alpha_{l,i}  \lambda_{lj,h} \psi_{j,h}, ~ i = 1, \cdots, N.
\end{eqnarray*}
Consequently,  for any $\tau \in \mathcal{T}_h$
\begin{eqnarray*}
\eta^2_h(\Psi_h', \tau) &=& h_\tau^2\|\mathcal{R}_{\tau}(\Psi_h')\|_{0,\tau}^2 + \sum_{e\in
\mathcal{E}_h,e\subset\partial \tau} h_e \|J_e(\Psi_h')\|_{0,e}^2 \nonumber\\
&=&　\sum_{i=1}^N  \Big(h_\tau^2 \|H_{\Psi_h'}\psi_{i,h}'-\sum_{j=1}^N\lambda_{ij,h}'\psi_{j,h}' \|_{0,\tau}^2 + \sum_{e\in
\mathcal{E}_h,e\subset\partial \tau} h_e　\|j_e(\psi_{i,h}') \|_{0,e}^2\Big)\nonumber\\
&=& \sum_{i=1}^N　\Big(h_\tau^2 \| H_{\Psi_h} \sum_{l=1}^{N} \alpha_{l, i}\psi_{l,h} - \sum_{l,j =1}^N　\alpha_{l,i}  \lambda_{lj,h} \psi_{j,h}\|_{0,\tau}^2　\nonumber\\
&& +　\sum_{e\in
\mathcal{E}_h,e\subset\partial \tau} h_e　\|j_e(\sum_{l=1}^{N} \alpha_{l,i}\psi_{l,h}) \|_{0,e}^2\Big).
\end{eqnarray*}
Thus, by triangle inequality and H\"{o}lder inequality, we may estimate as follows
\begin{eqnarray*}
\eta^2_h(\Psi_h', \tau)　
&\leq& \sum_{i=1}^N　\Big(h_\tau^2 \big(\sum_{l=1}^{N} \alpha_{l,i}　\| H_{\Psi_h}　\psi_{l,h} - \sum_{j=1}^N　 \lambda_{lj, h} \psi_{j,h}\|_{0,\tau}\big)^2　\nonumber\\
&& +　\sum_{e\in
\mathcal{E}_h,e\subset\partial \tau} h_e　\big(\sum_{l=1}^{N} \alpha_{l,i}\|j_e(\psi_{l,h}) \|_{0,e}\big)^2\Big)\nonumber\\
 &\leq&  \sum_{i=1}^N　\Big( \big(\sum_{l=1}^{N} \alpha_{l,i}^2\big) \big(　h_\tau^2\sum_{l=1}^{N} \| H_{\Psi_h}　\psi_{l,h} - \sum_{j=1}^N　  \lambda_{lj, h} \psi_{j,h}\|_{0,\tau}^2　
 \nonumber\\
&& +　\sum_{l=1}^{N}  \sum_{e\in
\mathcal{E}_h,e\subset\partial \tau} h_e \|j_e(\psi_{l,h}) \|_{0,e}^2\big)\Big)\nonumber\\
 &=&  \sum_{i=1}^N　\Big(  \sum_{l=1}^{N} h_\tau^2 \| H_{\Psi_h}　\psi_{l,h} - \sum_{j=1}^N　  \lambda_{lj, h} \psi_{j,h}\|_{0,\tau}^2　+　\sum_{l=1}^{N} \sum_{e\in
\mathcal{E}_h,e\subset\partial \tau} h_e \|j_e(\psi_{l,h}) \|_{0,e}^2 \Big)\nonumber\\
 &=& N \eta^2_h(\Psi_h, \tau), ~\forall \tau \in \mathcal{T}_h,
\end{eqnarray*}
where the fact $\sum_{l=1}^{N} \alpha_{l,i}^2 = 1$ is used.
That is,
\begin{eqnarray}\label{lemma-eta-etah-T-ineq1}
\eta_h^2(\Psi_h',\tau)\leq  N   \eta_h^2(\Psi_h, \tau), ~\forall \tau \in \mathcal{T}_h.
\end{eqnarray}

On the other  hand, $\Psi_h' =\Psi_h  U $ implies $\Psi_h =\Psi_h'  U^T $. Hence,
\begin{eqnarray*}
\psi_{i,h} = \sum_{j=1}^N \alpha_{i,j}  \psi_{j,h}', \quad i=1, \cdots, N.
\end{eqnarray*}
By the similar process we obtain that
\begin{eqnarray}\label{lemma-eta-etah-T-ineq2}
\eta^2_h(\Psi_h, \tau)　
 &\leq& N \eta^2_h(\Psi_h', \tau), ~~\forall \tau \in \mathcal{T}_h.
\end{eqnarray}
%Therefore,
%\begin{eqnarray}\label{lemma-eta-etah-T-ineq2}
%\sum_{\tau \in \mathcal{M}_h} \eta_h^2(\Psi_h, \tau)
%\leq N \sum_{\tau \in \mathcal{M}_h} \eta_h^2(\Psi_h', \tau).
%\end{eqnarray}

Similarly, there have
\begin{eqnarray}\label{lemma-eta-etah-T-ineq3}
osc_h^2(\Psi_h', \tau)\leq  N   osc_h^2(\Psi_h, \tau),  ~~\forall \tau \in \mathcal{T}_h
\end{eqnarray}
and
\begin{eqnarray}\label{lemma-eta-etah-T-ineq4}
osc^2_h(\Psi_h, \tau)　
 \leq N osc^2_h(\Psi_h', \tau), ~~\forall \tau \in \mathcal{T}_h.
\end{eqnarray}

 We  obtain (\ref{lemma-eta-etah-T-conc1}) from (\ref{lemma-eta-etah-T-ineq1}) and  (\ref{lemma-eta-etah-T-ineq2}), and get  (\ref{lemma-eta-etah-T-conc2})  from  (\ref{lemma-eta-etah-T-ineq3}) and  (\ref{lemma-eta-etah-T-ineq4}). This completes the proof.
\end{proof}

Thanks to Lemma \ref{lemma-eta-etah-T}, we can get the bounds of $\| \Phi - \Phi_h\|_{1, \Omega}$ by computable terms  $\eta^2_h(\Psi_h, \Omega)$ and ${\rm osc}_h^2(\Psi_h, \Omega)$, other than the uncomputable term   $\eta^2_h(\Phi_h, \Omega)$ and ${\rm osc}_h^2(\Phi_h, \Omega)$ as in  Theorem \ref{theorem-posteriori},
 and then get the a posteriori error estimate for distance between the ground states and its approximation as follows.
%\begin{theorem}
%\label{theorem-posteriori-2}{\em (a posteriori error estimate)}
%Suppose $h_0 \ll 1$ and $h \in (0, h_0]$. Let $(\Lambda_h, \Psi_h)$ be solution of (\ref{problem-eigen-dis}), if Assumptions {\bf A2} and {\bf A3} are satisfied,
%then there hold
%\begin{eqnarray}\label{upper-bound-dist}
%dist^2_{\mathcal{H}}(\Psi_h, \mathcal{U})  \lesssim  \eta^2_h(\Psi_h, \Omega),
%\end{eqnarray}
%\vskip -0.6cm
%\begin{eqnarray}\label{lower-bound-dist}
%  \eta^2_h(\Psi_h, \Omega) \lesssim dist^2_{\mathcal{H}}(\Psi_h, \mathcal{U})  +  {\rm osc}_h^2(\Psi_h, \Omega).
%\end{eqnarray}
%\end{theorem}
%
%
%
%From Lemma \ref{lemma-eta-etah-T} and Theorem \ref{theorem-posteriori}, we can get the a posteriori error estimate
%for distance between the ground states and its approximation.
\begin{theorem}
\label{theorem-posteriori-3}{\em (a posteriori error estimate)}
Suppose $h_0 \ll 1$ and $h \in (0, h_0]$. Let $(\bbmu_h, \Psi_h)$ be solution of (\ref{problem-eigen-dis}), if Assumptions {\bf A2} and {\bf A3} are satisfied,
then there hold
\begin{eqnarray}\label{upper-bound-dist2}
d^2_{\mathcal{H}}(\overline{\Theta}_h, \Theta)  \lesssim  \eta^2_h(\Psi_h, \Omega),
\end{eqnarray}
\vskip -0.6cm
\begin{eqnarray}\label{lower-bound-dist2}
\eta^2_h(\Psi_h, \Omega) \lesssim d^2_{\mathcal{H}}(\overline{\Theta}_h, \Theta)  +  {\rm osc}_h^2(\Psi_h, \Omega),
\end{eqnarray}
here $\overline{\Theta}_h = \left\{(\Lambda_h,\Phi_h)\in\mathbb{R}^{N\times N}\times (\mathbb{Q}\cap V_h):
\Phi_h \in [\Psi_h], ~\mbox{and}~ \Lambda_h = \Phi_h^T H_{\Phi_h} \Phi_h  \right\} \subseteq \Theta_h$.
\end{theorem}

Our analysis is based on the following crucial technical result, which can be obtain directly from Lemma \ref{lemma-eta-etah-T}.
\begin{lemma}\label{eta-etah}
% Let $\Phi_k\in \mathbb{Q}\cap V_k$.
Let $(\Lambda_h,\Phi_h)$ be any solution of  (\ref{problem-eigen-compact-dis}). If  there exists  constant $\theta \in (0, 1)$ satisfying
\begin{eqnarray}\label{lemma-eta-etah-cond}
\sum_{\tau\in \mathcal{M}_h} \eta_h^2(\Phi_h, \tau) \geq \theta \eta_h^2(\Phi_h, \Omega),
\end{eqnarray}
 then for any $\Phi_h' =  \Phi_h U$ with $U$ being some orthogonal matrix,
there exists a constant $\theta' \in (0, 1)$, such that
\begin{eqnarray}\label{lemma-eta-etah-conc}
\sum_{\tau\in \mathcal{M}_h} \eta_h^2(\Phi_h', \tau) \geq \theta' \eta_h^2(\Phi_h', \Omega).
\end{eqnarray}
In further, we have $\theta' = \frac{\theta}{N^2}$.
\end{lemma}

\subsection{Convergence rate}
Now we turn to analyze the convergence rate of Algorithm   \ref{algorithm-AFEM-con-rate}.
Similar to \cite{chen-he-zhou10,  dai-xu-zhou08}, we shall first establish some relationships between two level finite element approximations. We use $\mathcal{T}_H$ to denote a coarse mesh and
$\mathcal{T}_h$ to denote a refined mesh of $\mathcal{T}_H$.

\begin{lemma}\label{lemma-bound-general}
Let $h, H \in (0, h_0]$ and $(\Lambda,\Phi)$ be a solution of \eqref{problem-eigen-compact-L}.  If Assumptions {\bf A2} and {\bf A3} are satisfied, then
\begin{eqnarray}\label{lemma-bound-general-conc-1}
\|\Phi  - \Phi_h  \|_{1,\Omega}&=& \|W^{H}-P_h W^{H} \|_{1,\Omega}
+ \mathcal{O} (\tilde{\kappa}(h_0)) \left( \|\Phi - \Phi_H \|_{1,\Omega}+ \|\Phi
- \Phi_h\|_{1,\Omega}\right),~~~~~
\end{eqnarray}
\vskip -0.6cm
\begin{eqnarray}\label{lemma-bound-general-conc-2}
{\rm osc}_h(\Phi_h,\Omega)&=& \widetilde{{\rm osc}}_h(P_h W^{H}, \Omega) +
\mathcal{O} (\tilde{\kappa}(h_0)) \left( \|\Phi  - \Phi_H \|_{1,\Omega} + \|\Phi-\Phi_h\|_{1,\Omega}\right),
\end{eqnarray}
and
\begin{eqnarray}\label{lemma-bound-general-conc-3}
\eta_{h}(\Phi_h, \Omega) &=&  \tilde{\eta}_{h}(P_h W^{H}, \Omega) + \mathcal{O}
(\tilde{\kappa}(h_0)) \left( \|\Phi  - \Phi_H \|_{1,\Omega} + \|\Phi - \Phi_h\|_{1,\Omega}\right).
\end{eqnarray}
\end{lemma}

\begin{proof}
First, we obtain (\ref{lemma-bound-general-conc-1}) from \eqref{stable1}, (\ref{thm1-neq1}) and the identity
\begin{eqnarray*}\label{lem-bound-general2}
\Phi-\Phi_h = W^{H} - P_h W^{H} + P_h(W^{H} -W^{h})+\Phi-W^{H}.
\end{eqnarray*}

For the estimate of \eqref{lemma-bound-general-conc-2}, we get from $\Phi_h = P_h W^{H}+P_h(W^{h}-W^{H})$ that
\begin{eqnarray}\label{temp6}
\widetilde{{\rm osc}}_h(P_h W^{h},\Omega)\leq\widetilde{{\rm osc}}_h(P_h W^{H},\Omega)
+ \widetilde{{\rm osc}}_h(P_h (W^{h}-W^{H}),\Omega),
\end{eqnarray}
where $\widetilde{\rm osc}$ is given in Appendix.
Using \eqref{u-w} and the fact $\widetilde{{\rm osc}}_h(\Phi_h,\Omega)={\rm osc}_h(\Phi_h,\Omega)$,
we know that it is only necessary to estimate $\widetilde{{\rm osc}}_h(P_h (W^{h}-W^{H}),\Omega)$.

Since $\mathcal{L}W^{h}=\Phi_h \Lambda_h -V \Phi_h - \mathcal{N}(\rho_{\Phi_h})\Phi_h$
and $\mathcal{L}W^{H}=\Phi_H \Lambda_H -V \Phi_H-\mathcal{N}(\rho_{\Phi_H})\Phi_H$,
we obtain
$$
\mathcal{L}(W^{h} - W^{H})=\Phi_h \Lambda_h -\Phi_H \Lambda_H
+V(\Phi_H-\Phi_h) +\mathcal{N}(\rho_{\Phi_H})\Phi_H -\mathcal{N}(\rho_{\Phi_h})\Phi_h.
$$
Let $G=P_h(W^{h}-W^{H})$ and $\tilde{\mathcal{R}}_{\tau}(G)$ be defined by \eqref{residual-modelproblem}
with $\Gamma$ being replaced by $G$. We have
\begin{eqnarray*}
\tilde{\mathcal{R}}_{\tau}(G) = \Phi_h \Lambda_h -\Phi_H \Lambda_H +V(\Phi_H-\Phi_h)
+\mathcal{N}(\rho_{\Phi_H})\Phi_H-\mathcal{N}(\rho_{\Phi_h})\Phi_h -\mathcal{L}G
\end{eqnarray*}
and
\begin{eqnarray}\label{temp4}
&&\widetilde{{\rm osc}}^2_h(P_h(W^{h}-W^{H}),\Omega) = \sum_{\tau\in \mathcal{T}_h}
\widetilde{{\rm osc}}^2_h(G,\tau) = \sum_{\tau\in \mathcal{T}_h}h^2_{\tau}\|\tilde{\mathcal{R}}_{\tau}
(G)-\overline{\tilde{\mathcal{R}} _{\tau}(G)}\|^2_{0,\tau} \nonumber \\
&\leq & \sum_{\tau\in \mathcal{T}_h}h^2_{\tau}\|\tilde{\mathcal{R}}_{\tau}(G)
+\mathcal{L}G-\overline{(\tilde{\mathcal{R}}_{\tau}(G) +\mathcal{L}G)}\|^2_{0,\tau}
+\sum_{\tau\in\mathcal{T}_h}h^2_{\tau}\|\mathcal{L}G -\overline{\mathcal{L}G}\|^2_{0,\tau}. \quad
\end{eqnarray}

Using  the inverse inequality, and the fact that
$\Phi_h\Lambda_h$ and $\Phi_H\Lambda_H$ are piecewise
polynomials vectors over $\mathcal{T}_h$ and $\mathcal{T}_H$ respectively, \eqref{L2-norm}, and \eqref{V-N-property}, we may estimate as follows
\begin{eqnarray}\label{rightterm}
& &\big(\sum_{\tau\in
\mathcal{T}_h}h^2_{\tau}\|\tilde{\mathcal{R}}_{\tau}
(G)+\mathcal{L}G- \overline{(\tilde{\mathcal{R}}_{\tau}(G)
+\mathcal{L}G)}\|^2_{0,\tau}\big)^{1/2}\nonumber\\
&\lesssim& \sum_{\tau\in \mathcal{T}_h}h_{\tau}\big(\|V(\Phi_H-\Phi_h)\|_{0,\tau}+
\|\mathcal{N}(\rho_{\Phi_H})\Phi_H-\mathcal{N}(\rho_{\Phi_h}) \Phi_h\|_{0,\tau}\big)\nonumber\\
&\lesssim&\tilde{\kappa}(h_0)\left( \|\Phi - \Phi_h\|_{1, \Omega} + \|\Phi - \Phi_H\|_{1, \Omega}\right).
\end{eqnarray}
Combining the inverse inequality, \eqref{stable1} and (\ref{thm1-neq1}), we arrive at
\begin{eqnarray}\label{temp_osc}
&& \big(\sum_{\tau\in\mathcal{T}_h}h^2_{\tau}\|\mathcal{L}G -\overline{\mathcal{L}G}\|^2_{0,\tau}\big)^{1/2}
\lesssim \big(\sum_{\tau\in\mathcal{T}_h}h^2_{\tau}\|\mathcal{L}G \|^2_{0,\tau}\big)^{1/2} \lesssim \|G\|_{1, \Omega} \nonumber\\
&=& \|P_h(W^{h}-W^{H})\|_{1,\Omega} \lesssim \tilde{\kappa}(h_0)\left( \|\Phi - \Phi_h\|_{1,\Omega}
+ \|\Phi - \Phi_H\|_{1, \Omega}\right). \quad
\end{eqnarray}
Taking \eqref{temp4}, \eqref{rightterm} and \eqref{temp_osc} into account, we have
\begin{eqnarray}\label{temp5}
\widetilde{{\rm osc}}_h(P_h(W^{h}-W^{H}),\Omega)\lesssim \tilde{\kappa}(h_0)
\left( \|\Phi - \Phi_h\|_{1, \Omega} + \|\Phi - \Phi_H\|_{1,\Omega}\right),
\end{eqnarray}
which together with \eqref{temp6} leads to (\ref{lemma-bound-general-conc-2}).

Finally, we shall prove (\ref{lemma-bound-general-conc-3}).
We obtain from (\ref{boundary-lower}), (\ref{thm1-neq1}) and (\ref{temp5}) that
\begin{eqnarray*}\label{lem-bound-general5}
\tilde{\eta}_h(P_h( W^{h} - W^{H}), \Omega) &\lesssim&
\| (W^{h} - W^{H}) -P_h(W^{h} - W^{H})\|_{1, \Omega} \nonumber\\
&& +
\widetilde{{\rm osc}}_h(P_h(W^{h} - W^{H}), \Omega)  \nonumber\\
&\lesssim &\tilde{\kappa}(h_0)\left( \|\Phi - \Phi_h\|_{1,\Omega} + \|\Phi - \Phi_H\|_{1,\Omega}\right).
\end{eqnarray*}
This together with the fact
\begin{eqnarray*}
\tilde{\eta}_h(P_h W^{h}, \Omega) = \tilde{\eta}_h(P_h W^{H} + P_h( W^{h} - W^{H}), \Omega)
\end{eqnarray*}
leads to
\begin{eqnarray*}\label{proof-6}
\tilde{\eta}_h(P_h W^{h}, \Omega) = \tilde{\eta}_h(P_h W^{H}, \Omega)
+ \mathcal{O} (\tilde{\kappa}(h_0))\left( \|\Phi - \Phi_h\|_{1,\Omega} + \|\Phi -
\Phi_H\|_{1,\Omega}\right),
\end{eqnarray*}
which is nothing but (\ref{lemma-bound-general-conc-3}). This
completes the proof.
\end{proof}

For the convenience of the statement of the following results, we need some definition.
For $(\Lambda,\Phi)\in\Theta$ and $\Phi_h\in V_h$,
%we say that the equivalence class $[\Phi_h]$ is
%an approximation of equivalence
%class $[\Phi]$
we say the equivalence class $[\Phi_h]$ approximate the equivalence class $[\Phi]$ if
$$
D_{\mathcal{H}}([\Phi_h],[\Phi]) <  D_{\mathcal{H}}([\Phi_h],[\tilde{\Phi}]), ~~~~\forall (\tilde{\Lambda},\tilde{\Phi})\in\Theta ~~\mbox{and}~ [\Phi] \neq [\tilde{\Phi}],
$$
%where the distance $d_{\rm ec}$ between two equivalence classes is defined by
%$$
%d_{\rm ec}([\Psi],[\Phi]) = \min_{U\in\mathcal{O}^{N\times N}}\|\Psi U-\Phi\|_{1,\Omega}.
%$$
the distance between sets
$X,Y\subset \mathcal{H}$ is defined by
\begin{eqnarray*}
D_{\mathcal{H}}(X,Y)=\sup_{\Phi \in X}\inf_{\Psi \in Y} \|\Phi-\Psi\|_{1,\Omega}.
\end{eqnarray*}

%Using the theoretical results of the linear boundary value problem in Appendix,   together with
Thanks to Theorem \ref{convergence-boundary}, Lemma \ref{eta-etah}, and Lemma \ref{lemma-bound-general}, by using the similar argument in \cite{chen-he-zhou10, dai-he-zhou12, dai-xu-zhou08}, we get the following theorem.

\begin{theorem}\label{thm-error-reduction}{\em (error reduction)}
Let  $\theta \in (0, 1)$ and $h_0 \ll 1$.
%Given a sufficiently fine initial mesh $\mathcal{T}_0$
Let $\{\Psi_k\}_{k\in\mathbb{N}_0}$ be a sequence of finite
element solutions  corresponding to a sequence of nested finite element spaces
$\{V_k\}_{k\in \mathbb{N}_0}$ produced by  Algorithm \ref{algorithm-AFEM-con-rate}. Assume  $[\Psi_{k_i}]$ is an approximation of some $[\Phi]$ with $\Phi$ being one solution of (\ref{problem-eigen-compact-L}),
    denote $k_{i+1}( > {k_i})$ the minimal index among all indexes $k(>k_i)$  which satisfy that $[\Psi_{k}]$ approximates  $[\Phi]$.
If Assumption {\bf A2} is true and $(\Lambda, \Phi)$ satisfies Assumption {\bf A3}, then
% for any $\Phi \in [\Phi]$ satisfying \eqref{assumption-a3} we have
\begin{eqnarray}\label{error-reduction-neq1-1}
\|\Phi - \Phi_{k_{i+1}} \|_{1,\Omega}^2+\gamma  \eta_{k_{i+1}}^2(\Phi_{k_{i+1}},
\mathcal{T}_{k_{i+1}}) \leq \xi^2 \left(  \| \Phi- \Phi_{k_{i}}\|_{1,\Omega}^2
+\gamma  \eta_{k_{i}}^2(\Phi_{k_{i}} , \mathcal{T}_{k_{i}})\right)~~~~~
\end{eqnarray}
with $\Phi_{k_{i+1}}\in  X_{\Phi,k_{i+1}}$ and $\Phi_{k_{i}}\in
  X_{\Phi,k_i}$ satisfying the a priori error estimates \eqref{H1-norm} and
\eqref{L2-norm} when $h$ is replaced by $h_{k_{i+1}}$ and $h_{k_i}$, respectively,
%with $\Phi_{k_i}\in [\Psi_{k_i}]\cap X_{\Phi,k_i}$,
 $\gamma>0$ and $\xi \in
(0,1)$ some constants depending only on the coercivity constant $c_a$, the shape regularity constant
$\gamma^{\ast}$, and the marking parameter $\theta$.
\end{theorem}

\begin{proof}
%Recall that $\Phi_{k_{i+1}}\in [\Psi_{k_{i+1}}]\cap X_{\Phi,k_{i+1}}$ and $\Phi_{k_{i}}\in
%[\Psi_{k_i}]\cap X_{\Phi,k_i}$ satisfy the a priori error estimates \eqref{H1-norm} and
%\eqref{L2-norm} when $h$ is replaced by $k_{i+1}$ and $k_i$.
For convenience,
we use $\Phi_h$, $\Phi_H$ to denote $\Phi_{k_{i+1}}$ and $\Phi_{k_i}$,
respectively. Then it is sufficient to prove that for $\Phi_h$
and $\Phi_H$, there holds,
\begin{eqnarray*}\label{thm-error-reduction-neq2}
\|\Phi-\Phi_h\|_{1,\Omega}^2 + \gamma \eta^2_{h}(\Phi_h, \Omega)\leq \xi^2
\big(\|\Phi-\Phi_H\|_{1,\Omega}^2 +\gamma \eta^2_{H}(\Phi_H,
\Omega)\big).
\end{eqnarray*}

%From Theorem 4.1, we know that $\Phi_h$ and $\Phi_H$ are solution of (\ref{problem-eigen-compact-dis}) in different finite element space.
Note that $\Phi_H$ and $\Psi_H$ are solutions of (\ref{problem-eigen-compact-dis}) and (\ref{problem-eigen-dis}), respectively. From the relationship of (\ref{problem-eigen-compact-dis}) and (\ref{problem-eigen-dis}), we have that if $[\Phi_H]$ and $[\Psi_H]$ approximate the same $[\Phi]$, then $\Phi_H =  \Psi_H U_H$ with $U_H$ being
some unitary transform.  Therefore, we obtain from Lemma \ref{eta-etah}  that D\"{o}rfler Marking strategy in Algorithm \ref{algorithm-AFEM-con-rate}
implies that there exists a constant $\theta'=\frac{\theta}{N^2} \in (0,1)$, such that
\begin{eqnarray*}
\sum_{\tau\in \mathcal{M}_H} \eta_H^2(\Phi_H, \tau) \geq \theta' \eta_H^2(\Phi_H, \Omega).
\end{eqnarray*}
Thus,   from
%$W^{h}=K(\Phi_h \Lambda_h-V \Phi_h-\mathcal{N}(\rho_{\Phi_h})\Phi_h)$,
$ W^{H}=K(\Phi_H \Lambda_H-V \Phi_H-\mathcal{N}(\rho_{\Phi_H})\Phi_H)$ and $\Phi_H=P_H W^{H}$, we have that D\"{o}rfler strategy is satisfied for $W^{H}$
with $\theta' = \frac{\theta}{N^2}$.
So we conclude from Theorem \ref{convergence-boundary}
that there exist constants $\tilde{\gamma}>0$ and $\tilde{\xi}\in (0,1)$
satisfying
\begin{eqnarray}\label{error-reduction-neq-4}
&&  \|W^{H}-P_h W^{H}\|_{1,\Omega}^2 +\tilde{\gamma}\tilde{\eta}^2_h(P_h W^{H}, \Omega) \leq
\tilde{\xi}^2  \big( \|W^{H}- \Phi_H \|_{1,\Omega}^2+ \tilde{\gamma} \eta^2_H(\Phi_H, \Omega)\big),~~~~
\end{eqnarray}
where the fact $\tilde{\eta}_H(P_H W^H, \Omega) = \eta_H(\Phi_H, \Omega)$ is used.

From (\ref{thm1-neq1}), we get that there exists constant $\hat{C}_1>0$ such that
\begin{eqnarray}\label{convergence-neq2}
\Big(1+\hat{C}_1 \tilde{\kappa} (h_0)\Big) \|\Phi-\Phi_H\|_{1,\Omega}^2 + \tilde{\gamma}  \eta^2_{H}(\Phi_H, \Omega) \ge \|W^{H}- P_H W^H\|_{1,\Omega}^2
+ \tilde{\gamma}  \eta^2_H(\Phi_H, \Omega).  ~~~~~~~
\end{eqnarray}
We obtain from Lemma \ref{lemma-bound-general} and the Young's inequality  that there exists
constant $ \hat{C}_2 >0$ such that

\begin{eqnarray}\label{convergence-neq3}
\|\Phi-\Phi_h\|_{1,\Omega}^2 +  \tilde{\gamma} \eta^2_{h}(\Phi_h,\Omega) &\le& (1+\delta_1) \|W^{H}-P_h W^{H}\|_{1,\Omega}^2
+  (1+\delta_1) \tilde{\gamma} \tilde{\eta}^2_h(P_h W^{H}, \Omega) ~~~~\nonumber \\
&+&\hat{C}_2(1 + \delta_1^{-1}) \tilde{\kappa}^2(h_0) \big(\|\Phi-\Phi_h\|_{1,\Omega}^2 + \|\Phi-\Phi_H\|_{1,\Omega}^2\big).\quad
\end{eqnarray}
where $\delta_1 \in (0, 1)$ satisfies $(1+\delta_1) \xi <1$.

Combining~\eqref{error-reduction-neq-4}, \eqref{convergence-neq2} with \eqref{convergence-neq3}, we have that
\begin{eqnarray*}
&&\Big(1 - \hat{C}_2(1 + \delta_1^{-1}) \tilde{\kappa}^2(h_0)\Big)\|\Phi-\Phi_h\|_{1,\Omega}^2+ \tilde{\gamma} \eta^2_{h}(\Phi_h, \Omega) \nonumber\\
& \leq & \Big((1+\delta_1)\tilde\xi^2 +  (1+\delta_1)\tilde\xi^2 \hat{C}_1 \tilde{\kappa}(h_0)  + \hat{C}_2(1+\delta_1^{-1})\tilde{\kappa}^2(h_0)\Big) \|\Phi-\Phi_H\|_{1,\Omega}^2 \nonumber\\
&&+(1+\delta_1))\tilde\xi^2 \tilde{\gamma} \eta^2_{H}(\Phi_H, \Omega). \quad
\end{eqnarray*}
Since $h_0\ll 1$ implies ${\tilde k}(h_0)\ll 1$, there holds
\begin{eqnarray*}
&&\|\Phi-\Phi_h\|_{1,\Omega}^2+\frac{\tilde{\gamma}}{1-\hat{C}_3\delta_1^{-1} \tilde{\kappa}^2(h_0)}\eta^2_{h}(\Phi_h, \Omega) \\
&\leq& \frac{(1+\delta_1) \tilde \xi^2+\hat{C}_3 \tilde{\kappa}(h_0)}{1-\hat{C}_3\delta_1^{-1}\tilde{\kappa}^2(h_0)}
\left(\|\Phi-\Phi_H\|_{1,\Omega}^2 +\frac{\tilde\xi^2\tilde{\gamma}}{(1+\delta_1) \tilde \xi^2
+\hat{C}_3 \tilde{\kappa}(h_0)}\eta^2_{H}(\Phi_H, \Omega)\right),
\end{eqnarray*}
with $\hat{C}_3$ some constant depending on $\hat{C}_1$ and $\hat{C}_2$.
Note that $h_0\ll 1$ implies ${\tilde k}(h_0)\ll 1$, we see that the constant $\xi$ defined by
\begin{eqnarray*}
\xi = \left(\frac{ (1+ \delta_1) \tilde{\xi}^2 +  \hat{C}_3
\tilde{\kappa}(h_0)}{1 - \hat{C}_3 \delta_1^{-1} \tilde{\kappa}^2(h_0)}\right)^{1/2}
\end{eqnarray*}
satisfies $\xi \in (0,1)$ when $h_0\ll 1$.

Finally, we arrive at (\ref{error-reduction-neq1-1}) by using the fact that
\begin{eqnarray}\label{gamma}
\frac{ \tilde{\xi}^2 \tilde{\gamma}}{ (1+ \delta_1)
\tilde{\xi}^2 + \hat{C}_3\tilde{\kappa}(h_0) }<\gamma
\quad{\rm with}\quad
\gamma = \frac{\tilde{\gamma}}{1 - \hat{C}_3 \delta_1^{-1}\tilde{\kappa}^2(h_0)}.
\end{eqnarray}
This completes the proof.
\end{proof}

We have from Theorem \ref{theorem-convergence-adaptive}  that
if $\{\Psi_k\}$ is obtained by Algorithm \ref{algorithm-AFEM-con-rate},
then there exists a subsequence $\{[\Psi_{k_i}]\}$ that converge to some equivalent class $[\Phi]$,
where $\Phi$ is a solution of (\ref{problem-eigen-compact-L}).
Here, a sequence $\{[\Psi_{k_i}]\}$ converges to a equivalent class $[\Phi]$ means that there exist
unitary  matrices $U_{k_i}\in\mathcal{O}^{N\times N}$, such that
$$
\lim_{i\rightarrow\infty}\Psi_{k_i} U_{k_i}=\Phi.
$$
Therefore, combining Theorem \ref{thm-error-reduction}, we  have the following theorem.
\begin{theorem}\label{theorem-convergence-rate} {\em (convergence rate)}
Let  $\theta \in (0, 1)$ and $h_0 \ll 1$.
Let $\{\Psi_k\}_{k\in\mathbb{N}_0}$ be a sequence of finite element approximations
obtained by Algorithm \ref{algorithm-AFEM-con-rate} and $\{[\Psi_{k_i}]\}$ be the subsequence
that converges to some  $[\Phi]$, where $\Phi$ is a solution of (\ref{problem-eigen-compact-L}).
If Assumptions {\bf A2} and {\bf A3} are satisfied, then there holds
%then
%there exists $\Phi_{k_i}\in[\Psi_{k_i}]\cap X_{\Phi,k_i}$ for all $i\in\mathbb{N}$, such that
%
\begin{eqnarray}\label{error-reduction-neq1}
\|\Phi - \Phi_{k_{i+1}} \|_{1,\Omega}^2+\gamma  \eta_{k_{i+1}}^2(\Phi_{k_{i+1}},
\mathcal{T}_{k_{i+1}}) \leq \xi^2 \left(  \| \Phi- \Phi_{k_{i}}\|_{1,\Omega}^2
+\gamma  \eta_{k_{i}}^2(\Phi_{k_{i}} , \mathcal{T}_{k_{i}})\right),~~~~
\end{eqnarray}
where $\Phi_{k_{i+1}}\in  X_{\Phi,k_{i+1}}$ and $\Phi_{k_{i}}\in
  X_{\Phi,k_i}$ satisfy the a priori error estimates (\ref{H1-norm}) and (\ref{L2-norm}) with $h$ being replaced by $h_{k_{i+1}}$ and $h_{k_i}$, respectively,
 $\gamma>0$ and $\xi \in (0,1)$ are constants depending only on the coercivity  constant $c_a$,
the shape regularity constant $\gamma^{\ast}$ and the marking parameter $\theta$. Therefore, the $k_m$-th iteration solution of
Algorithm \ref{algorithm-AFEM-con-rate} satisfies
\begin{eqnarray}\label{error-reduction-neq2}
\|\Phi - \Phi_{k_{m}} \|_{1,\Omega}^2+\gamma  \eta_{k_{m}}^2(\Phi_{k_{m}},
\mathcal{T}_{k_{m}}) \leq \xi^{2m} \left(  \| \Phi- \Phi_{k_{0}}\|_{1,\Omega}^2
+\gamma  \eta_{k_{0}}^2(\Phi_{k_{0}} , \mathcal{T}_{k_{0}})\right),
\end{eqnarray}
and
\begin{eqnarray}\label{error-reduction-neq3}
|\Lambda - \Lambda_{k_{m}} |  \lesssim \xi^{2m}.
\end{eqnarray}

In further, we have
\begin{eqnarray}\label{error-reduction-space}
d_{\mathcal{H}}(\Theta_{k_m}, \Theta)   \lesssim \xi^{2m}.
\end{eqnarray}
\end{theorem}

%
%In further, we can get the following result.
%\begin{theorem}\label{theorem-convergence-rate-space} {\em (convergence rate)}
%Let  $\theta \in (0, 1)$ and $h_0 \ll 1$.
%Let $\{\Psi_k\}_{k\in\mathbb{N}_0}$ be a sequence of finite element approximations
%obtained by Algorithm \ref{algorithm-AFEM-con-rate} and $\{[\Psi_{k_i}]\}$ be the subsequence
%that converges to some  $[\Phi]$, where $\Phi$ is a solution of (\ref{problem-eigen-compact-L}).
%%
%If Assumptions {\bf A2} and {\bf A3} are satisfied, then there holds
%\begin{eqnarray}\label{error-reduction-space}
%d_{\mathcal{H}}(\Theta_{k_m}, \Theta)   \lesssim \xi^{2m}.
%\end{eqnarray}
%\end{theorem}
%

\subsection{Complexity}
Finally, we study the complexity of Algorithm 4.1 in a class of functions. Following \cite{cascon-kreuzer-nochetto-siebert08, dai-xu-zhou08}, define
\begin{eqnarray*}
\mathcal{A}_{\gamma}^s=\{\Psi \in \mathcal{H}: |\Psi|_{s,\gamma} < \infty \},
\end{eqnarray*}
where $\gamma>0$ is some constant and
\begin{eqnarray*}
|\Psi|_{s,\gamma} = \sup_{\varepsilon
>0}\varepsilon \inf_{\{\mathcal{T} \subset \mathcal{T}_0:~
\inf_{\Psi_{\mathcal{T}} \in V_{\mathcal{T}}}(\|\Psi-\Psi_{\mathcal{T}}\|_{1,\Omega}^2 + (\gamma +1)
{\rm osc}^2_\mathcal{T}(\Psi_{\mathcal{T}}, \mathcal{T}))^{1/2} \leq \varepsilon\}} \big(\#\mathcal{T} - \#
\mathcal{T}_0\big)^s
\end{eqnarray*}
and  $\mathcal{T}\subset \mathcal{T}_0$ means $\mathcal{T}$ is a refinement of $\mathcal{T}_0$.
We see that, for all $\gamma>0$,
$\mathcal{A}_{\gamma}^s = \mathcal{A}_1^s$. For simplicity,
we use $\mathcal{A}^s$ to stand for $\mathcal{A}_1^s$, and use $|\Psi|_{s}$ to denote
 $|\Psi|_{s, \gamma}$. So $\mathcal{A}^s$ is the class of functions that can be approximated within a given tolerance
$\varepsilon$ by continuous piecewise polynomial functions over a
partition $\mathcal{T}_k$ with number of degrees of freedom satisfying
$\#\mathcal{T}_k-\# \mathcal{T}_0 \lesssim \varepsilon^{-1/s} |\Psi|_{s}^{1/s}$.

\begin{lemma}\label{optimal-marking}
Suppose    $\theta \in (0, 1)$ and $h_0\ll 1$. Let $\Psi_H$ and $\Psi_h$ be the solutions of \eqref{problem-eigen-dis} over a conforming mesh
$\mathcal{T}_H$ and its refinement $\mathcal{T}_h$, and $[\Psi_H]$ and $[\Psi_h]$ approximate the same
solution class $[\Phi]$,
where $\Phi$ is a solution of (\ref{problem-eigen-compact-L}).
Suppose Assumption {\bf A2} is true. If for some $\Phi \in [\Phi]$ satisfying \eqref{assumption-a3}, we have
\begin{eqnarray}\label{temp-3}
\|\Phi  - \Phi_h\|_{a,\Omega}^2 + \gamma_{\ast} {\rm osc}^2_h(\Phi_h,\Omega)
\leq \beta_{\ast}^2  \big(\|\Phi -\Phi_H\|_{a,\Omega}^2 + \gamma_{\ast} {\rm osc}^2_H(\Phi_H,\Omega)\big)
\end{eqnarray}
with   $\Phi_{h}\in  X_{\Phi, h}$ and $\Phi_{H}\in
  X_{\Phi, H}$ satisfying the a priori error estimates \eqref{H1-norm} and
\eqref{L2-norm} when $h$ is replaced by $h$ and $H$, respectively,
%$\Phi_{h}\in [\Psi_{h}]\cap X_{\Phi,h}$ and  $\Phi_{H}\in [\Psi_{H}]\cap X_{\Phi,H}$ ,
 $\gamma_{\ast}>0$ and $\beta_{\ast}\in
(0,\sqrt{\frac{1}{2}})$.
Then, the set $\mathcal{R}=\mathcal{R}_{\mathcal{T}_H\rightarrow\mathcal{T}_h}$ satisfies the following inequality
\begin{eqnarray*}\label{optimal-marking-neq1}
\sum_{\tau \in \mathcal{R}} \eta^2_H(\Phi_H, \tau) \geq \hat{\theta}  \sum_{\tau \in \mathcal{T}_H} \eta^2_H(\Phi_H, \tau),
\end{eqnarray*}
here
$\hat{\theta} =\frac{\tilde{C}_2(1-2\tilde{\beta}_{\ast}^2)}
{\tilde{C}_0(\tilde{C}_1+(1+2C_{\ast}^2\tilde{C}_1)
\tilde{\gamma}_{\ast})}$ ,
%$\tilde{C}_0=\max\{1,\frac{\tilde{C}_3}{\tilde{\gamma}_{\ast}}\}$,
with $\tilde{C}_0, \tilde{\beta_{\ast}}, C_{\ast}$ and
$\tilde{\gamma_{\ast}}$ being
constants defined in the proof. % procedure.
\end{lemma}

\begin{proof}
For
$W^{H} = K\Big( \Phi_H \Lambda_{H} - V\Phi_H-\mathcal{N}(\rho_{\Phi_H})\Phi_H \Big)$,
we observe from Lemma \ref{lemma-bound-general} that
\begin{eqnarray*}
\|\Phi  - \Phi_h  \|_{1,\Omega}&=& \|W^{H}-P_h W^{H} \|_{1,\Omega} \nonumber\\
&&+ \mathcal{O} (\tilde{\kappa}(h_0)) \left( \|W^{H}-P_H W^{H} \|_{1,\Omega}+ \|W^{H}-P_h W^{H}\|_{1,\Omega}\right),\nonumber\\
{\rm osc}_h(\Phi_h,\Omega)&=& \widetilde{{\rm osc}}_h(P_h W^{H}, \Omega) +
\mathcal{O} (\tilde{\kappa}(h_0)) \left( \|W^{H}-P_H W^{H}  \|_{1,\Omega} + \|W^{H}-P_H W^{H} \|_{1,\Omega}\right).
\end{eqnarray*}

Proceeding the similar procedure as in the proof of Theorem  \ref{thm-error-reduction}, we have
\begin{eqnarray}\label{lemma-complexity-eigen-bound-conc2}
&& \|W^H-P_h W^H\|_{a,\Omega}^2 + \tilde{\gamma}_{\ast} \widetilde{{\rm osc}}_h^2 (P_h W^H, \Omega) \nonumber\\
&\leq & \tilde{\beta}_{\ast}^2 \big(\|W^H-P_H W^H\|_{a,\Omega}^2 +
\tilde{\gamma}_{\ast} \widetilde{{\rm osc}}_H^2 (P_H W^H, \Omega)\big)
\end{eqnarray}
with
\begin{eqnarray}\label{complexity-eigen-boundary-beta-gamma}
\tilde{\beta}_{\ast} = \left(\frac{\beta_{\ast}^2(1+\delta_1)+\hat{C}_4
\tilde{\kappa}(h_0)}{1-\hat{C}_4 \delta_1^{-1} \tilde{\kappa}^2(h_0)}\right)^{1/2}, \quad
\tilde{\gamma}_{\ast} =
 \frac{\gamma_{\ast}}{1 -\hat{C}_4 \delta_1^{-1} \tilde{\kappa}^2(h_0)},
\end{eqnarray}
where $\hat{C}_4 $ is some positive constant and $\delta_1 \in (0, 1)$ is some constant as shown
in the proof of Theorem \ref{thm-error-reduction}.

Set
$\tilde{C}_0=\max\{1,\frac{\tilde{C}_3}{\tilde{\gamma}_{\ast}}\}$,
we get from (\ref{boundary-lower}) that
\begin{eqnarray*}
(1-2\tilde{\beta_{\ast}}^2) \tilde{C}_2 \tilde{\eta}^2_{H}(P_H W^{H}, \Omega) \leq (1 -
2\tilde{\beta_{\ast}}^2) \left(\|W^{H} - P_H W^{H}\|^2_{a, \Omega} + \tilde{C}_3 \widetilde{{\rm osc}}_H^2(P_H W^{H},\Omega)
\right) \nonumber\\
\leq  \tilde{C}_0 (1 - 2 \tilde{\beta_{\ast}}^2) \left(\|W^{H} - P_H W^{H}\|^2_{a, \Omega} + \tilde{\gamma}_{\ast}
\widetilde{{\rm osc}}_H^2(P_H W^{H},\Omega)\right),
\end{eqnarray*}
which together with \eqref{lemma-complexity-eigen-bound-conc2}
produces
\begin{eqnarray}\label{optimal-mariking-neq1}
\frac{\tilde{C}_2}{\tilde{C}_0} (1-2\tilde{\beta_{\ast}}^2)  \sum_{\tau \in \mathcal{T}_H} \tilde{\eta}^2_H(P_H W^{H}, \tau) \leq
\big(\|W^{H} - P_H W^{H}\|^2_{a, \Omega} +
\tilde{\gamma}_{\ast}\widetilde{{\rm osc}}_H^2(P_H W^{H},\Omega)  \nonumber\\
 -\|W^{H} - P_h W^{H}\|^2_{a, \Omega} -2 \tilde{\gamma}_{\ast}\widetilde{{\rm osc}}_h^2(P_h W^{H},\Omega)\big).~~~~~~
\end{eqnarray}
Thus using equality
\begin{eqnarray*}
\|W^{H}-P_H W^{H}\|_{a,\Omega}^2-\|W^{H}-P_h W^{H}\|_{a,\Omega}^2 = \|P_HW^{H}-P_h W^{H}\|_{a,\Omega}^2
\end{eqnarray*}
and Theorem \ref{localized-upper-bound}, we obtain that
\begin{eqnarray}\label{temp-4}
\|W^{H}-P_H W^{H}\|_{a,\Omega}^2-\|W^{H}-P_h W^{H}\|_{a,\Omega}^2\leq \tilde{C}_1 \sum_{\tau \in \mathcal{R}}\tilde{\eta}_H^2(P_H W^{H},\tau).
\end{eqnarray}
By  the triangle inequality,  the inverse inequality, and the Young's inequality,
we get
\begin{eqnarray*}
\sum_{\tau \in \mathcal{T}_H\cap \mathcal{T}_h}\widetilde{{\rm osc}}_H^2(P_H W^{H},\tau)\leq 2 \sum_{\tau \in \mathcal{T}_H\cap
\mathcal{T}_h}\widetilde{{\rm osc}}_h^2(P_h W^{H},\tau) +2C_{\ast}^2 \|P_H W^{H}-P_h W^{H}\|_{a,\Omega}^2
\end{eqnarray*}
where $C_{\ast}$ is a positive constant depending on the shape regularity constant $\gamma^{\ast}$.
Hence, using the fact
\begin{eqnarray*}
\widetilde{{\rm osc}}_H^2(P_H W^{H},\tau) \leq  \tilde{\eta}^2_H(P_H W^{H}, \tau)\quad \forall \tau \in \mathcal{T}_H,
\end{eqnarray*}
we may estimate as follows
\begin{eqnarray}\label{optimal-mariking-neq2}
&& \widetilde{{\rm osc}}_H^2(P_H W^{H},\Omega)
-2 \widetilde{{\rm osc}}_h^2(P_h W^{H},\Omega)\nonumber\\
&\leq&  \sum_{\tau \in \mathcal{R}} \tilde{\eta}^2_H(P_H W^{H}, \tau) + \sum_{\tau \in
\mathcal{T}_H\cap \mathcal{T}_h}\widetilde{{\rm osc}}_H^2(P_H
W^{H},\tau) -2 \sum_{\tau \in \mathcal{T}_H\cap \mathcal{T}_h}
\widetilde{{\rm osc}}_h^2(P_h W^{H},\tau) \nonumber\\
&\leq&  \sum_{\tau \in \mathcal{R}} \tilde{\eta}^2_H(P_H W^{H}, \tau)
+ 2 C_{\ast}^2  \|P_H W^{H}-P_h W^{H}\|^2_{a,\Omega} \nonumber\\
&\leq& (1 +  2 C_{\ast}^2 \tilde{C}_1)  \sum_{\tau \in \mathcal{R}} \tilde{\eta}^2_H(P_H W^{H}, \tau).
\end{eqnarray}
Combining \eqref{optimal-mariking-neq1}, \eqref{temp-4} and (\ref{optimal-mariking-neq2}), we then arrive at
\begin{eqnarray*}
~~~~\frac{\tilde{C}_2}{\tilde{C}_0} (1-2\tilde{\beta_{\ast}}^2) \sum_{\tau \in \mathcal{T}_H} \tilde{\eta}^2_H(P_H W^{H}, \tau)
\leq (\tilde{C}_1 + (1 + 2 C_{\ast}^2 \tilde{C}_1) \tilde{\gamma}_{\ast} ) \sum_{\tau \in \mathcal{R}} \tilde{\eta}^2_H(P_H W^{H}, \tau),
\end{eqnarray*}
that is,
\begin{eqnarray*}
\sum_{\tau \in \mathcal{R}} \eta^2_H(\Phi_H, \tau) \geq \hat{\theta}  \sum_{\tau \in \mathcal{T}_H}
\eta^2_H(\Phi_H, \tau)
\end{eqnarray*}
with
\begin{eqnarray*}
\hat{\theta} & = &
\frac{\tilde{C}_2(1-2\tilde{\beta_{\ast}}^2)}{\tilde{C}_0(
\tilde{C}_1  + (1 +  2 C_{\ast}^2 \tilde{C}_1)
\tilde{\gamma}_{\ast})}.
\end{eqnarray*}
This completes the proof.
\end{proof}

Similar for  the boundary value problem \cite{cascon-kreuzer-nochetto-siebert08} and the  linear eigenvalue problems \cite{dai-he-zhou12}, to analyze the complexity of Algorithm \ref{algorithm-AFEM}, we need more requirements than for the convergence rate.

\begin{assumption}\label{assump-5.1}
\begin{enumerate}
\item  The marking parameter $\theta$ satisfies $\theta \in (0, \theta_{\ast})$, with
\begin{eqnarray*}
\theta_{\ast} = \frac{1}{N^2} \frac{C_2 \gamma}{C_3 ( C_1 + (1 + 2C_{\ast}^2 C_1)\gamma )}).
\end{eqnarray*}
\item The marked $\mathcal{M}_{h_k}$ satisfy (\ref{mark-strategy}) with minimal cardinality.
\item The distribution of refinement edges on $\mathcal{T}_{h_0}$ satisfies condition (b) of section 4
in \cite{stevenson08}.
\end{enumerate}
\end{assumption}

We mention that D\"{o}rfler  Marking Strategy selects the marked set
$\mathcal{M}_k$ with minimal cardinality.
\begin{lemma}
%\label{Cardinality}
  Let  $\theta \in (0, 1)$ and $h_0\ll 1$, $\{\Psi_k\}_{k\in\mathbb{N}_0}$ be a sequence of finite element solutions
corresponding to a sequence of nested finite element
spaces $\{V_k\}_{k\in \mathbb{N}_0}$ produced by  Algorithm \ref{algorithm-AFEM-con-rate}.
%and $\theta\in \frac{1}{N^2}(0,\frac{C_2 \gamma}{ C_3(C_1 + (1 + 2C^2_{\ast} C_1)\gamma )})$.
 Suppose Assumption {\bf A2} is true.
If~$[\Psi_k]$ approximates the solution class
$[\Phi]$,
where $\Phi$ is a solution of (\ref{problem-eigen-compact-L}), then
 for any $\Phi\in [\Phi] \cap \mathcal{A}^s$ satisfying \eqref{assumption-a3}, we have
\begin{eqnarray}\label{DOF}
\# \mathcal{M}_k \lesssim \left(\|\Phi-\Phi_k\|_{a,\Omega}^2 +
\gamma osc^2_k(\Phi_k, \Omega)\right)^{-1/2s} |\Phi|_{s}^{1/s},
\end{eqnarray}
where $\Phi_{k}\in   X_{\Phi,k}$ satisfies the a priori error estimates \eqref{H1-norm} and
\eqref{L2-norm} with $h$ being replaced by  $h_{k_i}$, and the hidden constant depends on the discrepancy between the
marking parameter  $\frac{1}{N^2} \frac{C_2 \gamma}{C_3( C_1 + (1 +
2C^2_{\ast}C_1)\gamma)}$ and $\theta$.
\end{lemma}
\begin{proof}
Let $\alpha, \alpha_1 \in (0,1)$ satisfy $\alpha_1\in (0,\alpha) $
and
$$
\theta < \frac{1}{N^2} \frac{C_2 \gamma}{ C_3(C_1 + (1 + 2C_{\ast}^2C_1)\gamma)
}(1-\alpha^2).
$$
We choose $\delta_1\in (0,1)$ to satisfy $(1 + \delta_1) \tilde{\xi}^2 <1$
%(\ref{error-reduction-neq-delta})
 and
\begin{eqnarray}\label{thm-complexity-delta-cond-1}
(1+\delta_1)^2 \alpha_1^2 \leq \alpha^2,
\end{eqnarray}
which implies
\begin{eqnarray}\label{thm-complexity-delta-cond-3}
(1+\delta_1) \alpha_1^2 <1.
\end{eqnarray}

Define
$$
\varepsilon = \frac{1}{\sqrt{2}} \alpha_1
\big(\|\Phi-\Phi_k\|_{a,\Omega}^2 + \gamma osc^2_k(\Phi_k,
\Omega)\big)^{1/2}
$$
and let $\mathcal{T}_{\varepsilon}$ be a refinement of
$\mathcal{T}_0$ with minimal degrees of freedom satisfying
\begin{eqnarray}\label{complexity-optimal-neq0}
\|\Phi - \Phi_{\varepsilon}\|_{a,\Omega}^2 + (\gamma + 1)
osc^2_{\varepsilon}(\Phi_{\varepsilon}, \Omega) \leq \varepsilon^2.
\end{eqnarray}
We get from  $\Phi\in\mathcal{A}^s$ that
\begin{eqnarray*}\label{upper-bound-dof-neq1}
\#\mathcal{T}_{\varepsilon} - \# \mathcal{T}_0 \lesssim
{\varepsilon}^{-1/s} |\Phi|_{s}^{1/s}.
\end{eqnarray*}
Let $\mathcal{T}_{\ast}$ be the smallest  common refinement of $\mathcal{T}_k$
and $\mathcal{T}_{\varepsilon}$. Since
$W^{\varepsilon}=K(\Phi_\varepsilon\Lambda_\varepsilon -V
\Phi_\varepsilon-\mathcal{N}(\rho_{\Phi_\varepsilon})\Phi_\varepsilon)$,
we obtain from the triangle inequality, the inverse inequality, and the Young's
inequality that
\begin{eqnarray*}
\widetilde{osc}^2_{\ast}(P_{\ast}W^{\varepsilon},\Omega) \leq
2\widetilde{osc}^2_{\ast}(P_{\varepsilon}W^{\varepsilon},\Omega)
+2C^2_{\ast}\|P_{\varepsilon}W^{\varepsilon}-
P_{\ast}W^{\varepsilon}\|^2_{a,\Omega},
\end{eqnarray*}
where $P_{\varepsilon}$ and $P_{\ast}$ are Galerkin projections on
$\mathcal{T}_{\varepsilon}$ and $\mathcal{T}_{\ast}$ defined by
\eqref{projection}.
Note that
\begin{eqnarray*}\label{ortho-relation}
\|W^{\varepsilon}-P_{\ast}W^{\varepsilon}\|^2_{a, \Omega} =
\|W^{\varepsilon} - P_{\varepsilon}W^{\varepsilon}\|^2_{a, \Omega} -
\|P_{\ast}W^{\varepsilon} - P_{\varepsilon}W^{\varepsilon}\|_{a,
\Omega}^2,
\end{eqnarray*}
we have
\begin{eqnarray*}
\| W^{\varepsilon} - P_{\ast} W^{\varepsilon}\|_{a,\Omega}^2 +
\frac{1}{2 C_{\ast}^2}\widetilde{osc}^2_{\ast}
(P_{\ast}W^{\varepsilon},\Omega) ~\leq~  \|W^{\varepsilon} -
P_{\varepsilon} W^{\varepsilon}\|_{a,\Omega}^2+ \frac{1}{C_{\ast}^2}
osc^2_{\varepsilon}(P_{\varepsilon}W^{\varepsilon},\Omega).
\end{eqnarray*}
Since (\ref{gamma-boundary}) implies $\tilde{\gamma} \leq \frac{1}{2
C_{\ast}^2}$,  we get that
\begin{eqnarray*}
\| W^{\varepsilon} - P_{\ast} W^{\varepsilon}\|_{a,\Omega}^2 +
\tilde{\gamma}\widetilde{osc}^2_{\ast}(P_{\ast}W^{\varepsilon},\Omega)
&\leq & \|W^{\varepsilon} - P_{\varepsilon}
W^{\varepsilon}\|_{a,\Omega}^2 + \frac{1}{ C_{\ast}^2}
osc^2_{\varepsilon}(P_{\varepsilon}W^{\varepsilon},\Omega)\nonumber\\
&\leq & \|W^{\varepsilon} - P_{\varepsilon}
W^{\varepsilon}\|_{a,\Omega}^2+ (\tilde{\gamma} + \sigma)
osc^2_{\varepsilon}(P_{\varepsilon}W^{\varepsilon},\Omega),
\end{eqnarray*}
where $\sigma = \frac{1}{C_{\ast}^2} - \tilde{\gamma} \in (0, 1)$.
We may conclude from using the similar argument as that in proof of Theorem \ref{thm-error-reduction}  that
%Applying the similar argument as that in the proof of Theorem
%\ref{thm-error-reduction} when (\ref{lemma-bound-general-conc-3}) is
%replaced by (\ref{lemma-bound-general-conc-2}), we arrive at
\begin{eqnarray}\label{complexity-optimal-neq2}
\|\Phi - \Phi_{\ast}\|_{a,\Omega}^2 + \gamma
osc^2_{\ast}(\Phi_{\ast}, \Omega) &\leq & \alpha_0^2\left(
\|\Phi-\Phi_{\varepsilon}\|_{a,\Omega}^2 + (\gamma +
\sigma)osc^2_{\varepsilon}(P_{\varepsilon}
W^{\varepsilon},\Omega)\right) ~~~\nonumber\\
&\leq & \alpha_0^2\left(\|\Phi - \Phi_{\varepsilon}\|_{a,\Omega}^2 +
(\gamma + 1) osc^2_{\varepsilon}
(P_{\varepsilon}W^{\varepsilon},\Omega)\right),
\end{eqnarray}
where
\begin{eqnarray*}
\alpha_0^2 = \frac{ (1+ \delta_1)  +  \hat{C}_3\tilde{\kappa}(h_0) }{1 -
\hat{C}_3  \delta_1^{-1} \tilde{\kappa}^2(h_0)}
\end{eqnarray*}
and $\hat{C}_3$ is the constant appearing in the proof of Theorem
\ref{thm-error-reduction}. We derive from
(\ref{complexity-optimal-neq0}) and (\ref{complexity-optimal-neq2})
that
\begin{eqnarray*} \|\Phi - \Phi_{\ast}\|_{a,\Omega}^2
+ \gamma osc^2_{\ast}(\Phi_{\ast}, \mathcal{T}_{\ast}) \leq
\check{\alpha}^2 \big(\|\Phi-\Phi_k\|^2_{a,\Omega} + \gamma
osc^2_k(\Phi_k, \mathcal{T}_k)\big)
\end{eqnarray*}
with  $\check{\alpha} = \frac{1}{\sqrt{2}} \alpha_0 \alpha_1$. Using
(\ref{thm-complexity-delta-cond-3}), we obtain $\check{\alpha}^2\in
(0,\frac{1}{2})$ when $h_0\ll 1$.
 Set $\check{\theta} =
\frac{\tilde{C}_2(1-2\hat{\alpha}^2)}{\tilde{C}_0( \tilde{C}_1 +
(1+2 C_{\ast}^2 \tilde{C}_1) \hat{\gamma})}$, $\hat{\gamma}=
\frac{\gamma}{1 - \hat{C}_4 \delta_1^{-1}  \tilde{\kappa}^2(h_0)}$,
$\tilde{C}_0 = \max(1, \frac{\tilde{C}_3}{\hat{\gamma}})$, and
$\hat{\alpha}^2= \frac{ (1+ \delta_1)\check{\alpha}^2 + \hat{C}_4
\tilde{\kappa}(h_0) }{1 - \hat{C}_4 \delta_1^{-1} \tilde{\kappa}^2(h_0)}.$
Denote $\mathcal{R}=\mathcal{R}_{\mathcal{T}_k\rightarrow
\mathcal{T}_{\ast}}$ the refined elements from $\mathcal{T}_k$ to
$\mathcal{T}_{\ast}$, we obtain from Lemma \ref{optimal-marking} that $\mathcal{T}_{\ast}$
satisfies
\begin{eqnarray*}
\sum_{\tau\in \mathcal{R}}\eta^2_k(\Phi_k, \tau) \geq
\check{\theta}\sum_{\tau\in \mathcal{T}_k} \eta^2_k(\Phi_k,\tau).
\end{eqnarray*}

 Similar to the illustration in proof of \ref{thm-error-reduction}, from the relationship of (\ref{problem-eigen-compact-dis}) and (\ref{problem-eigen-dis}), we also have that
$\Psi_{k} = \Phi_{k} U_{k}$ with $U_k$ being some unitary matrix. Therefore, from Lemma \ref{eta-etah}, we have that there exists $\check{\theta}' =  \frac{\check{\theta}}{N^2}$, such that
\begin{eqnarray}\label{lemma-4.10-neq1}
\sum_{\tau\in \mathcal{R}}\eta^2_k(\Psi_k, \tau) \geq
\check{\theta}'\sum_{\tau\in \mathcal{T}_k} \eta^2_k(\Psi_k,\tau).
\end{eqnarray}

We obtain from the definition of $\gamma$
(see \eqref{gamma}) and $\tilde{\gamma}$ (see
\eqref{gamma-boundary}) that $\frac{\tilde{C}_3}{\hat{\gamma}} \geq \tilde{C}_3 C_{\ast}^2$. Note that
 $\tilde{C}_3$ and $C_{\ast}$ are constants appeared in upper bound, without loss of generality, we can assume
 $\tilde{C}_3 \geq 1$ and $C_{\ast} \geq 1$. Hence we have ${\tilde
C}_0 = \frac{\tilde{C}_3}{\hat{\gamma}}.$  Since $h_0 \ll 1$, we get
that $\hat{\gamma}>\gamma$ and $\hat{\alpha}\in
(0,\frac{1}{\sqrt{2}}\alpha)$ from
(\ref{thm-complexity-delta-cond-1}). We observe from \eqref{C1},
(\ref{coef-eigen-bound}) and $\hat{\gamma}>\gamma$ that
\begin{eqnarray*}
\check{\theta}'&=& \frac{1}{N^2}\frac{\tilde{C}_2(1-2\hat{\alpha}^2)}
{\frac{\tilde{C}_3}{\hat{\gamma}} (\tilde{C}_1 + (1 + 2
C_{\ast}^2\tilde{C}_1)\hat{\gamma} )} \geq \frac{1}{N^2}
\frac{\tilde{C}_2}{\tilde{C}_3( \frac{\tilde{C}_1}{\hat{\gamma}} + 1
+ 2 C_{\ast}^2\tilde{C}_1)}(1-\alpha^2)\nonumber\\
&=& \frac{1}{N^2} \frac{\frac{C_2}{(1-\tilde{C}\tilde{\kappa}
(h_0))^2}}{\frac{C_3}{(1-\tilde{C}\tilde{\kappa}(h_0))^2}
(\frac{C_1}{\hat{\gamma}((1+\tilde{C}\tilde{\kappa}(h_0))^2)}+1
+2C_{\ast}^2\frac{C_1}{(1+\tilde{C}\tilde{\kappa}(h_0))^2})}(1-\alpha^2)
\nonumber\\
&\geq&  \frac{1}{N^2} \frac{C_2 }{C_3 ( \frac{C_1}{\gamma} + (1 + 2
C_{\ast}^2 C_1) )}(1-\alpha^2) = \frac{1}{N^2} \frac{C_2 \gamma }{C_3 ( C_1 + (1 +
2 C_{\ast}^2 C_1)\gamma )}(1-\alpha^2)> \theta
\end{eqnarray*}
when $h_0 \ll 1$.

Therefore, from (\ref{lemma-4.10-neq1}), we deduce
\begin{eqnarray}\label{lemma-4.10-neq2}
\sum_{\tau\in \mathcal{R}}\eta^2_k(\Psi_k, \tau) \geq
\theta \sum_{\tau\in \mathcal{T}_k} \eta^2_k(\Psi_k,\tau).
\end{eqnarray}
Since $\mathcal{M}_k$ satisfies (\ref{lemma-4.10-neq2}) with minimal cardinality,  we arrive at
\begin{eqnarray*}
\#\mathcal{M}_k &\leq& \#\mathcal{R}_{\mathcal{T}_{\ast} \to \mathcal{T}_k} \leq \#\mathcal{T}_{\ast} -
\#\mathcal{T}_k
\leq \#\mathcal{T}_{\varepsilon}- \#\mathcal{T}_0 \nonumber\\
&\lesssim &(\frac{1}{\sqrt{2}}\alpha_1)^{-1/s} \left(\|\Phi -
\Phi_k\|_{a,\Omega}^2 + \gamma osc^2_k(\Phi_k, \mathcal{T}_k)\right)
^{-1/2s} |\Phi|_{s}^{1/s},
\end{eqnarray*}
which is nothing but (\ref{DOF}) with an explicit
dependence on the discrepancy between $\theta$ and $\frac{C_2
\gamma}{C_3( C_1 + (1 + 2C_{\ast}^2C_1)\gamma )}$ via $\alpha_1$.
This completes the proof.
\end{proof}

\begin{theorem}
\label{theorem-optimal-complexity} {\em (optimal complexity)}
Let  $\theta \in (0, 1)$ and $h_0 \ll 1$.   Assume that Assumption {\bf A2} is satisfied
and \eqref{problem-eigen-compact-L}  has m solutions (up to the invariance of unitary transform),
which are denoted as $[\Phi^{(l)}](l=1,\cdots,m)$ where $m$ can be chosen to be $\infty$.  Let $\{\Psi_k\}_{k\in\mathbb{N}_0}$ be a
sequence of finite element solutions  corresponding to a sequence of nested finite element spaces $\{V_k\}_{k\in \mathbb{N}_0}$
produced by  Algorithm \ref{algorithm-AFEM-con-rate}.
%, $\{[\tilde \Phi_k]\}_{k\in\mathbb{N}_0}$ be the corresponding finite element
%solution spaces.
Then  the following quasi-optimal bound is valid
%the $n$-th iterate solution space $[\tilde\Phi_{h_n}]$ satisfies the quasi-optimal bound
\begin{eqnarray}
\#\mathcal{T}_n-\#\mathcal{T}_0 \lesssim \sum_{l=1}^{m} \left(\|\Phi^{l}-\Phi^{l}_{k_{n_{l}}}
\|_{1,\Omega}^2 + \gamma {\rm osc}^2_{k_{n_{l}}} (\Phi^{l}_{k_{n_{l}}} , \Omega)\right)^{-1/2s},
\end{eqnarray}
where $\Phi^{l} \in [\Phi^{(l)}] \cap \mathcal{A}^s$ satisfies \eqref{assumption-a3}, $\Phi_{k_{n_l}}^l\in
 X_{\Phi^l,k_{n_l}}$ satisfies the a priori error estimates \eqref{H1-norm} and
\eqref{L2-norm} with $h$ being replaced by $h_{k_{n_l}}$, and the hidden constant depends on the exact solution $\Phi^{l}$
and the discrepancy between $\theta$ and $\frac{1}{N^2} \frac{C_2 \gamma}{C_3( C_1 + (1 + 2 C_{\ast}^2 C_1)\gamma) }$.
Here, $n_{l}$ and  $k_{n_{l}}$ are the total number and the maximal index
of iteration which approximate $[\Phi^{(l)}](l=1, \cdots, m)$ among the $n$ iteration, respectively.
\end{theorem}
\begin{proof}
Assume that among the iterate solution spaces $\{[\Psi_i]\}_{i=1}^{n}$,
there are $n_{l}$ approximations for $[\Phi^{(l)}](l=1, \cdots, m),$
which are denoted by $[\Psi_{k_i}]$ ($i=1, \cdots, n_{l}$). Here, $\sum_{l=1}^{m} n_{l} = n $, and $n_{l}$ can be $0$.
Recall that (see Theorem 6.1 in \cite{stevenson08})
\begin{eqnarray*}
& &\#\mathcal{T}_n - \#\mathcal{T}_0 \lesssim \sum_{l=1}^{m}\sum_{i=1}^{n_l}
\#\mathcal{M}_{k_i},
\end{eqnarray*}
we obtain from (\ref{DOF}) that
\begin{eqnarray*}
\#\mathcal{T}_n - \#\mathcal{T}_0 \lesssim \sum_{l=1}^{m}\sum_{i=1}^{n_l}
\left( \|\Phi^{l} - \Phi^{l}_{k_i} \|_{1,\Omega}^2+\gamma {\rm osc}^2_{k_i}(\Phi^{l}_{k_i},
\Omega)\right)^{-1/2s}(|\Phi^{l}|_{s}^{1/s}).
\end{eqnarray*}
Note that (\ref{lower-bound}) implies
\begin{eqnarray*}
\|\Phi^{l} - \Phi^{l}_{k_i}\|^2_{1,\Omega} + \gamma \eta^2_{k_i}(\Phi^{l}_{k_i} , \Omega)
\leq \check{C}  \big( \|\Phi^{l} - \Phi^{l}_{k_i}\|^2_{1,\Omega} +
\gamma {\rm osc}^2_{k_i}(\Phi^{l}_{k_i}, \Omega)\big),
\end{eqnarray*}
where $ \check{C} = \max(1 + \frac{\gamma}{C_2}, \frac{C_3}{C_2}),$
we conclude
\begin{eqnarray*}
\#\mathcal{T}_n - \#\mathcal{T}_0 &\lesssim& \sum_{l=1}^{m}\sum_{i=1}^{n_l}
\left( \big(\|\Phi^{l} - \Phi^{l}_{k_i}\|_{1,\Omega}^2 + \gamma\eta^2_{k_i}(\Phi^{l}_{k_i},
\Omega)\big)\right)^{-1/2s}(|\Phi^{l}|_{s}^{1/s}).
\end{eqnarray*}

Since  (\ref{error-reduction-neq1}) yields
\begin{eqnarray*}
\|\Phi^{l}-\Phi^{l}_{k_{n_{l}}}\|_{1,\Omega}^2 + \gamma \eta^2_{k_{n_{l}}}(\Phi^{l}_{k_{n_{l}}}, \Omega)
 \leq \xi^{2 (n_{l} - i)}  \left(\|\Phi^{l}- \Phi^{l}_{k_{i}} \|_{1,\Omega}^2
 + \gamma \eta^2_{k_{i}}(\Phi^{l}_{k_{i}}, \Omega)\right),
\end{eqnarray*}
we arrive at
\begin{eqnarray*}
\#\mathcal{T}_n-\#\mathcal{T}_0 &\lesssim& \sum_{l=1}^{m}\Big(|\Phi^{l}|_{s}^{1/s}
\big(\|\Phi^{l}-\Phi^{l}_{k_{n_{l}}}\|_{1,\Omega}^2 + \gamma\eta^2_{k_{n_{l}}} (\Phi^{l}_{k_{n_{l}}},
\Omega)\big)^{-1/2s} \sum_{i=1}^{n_l}\xi^{\frac{n_l-i}{s}}\Big) \nonumber\\
&\lesssim& \sum_{l=1}^{m}\Big(|\Phi^{l}|_{s}^{1/s} \big(\|\Phi^{l}-\Phi^{l}_{k_{n_{l}}}\|_{1,\Omega}^2
+ \gamma \eta^2_{k_{n_{l}}} (\Phi^{l}_{k_{n_{l}}}, \Omega)\big)^{-1/2s}\Big),
\end{eqnarray*}
where the fact $\xi<1$ is used.

Thus we obtain from ${\rm osc}_k(\Phi^{l}_k, \Omega) \leq  \eta_k(\Phi^{l}_k,\Omega)$ that
\begin{eqnarray*}
\#\mathcal{T}_n-\#\mathcal{T}_0 \lesssim \sum_{l=1}^{m} \left(\|\Phi^{l}-\Phi^{l}_{k_{n_{l}}} \|_{1,\Omega}^2
+ \gamma {\rm osc}^2_{k_{n_{l}}}(\Phi^{l}_{k_{n_{l}}}, \Omega)\right)^{-1/2s}.
\end{eqnarray*}
This completes the proof.
\end{proof}

\section{Numerical examples} \label{sec-numerical}
\setcounter{equation}{0}

In this section, we shall present some numerical simulations for three typical molecular systems:
$C_9H_8O_4$(Aspirin), $C_5H_9O_2 N$($\alpha$ amino acid), and $C_{60}$(fullerene), which support our theory.
Due to the length limitation for the paper,  we only show the results for pseudopotential approximations
for illustration.

Our numerical experiments are carried out on LSSC-III in
the State Key Laboratory of Scientific and Engineering Computing, Chinese Academy of Sciences, and our package RealSPACES
(Real Space Parallel Adaptive Calculation of Electronic Structure)  that are based on the toolbox PHG \cite{phg} of the State
Key Laboratory of Scientific and Engineering Computing, Chinese Academy of Sciences.

In our computations, we use the norm-conserving pseudopotential obtained by fhi98PP software and
the LDA exchange-correlation potential. We use Algorithm 4.1 and  apply the standard quadratic finite element discretizations.
Since the analytic solutions are not known even for the simplest systems,
we only show the convergence curve of the a posteriori error estimator $\eta_k^2(\Psi_k,\Omega)$ in our figures.
%
%The molecular illustrations in this paper is produced with PyMOL.
The mesh and density illustrations are drawn using ParaView.\vskip 0.2cm

{\bf Example 1}: Aspirin $C_9H_8O_4$.

The ground state energy obtained by SIESTA is $-119.621~a.u.$.
In our computations, we choose the computational domain to be $\Omega =[-20.0,20.0]^3$.

The atomic configuration, the calculated ground state charge density and the associated computational mesh
% for molecule $C_9H_8O_4$
are shown in Figure \ref{fig:C9H8O4}.  First, comparing the configuration figure (the left one of Figure \ref{fig:C9H8O4})
and the charge density figure (the middle one of Figure \ref{fig:C9H8O4}), we can see qualitatively that our calculations are correct,
the carbon-hydrogen bonds, carbon-oxygen bonds, and the oxygen-hydrogen bonds are preserved very well.
If we take a detailed look at the charge density figure, we can further see that the charge is more concentrative
around the oxygen than around the carbon. We also see from the mesh figure (the right one of Figure \ref{fig:C9H8O4})
and the charge density figure that our error estimator can catch the oscillations of the charge density very well,
which qualitatively confirms  that our error estimator is efficient.

\begin{figure}[!htbp]
  \begin{center}
    \includegraphics[width=4cm]{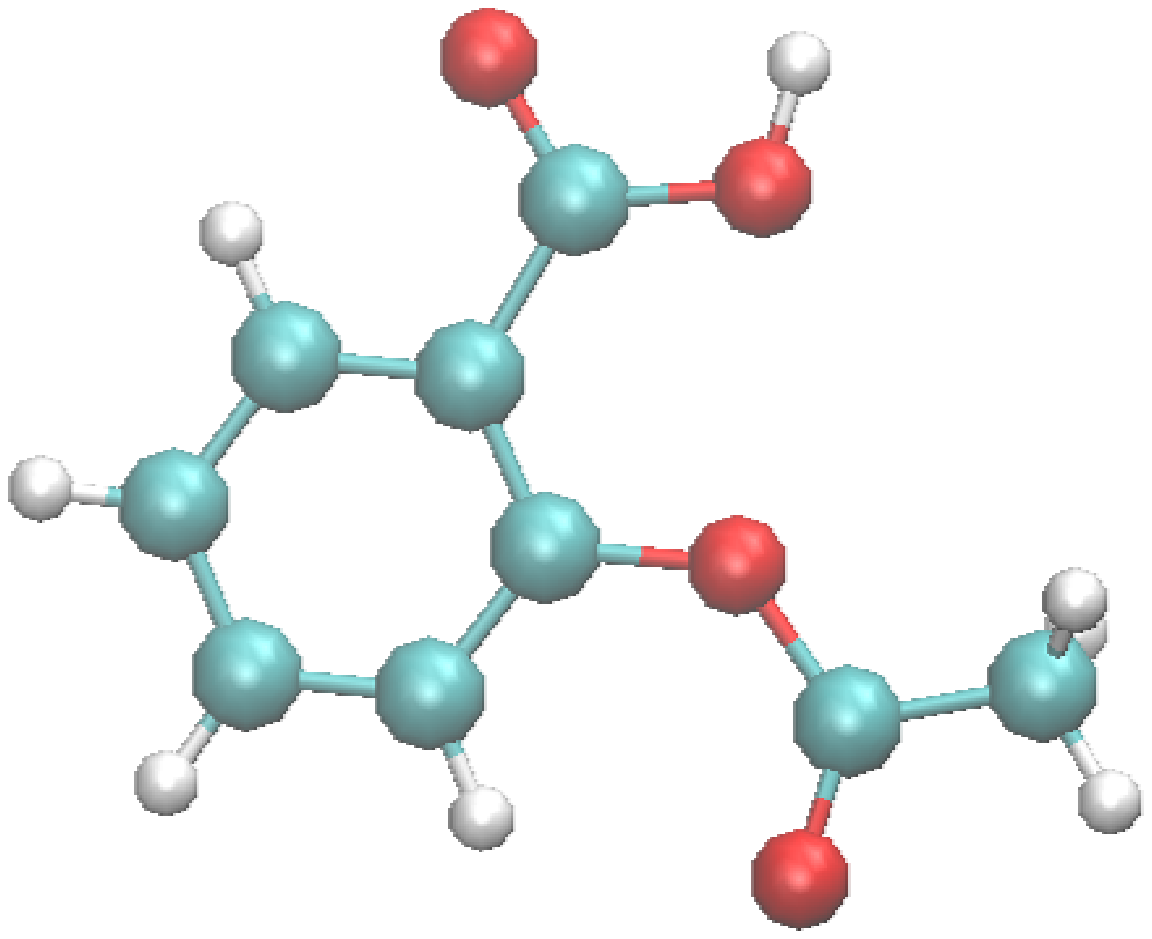}
    \hspace{0.2cm}
    \includegraphics[width=3.6cm]{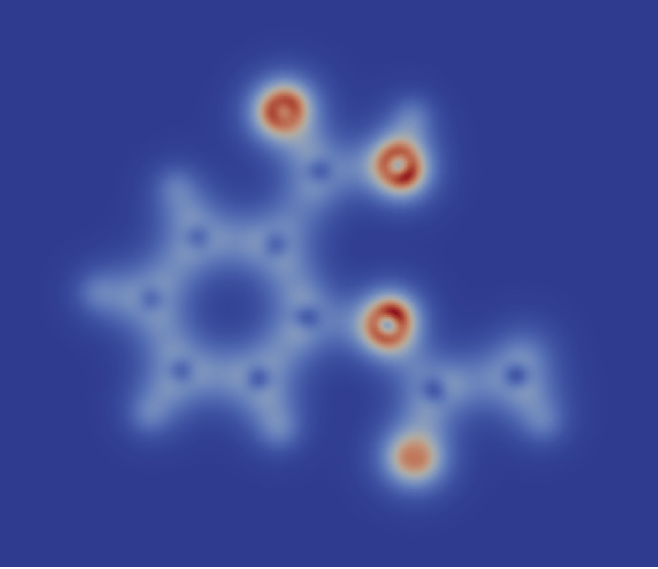}
    \hspace{0.2cm}
    \includegraphics[width=3.6cm]{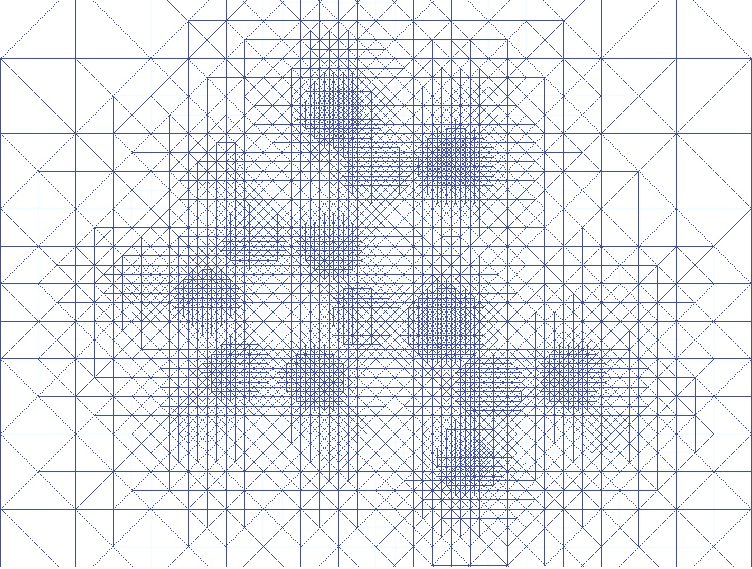}
    \caption{$C_9H_8O_4$: configuration, charge density and mesh on plane $z=0$.}
    \label{fig:C9H8O4}
  \end{center}
\end{figure}

We now turn to analyze some quantitative behavior of our calculations. The convergence curve of
the ground state energy is shown in the left of Figure \ref{figure-error-C9H8O4}.
We observe that the ground state energy approximations converge to $-119.918 ~a.u.$,
which is very close to the value given by SIESTA. This result validates our calculations quantitatively.
We see from the right of Figure \ref{figure-error-C9H8O4} that the convergence curve of the a posteriori error estimator
is parallel to the line with slope $-\frac{2}{3}$, which means that it reaches the optimal convergence rate.
%which also means  that the approximation of eigenvalue as well as the eigenfunction space have reached the optimal convergence rate,
%which coincide with our theory in Section 4.
From the analysis result for the a posteriori error estimator(Theorem 4.3)
%and the relationship between eigenvalues and eigenfunctions,
the optimal convergence of the a posteriori error estimator also indicates that the approximation of %eigenvalues as well as
the eigenfunction space have reached the optimal convergence rate, which coincides with our theory in Section 4.

\begin{figure}[htb]
  \centering
  \subfloat[ground state energy]{
    \includegraphics[width=4.3cm,angle=-90]{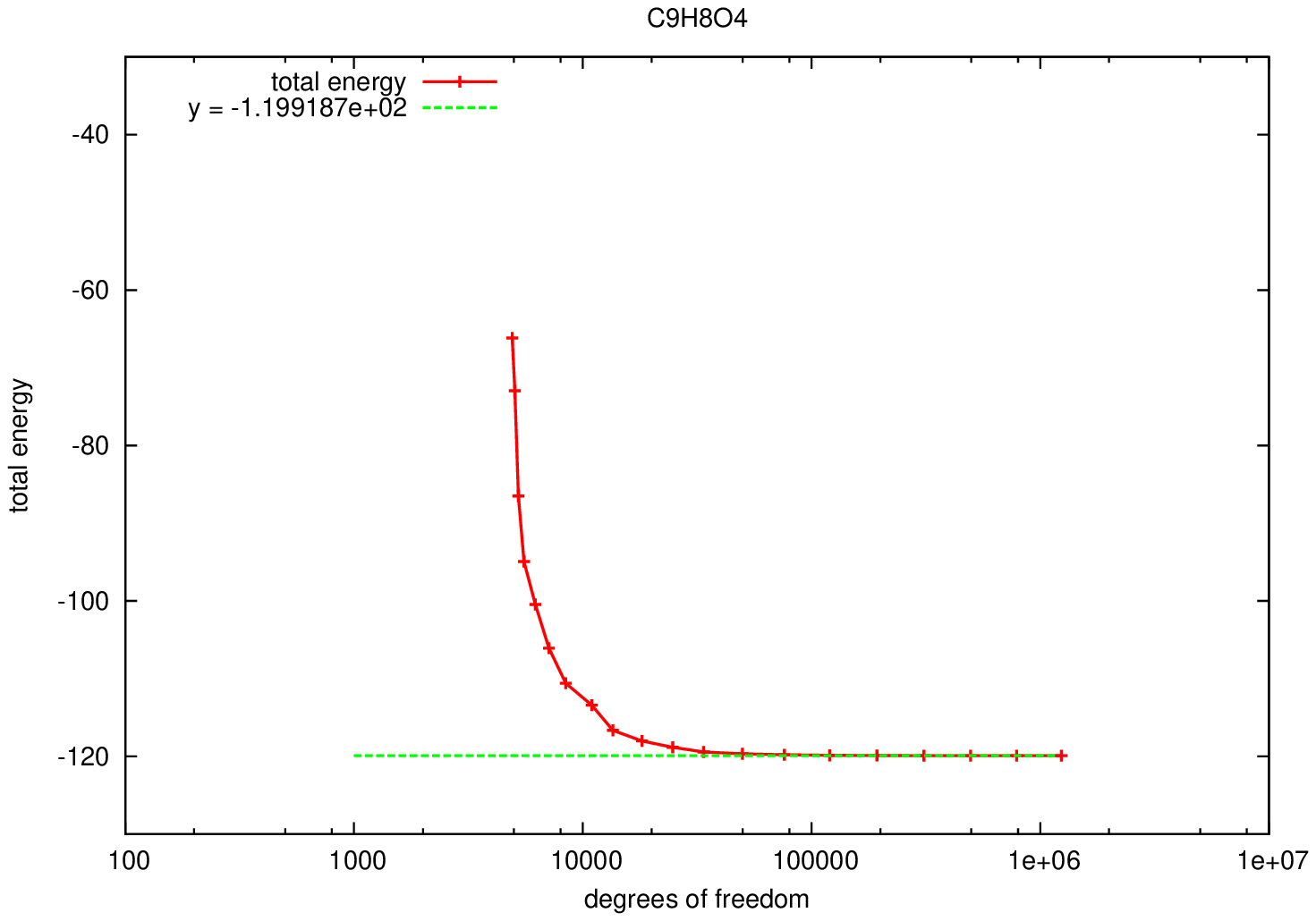}
  }
  \,
  \subfloat[$\eta_h(\Psi_h, \Omega)$]{
    \includegraphics[width=4.3cm,angle=-90]{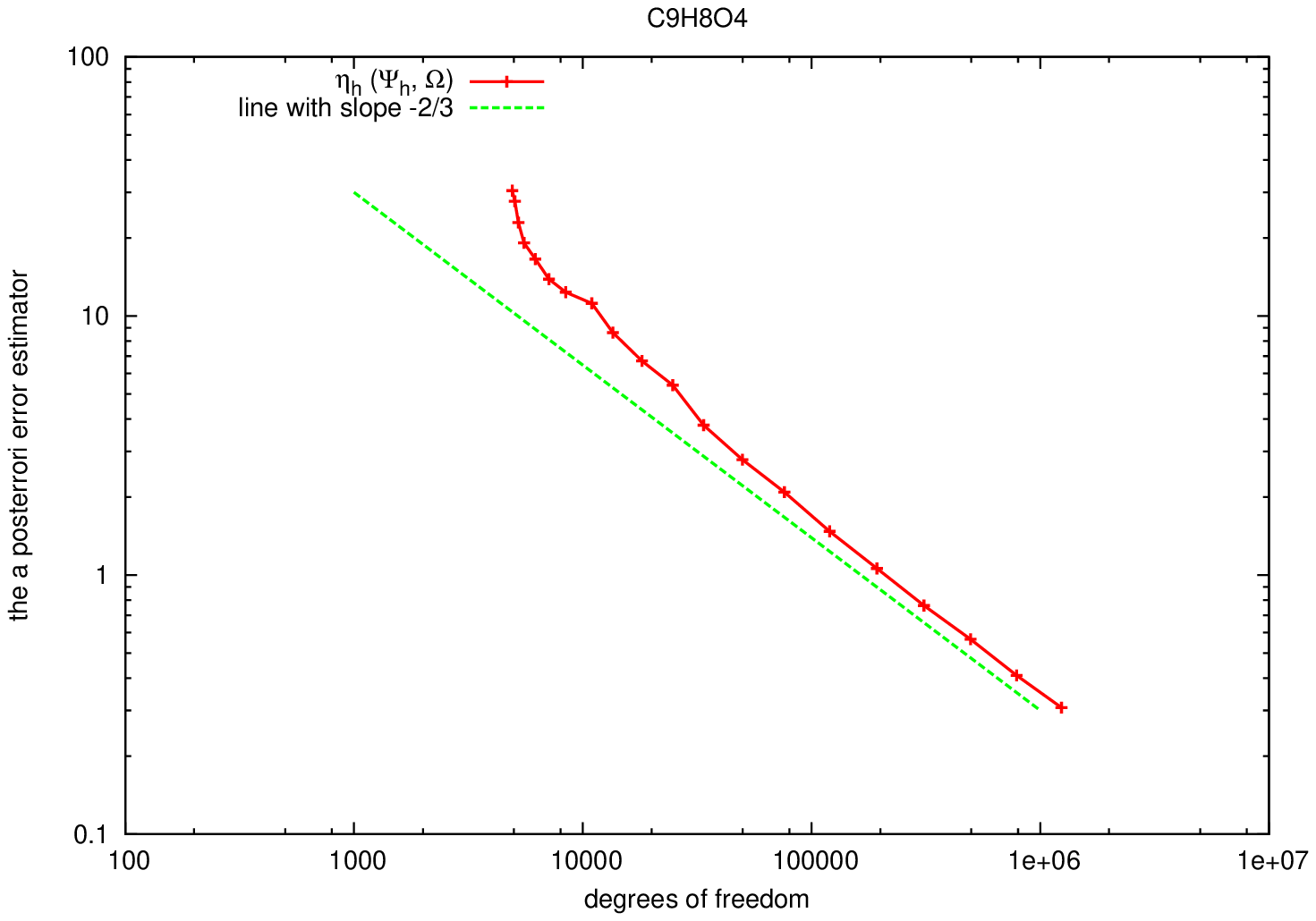}
  }
\caption{The convergence curves of the ground state energy and $\eta_h(\Psi_h,\Omega)$.}
  \label{figure-error-C9H8O4}
\end{figure}\vskip 0.2cm

{\bf Example 2}: $\alpha$ amino acid $C_5H_9O_2 N$.

The ground state energy obtained by SIESTA is $-75.494 ~a.u.$.
In our computations, we choose the computational domain to be $\Omega =[-10.0,10.0]^3$.

The atomic configuration, the calculated ground state charge density and the associated computational mesh are shown in
Figure \ref{fig:C5H9O2N}. We have to point out that for $C_5H_9O_2 N$, not more than $2$ atoms stay in the same plane. Therefore, it is very difficult to find a plane  where  the configuration and the charge density coincide very well
with each other as  Example 1. Similar to Example 1, we also choose the plane $z=0$ as our viewpoint. Anyway, we can see from
the figure for charge density and the figure for the adaptive mesh that our  error indicator is very efficient.
These results can validate our computations.
%correctness of our calculations and the efficiency of our error indicator qualitatively.

\begin{figure}[!htbp]
  \begin{center}
    \includegraphics[width=3.6cm]{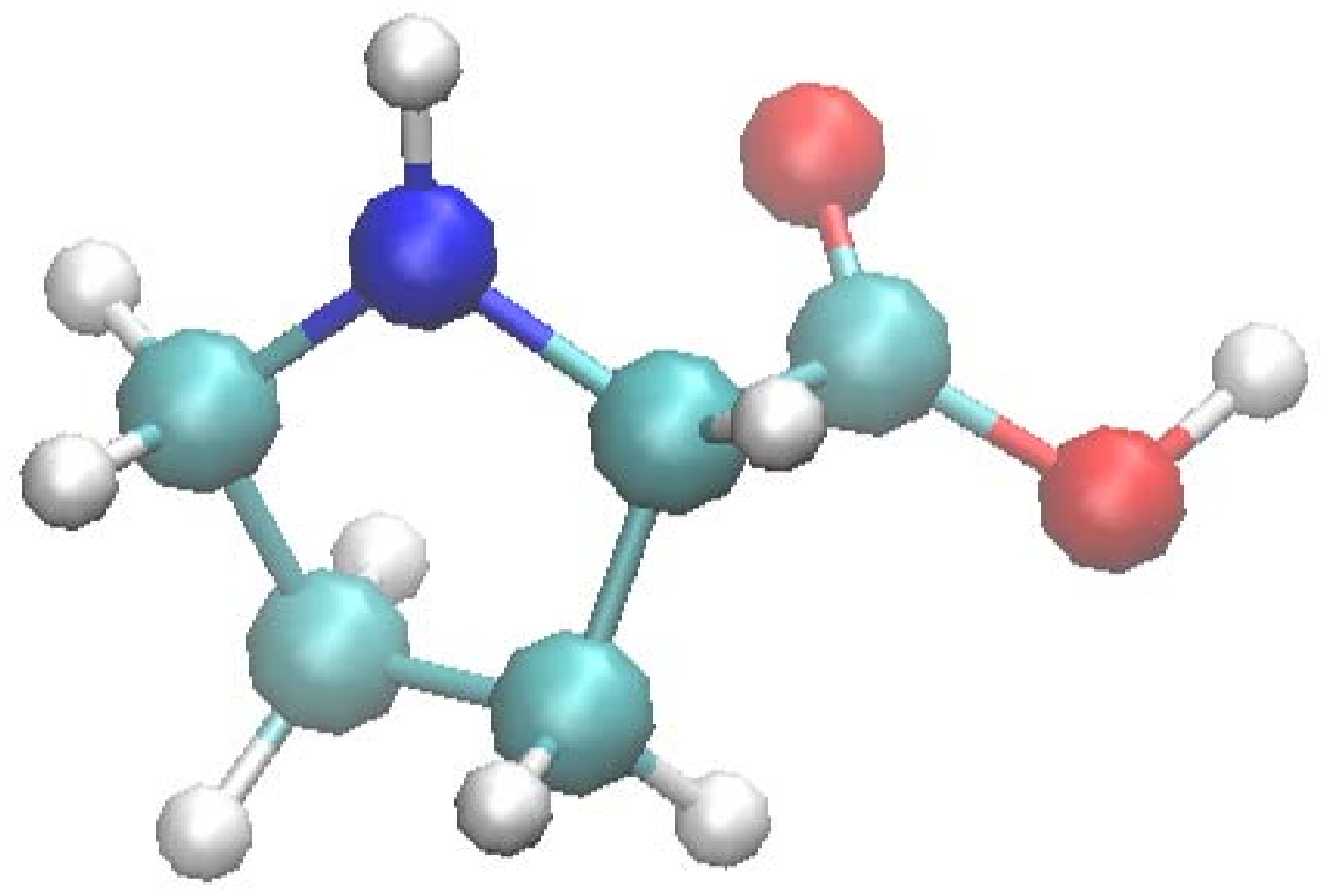}
    \hspace{0.2cm}
    \includegraphics[width=3.6cm]{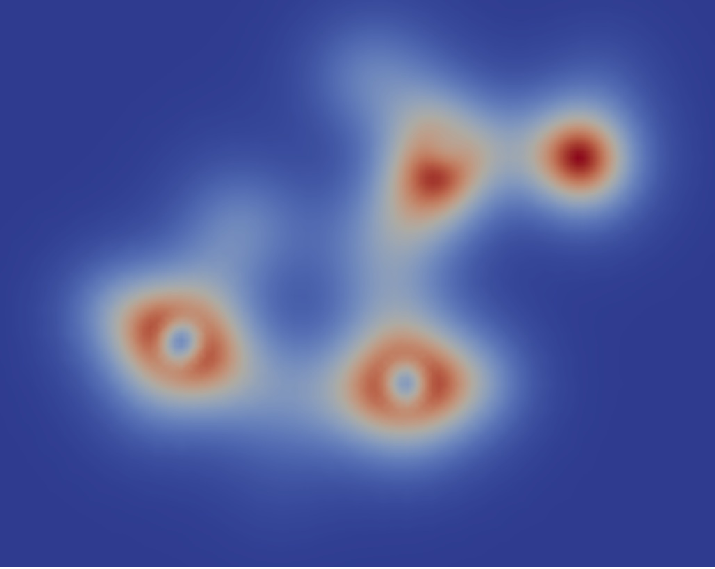}
    \hspace{0.2cm}
    \includegraphics[width=3.6cm]{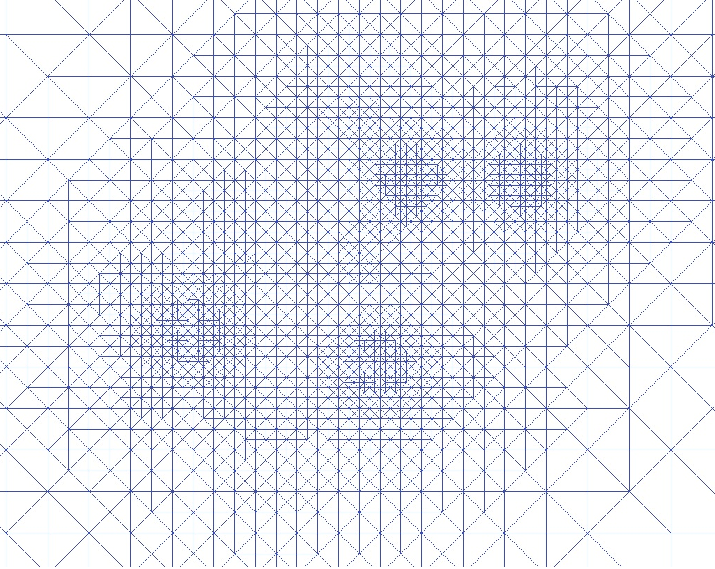}
    \caption{$C_5H_9O_2 N$: configuration, charge density and mesh on plane $z=0$.}
    \label{fig:C5H9O2N}
  \end{center}
\end{figure}

The convergence curves of the ground state energy and the a posteriori error estimator  $\eta_k(\Psi_k,\Omega)$
  obtained by the quadratic finite elements are shown in Figure \ref{figure-error-C5H9O2N},
from which we observe that the ground state energy approximations converge to $-75.494 ~a.u.$,
and the a posteriori error estimator decays with a rate $-\frac{2}{3}$.
This implies the similar conclusions as those for Example 1.

\begin{figure}[htb]
  \centering
  \subfloat[ ground state energy]{
    \includegraphics[width=4.3cm,angle=-90]{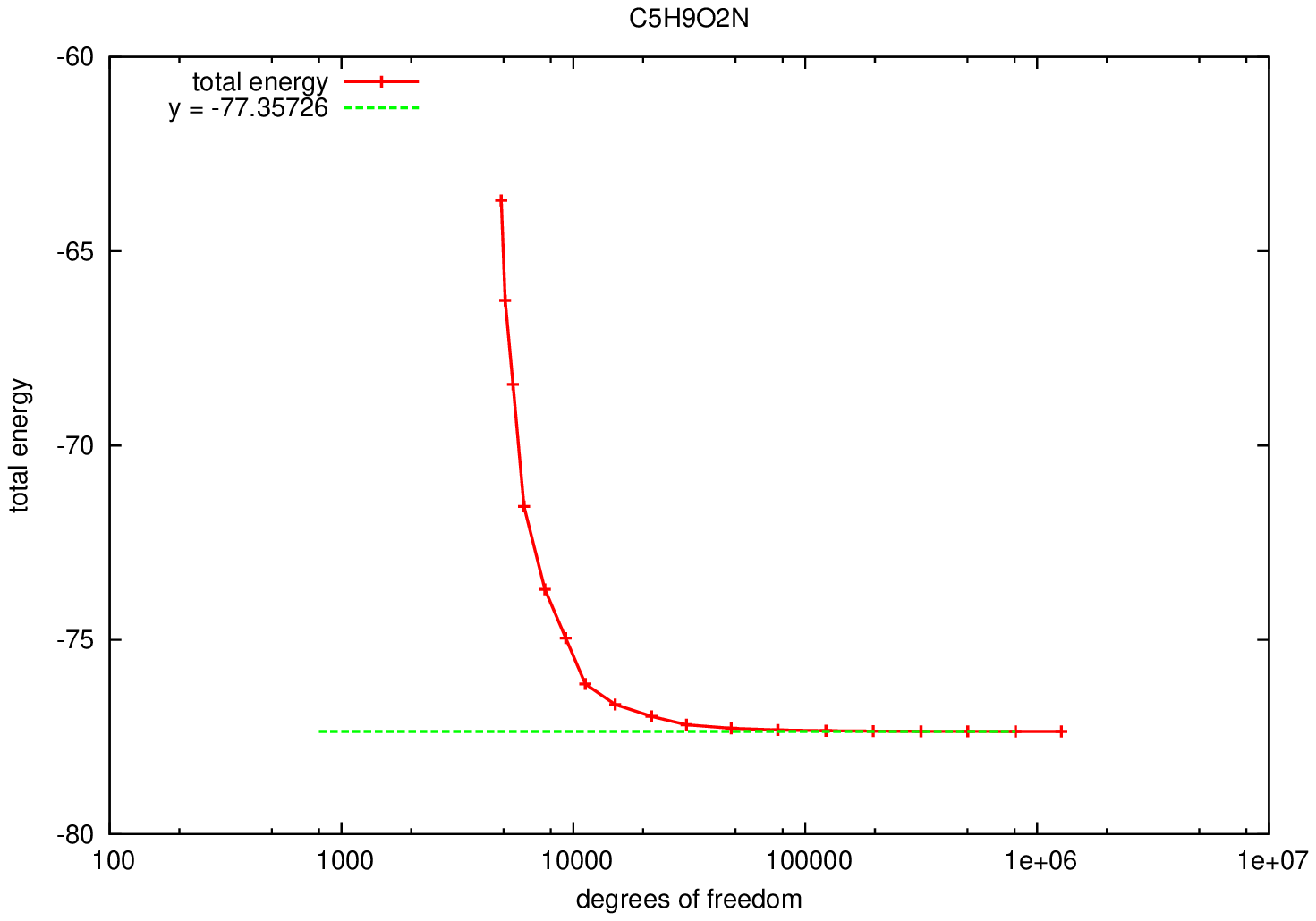}
  }
  \,
  \subfloat[ $\eta_h(\Psi_h, \Omega)$]{
    \includegraphics[width=4.3cm,angle=-90]{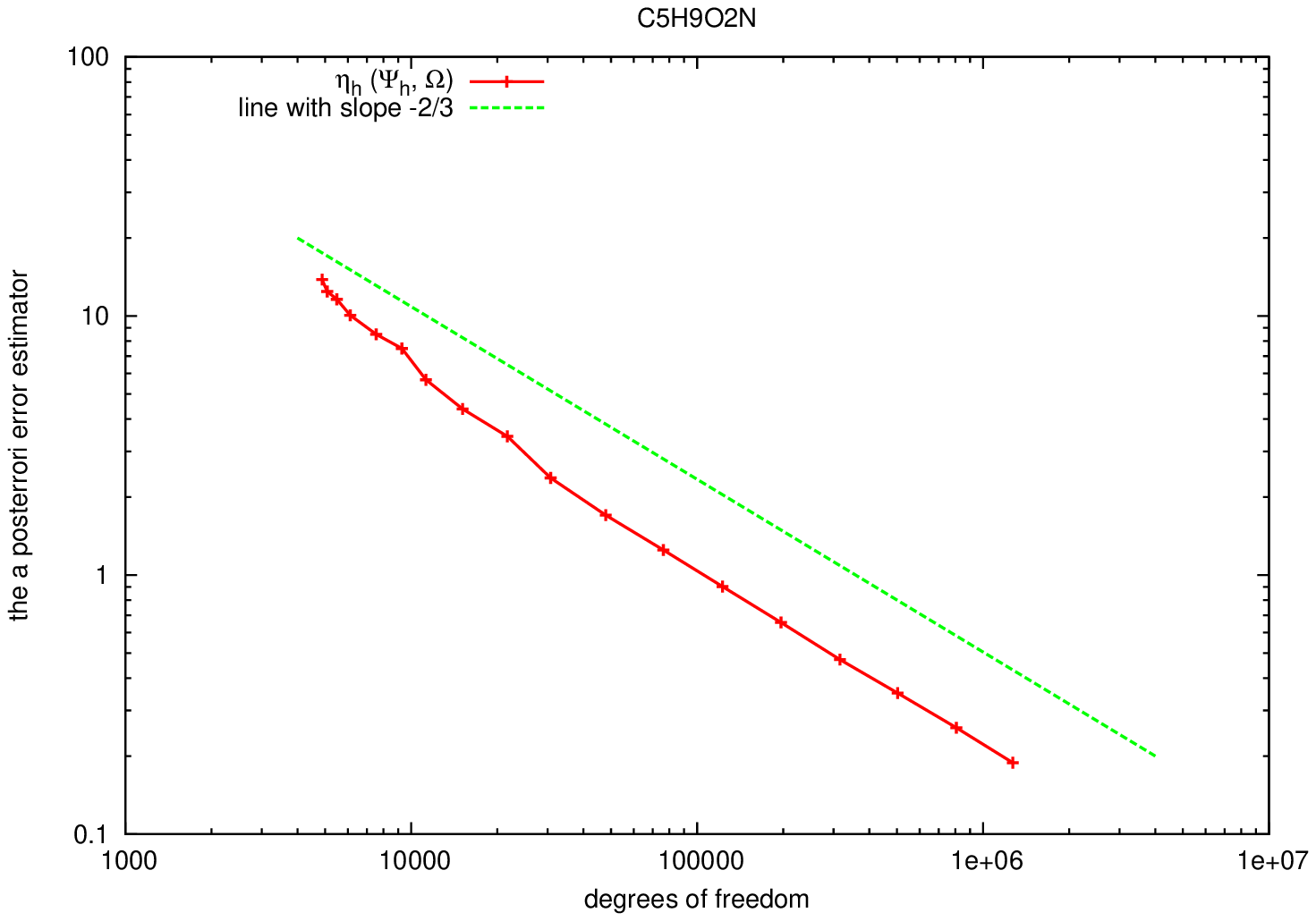}
  }
\caption{The convergence curves of the ground state energy and $\eta_h(\Psi_h,\Omega)$.}
  \label{figure-error-C5H9O2N}
\end{figure}
\vskip 0.2cm

{\bf Example 3}: Fullerene  $C_{60}$.

The ground state energy obtained by SIESTA is $-341.340 ~a.u.$. In
our computations, we choose  $\Omega =[-30.0, 16.0] \times [-23.0,
22.0] \times [-24.0, 21.0]$ to be the computational domain.

We can see the preservation of carbon-hydrogen bonds in Figure
\ref{fig:C60-illustration}, which validates our calculations.
Figure \ref{fig:C60-illustration}, Figure \ref{fig:C60-3d} and Figure \ref{fig:C60-2d} show that more mesh points are placed around the atoms.

\begin{figure}[htbp]
  \centering
  \includegraphics[width=4.5cm]{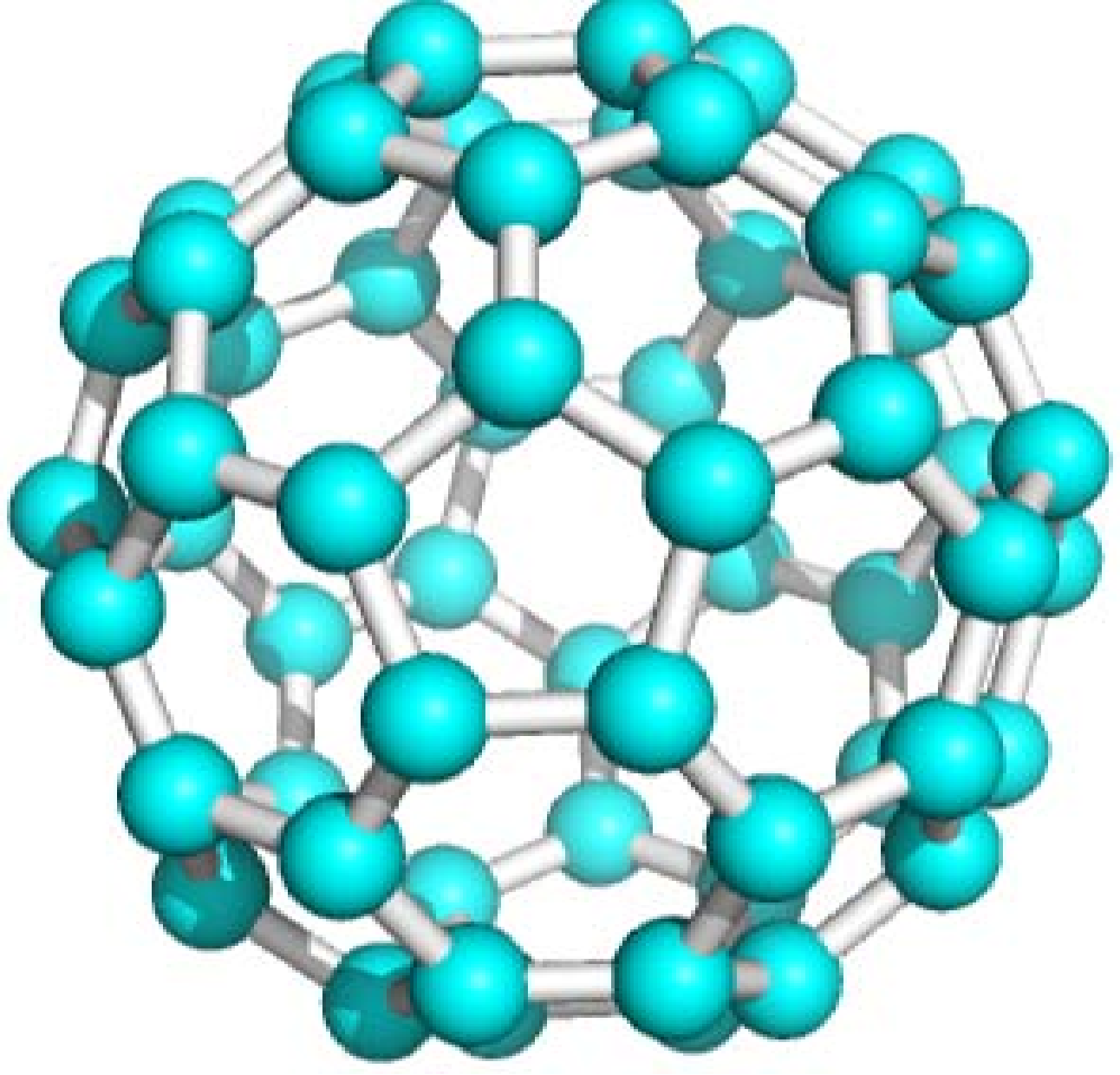}
  \hspace{0.2cm}
  \includegraphics[width=4.0cm]{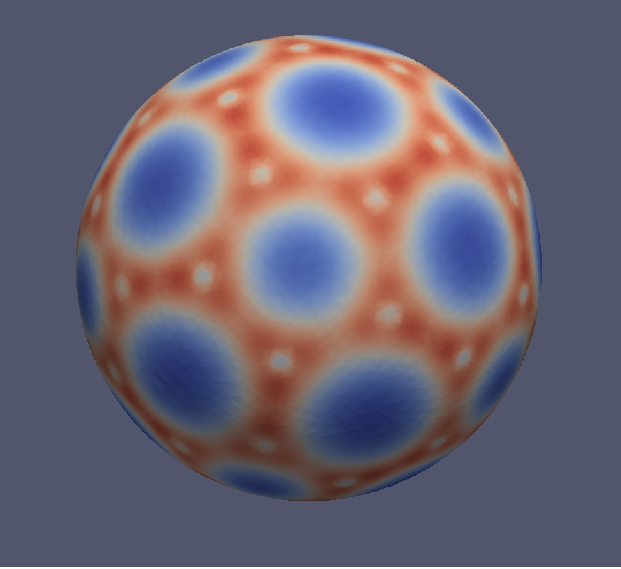}
  %\hspace{0.2cm}
%  \includegraphics[width=4.0cm]{figures/C60/fem-p2/density_3d}
 \caption{$C_{60}$: configuration and charge density on a sphere.}
  \label{fig:C60-illustration}
\end{figure}

\begin{figure}[htbp]
  \centering
  \includegraphics[width=4.0cm]{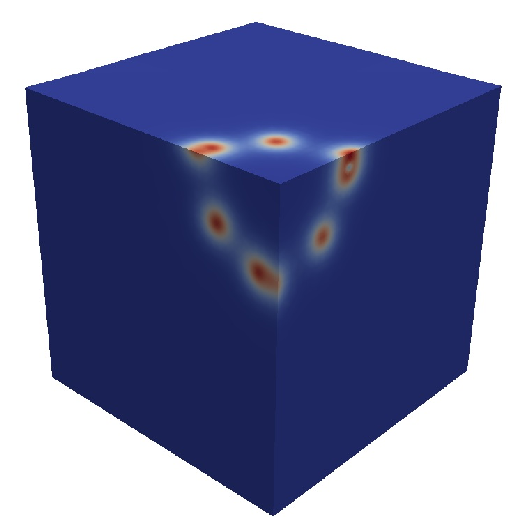}
  \hspace{0.2cm}
  \includegraphics[width=4.0cm]{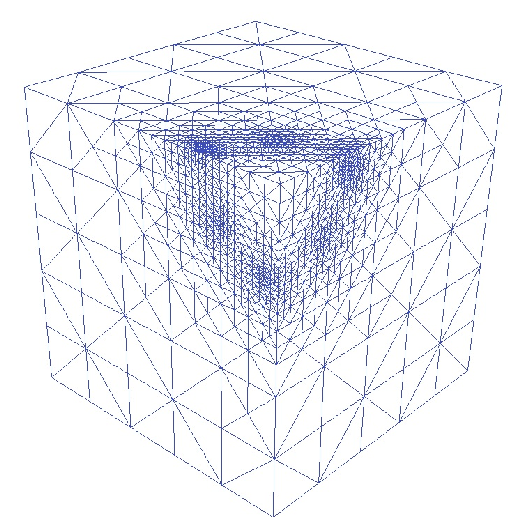}
  \caption{$C_{60}$: charge density and mesh on an interior cross-section.}
  \label{fig:C60-3d}
\end{figure}

\begin{figure}[htbp]
  \centering
  \includegraphics[width=4.0cm]{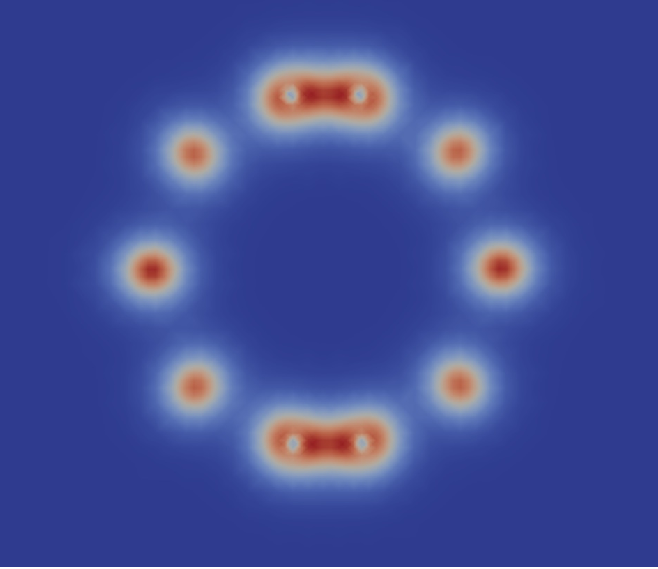}
  \hspace{0.2cm}
  \includegraphics[width=4.0cm]{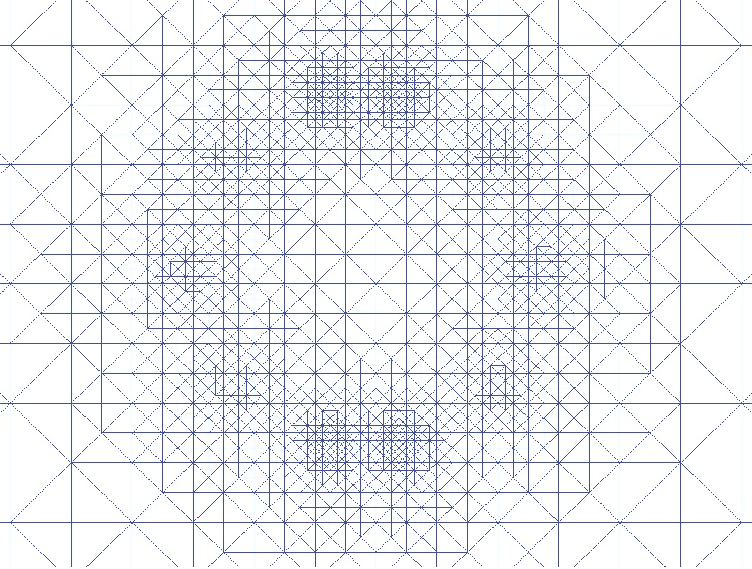}
  \caption{$C_{60}$: charge density and mesh on plane $z=0$.}
  \label{fig:C60-2d}
\end{figure}

The convergence curve of the ground state energy approximations is shown in the right of Figure \ref{figure-error-C60},
from which we observe a convergence to $-342.722~a.u.$, which is very close to the reference energy.
The convergence curve of the a posteriori error estimator obtained by the quadratic  finite element is shown in the left of
Figure \ref{figure-error-C60}, from which we see that it reaches the optimal convergence rate.

\begin{figure}[htb]
  \centering
  \subfloat[ ground state energy]{
    \includegraphics[width=4.3cm,angle=-90]{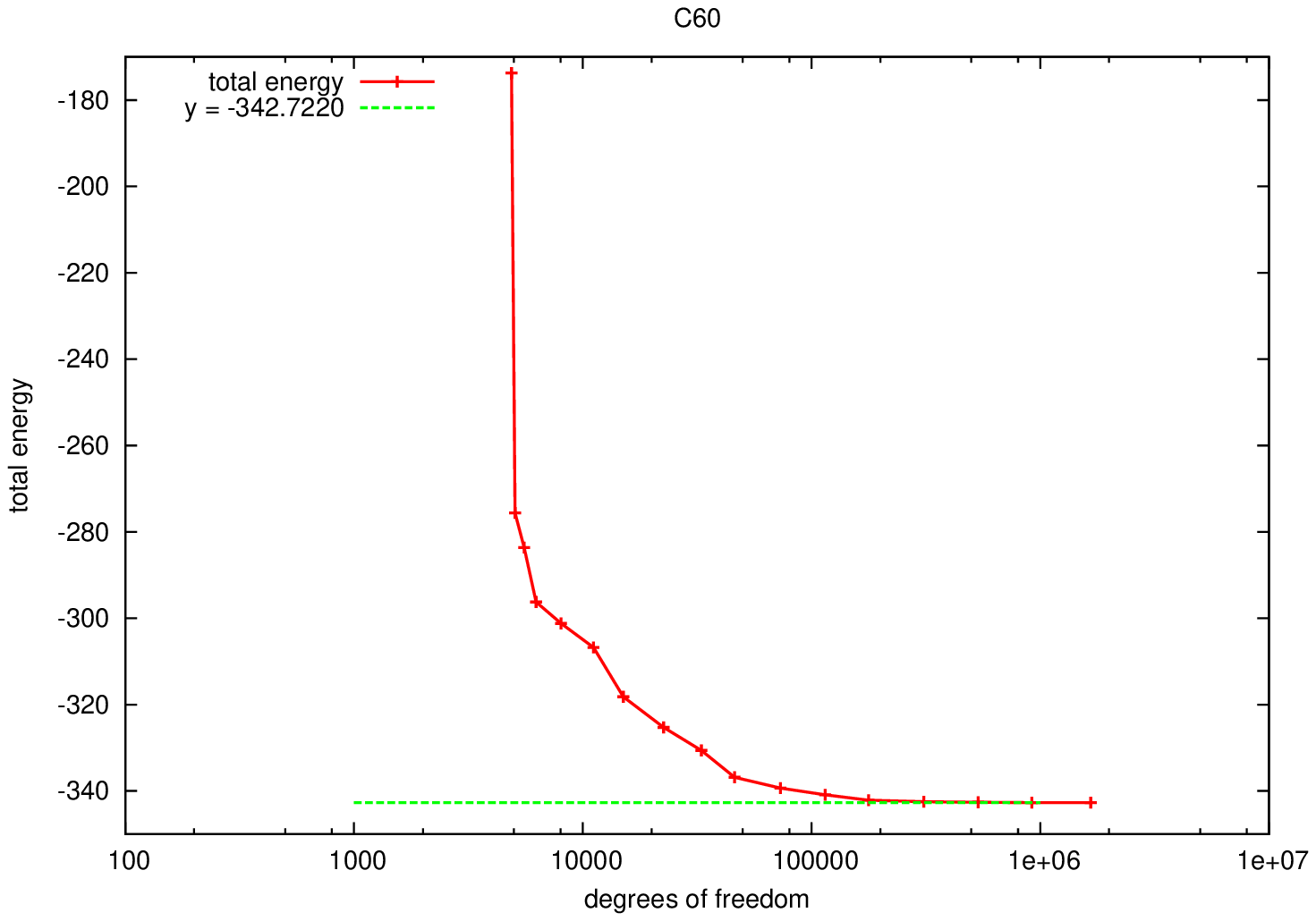}
  }
  \,
  \subfloat[ $\eta_h(\Psi_h, \Omega)$]{
    \includegraphics[width=4.3cm,angle=-90]{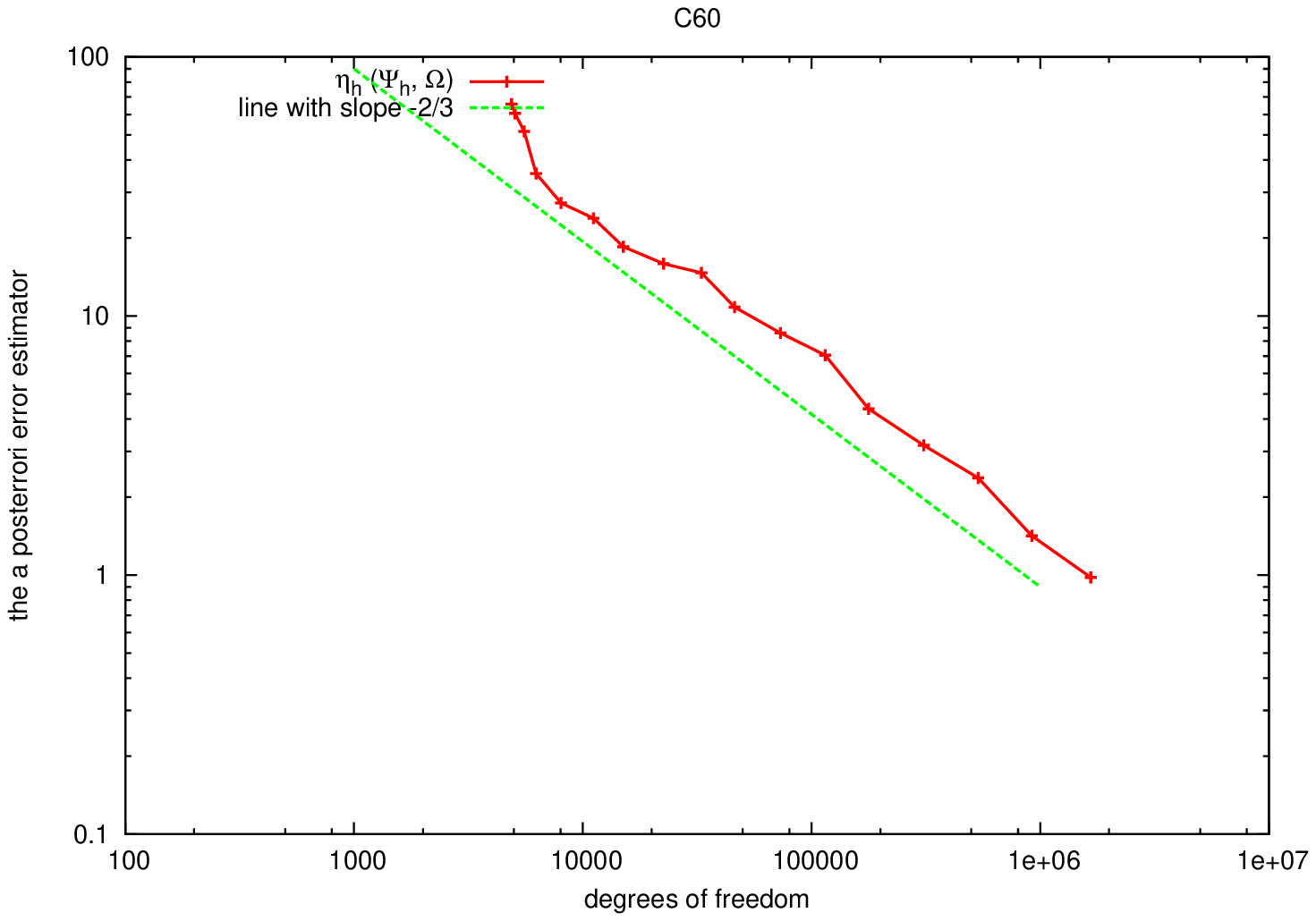}
  }
\caption{The convergence curves of the ground state energy and  $\eta_h(\Psi_h,\Omega)$.}
  \label{figure-error-C60}
\end{figure}

\section{Concluding remarks}\label{sec-conclude}
\setcounter{equation}{0}

In this paper, we have studied the AFE approximations of Kohn-Sham models.
We have obtained the convergence and quasi-optimal complexity of the AFE approximations.
 We have also curried out some typical numerical simulations that  not only support our theory,
but also show the robustness and efficiency of the adaptive finite element method in electronic structure calculations.

In our analysis of convergence rate and complexity of AFE approximations, for convenience, we have assumed that the numerical integration was exact
and the nonlinear algebraic eigenvalue problem was exactly
solved. Indeed, the same conclusion can be expected when the error resulting from the inexact solving of the
nonlinear algebraic
eigenvalue problem and the error coming from the inexact numerical integration are taken into account.

Suppose that $(\Lambda, \Phi)\in\Theta$, the associated exact solution over mesh $\mathcal{T}_h$
is $( \Lambda_h,  \Phi_{h})$, and the inexact numerical solution is
$(\hat{\Lambda}_{h}, \hat{\Phi}_{h})$. If the numerical
errors resulting from the solution of (nonlinear) algebraic system and the numerical integration are small enough, say, satisfy
\begin{eqnarray*}
\| \Phi_{h} - \hat{\Phi}_{h}\|_{1,\Omega}^2 + |  \Lambda_{h}
- \hat{\Lambda}_{h}| \lesssim {\tilde r}(h_0)   \eta_{h}^2(\hat{\Phi}_{h}, \Omega)
\end{eqnarray*}
with ${\tilde r}(h_0) \ll 1$ for $h_0 \ll 1$, then we have from the following triangle inequality
\begin{eqnarray*}
\|\Phi - \hat{\Phi}_{h}\|_{1,\Omega} &\leq& \|\Phi -  \Phi_{h}\|_{1,\Omega} + \| \Phi_{h} - \hat{\Phi}_{h}\|_{1,\Omega}, \nonumber\\
|\Lambda - \hat{\Lambda}_{h}| &\leq& |\Lambda -  \Lambda_{h}| + | \Lambda_{h} - \hat{\Lambda}_{h}|,
\end{eqnarray*}
and the similar perturbation arguments that the same convergence rate and quasi-optimal complexity can be derived.

%We may see from \cite{he-zhou11} that the similar conclusions on convergence rate and complexity of AFE approximations can be obtained for the all-electron calculations.
%It will be our further work to incorporate the generalized gradient approximation (GGA)  \cite{martin04,parr-yang94}, and even
%more complicate exchange-correlation approximations such as B3LYP \cite{becke98,LYP88} and PBE \cite{PBE} into our numerical analysis.

Finally, we point out that, in this paper, we have not given the convergence rate and complexity for the AFE approximations for the Lagrange multipliers $\Lambda$. Indeed, the related optimal results for Lagrange multipliers are not so obvious, and we need do some more detailed  analysis, which  increase the length of this paper. We will report elsewhere.

\section*{Appendix: A boundary value problem} \label{sec-model-problem}
\setcounter{equation}{0}
\renewcommand{\theequation}{A.\arabic{equation}}
\renewcommand{\thetheorem}{A.\arabic{theorem}}
\renewcommand{\thelemma}{A.\arabic{lemma}}
\renewcommand{\theproposition}{A.\arabic{proposition}}
\renewcommand{\thealgorithm}{A.\arabic{algorithm}}
\renewcommand{\thesection}{A}

In this appendix, we shall provide some basic results for the AFE approximations
of a model problem that was used in our previous analysis.
Consider a homogeneous boundary value problem:
\begin{equation}\label{model-problem}
\left\{\begin{array}{rcll}
\mathcal{L}\Phi&=&F &\mbox{in}~~\Omega,\\ \Phi&=& 0 &\mbox{on}~~\partial\Omega,
\end{array}\right.
\end{equation}
where $F=(f_i)_{i=1}^N\in (L^2(\Omega))^N$.
Note that (\ref{model-problem}) is equal to: Find $\Phi\in \mathcal{H}$ such that
\begin{eqnarray} \label{model-weak}
a(\Phi, \Gamma)= (F,\Gamma) \qquad\forall~\Gamma\in \mathcal{H}.
\end{eqnarray}
A standard finite element scheme for (\ref{model-weak}) is: Find
$\Phi_h \in V_h$ satisfying
\begin{eqnarray}\label{model-weak-fem}
a(\Phi_h, \Gamma) =(F, \Gamma) \qquad \forall~\Gamma \in V_h.
\end{eqnarray}

Let $\mathbb{T}$ denote the class of all conforming refinements by bisections of $\mathcal{T}_0$.
For $\mathcal{T}_h \in \mathbb{T}$ and any
$\Gamma=(\gamma_i)_{i=1}^N\in V_h$, we define the element residual
$\tilde{\mathcal{R}}_{\tau}(\Gamma)$ and the jump $\tilde{J}_e(\Gamma)$ by
 \begin{eqnarray}\label{residual-modelproblem}
  \tilde{\mathcal{R}}_{\tau}(\Gamma) &=&
  \left(f_i +\frac{1}{2} \Delta \gamma_i\right)_{i=1}^N\qquad \mbox{in}~ \tau\in
  \mathcal{T}_h,
\end{eqnarray}
\vskip -0.6cm
\begin{eqnarray*}
\tilde{J}_e(\Gamma) &=&\left(\frac{1}{2}\nabla \gamma_i|_{\tau_1}\cdot\overrightarrow{n_1}+
\frac{1}{2}\nabla\gamma_i|_{\tau_2}\cdot\overrightarrow{n_2}\right)_{i=1}^N
% =\left(\frac{1}{2} [[\nabla \gamma_i]]_e \cdot \overrightarrow{n_1} \right)_{i=1}^N
\quad\mbox{on}~ e\in  \mathcal{E}_h,
\end{eqnarray*}
where $e$ is the  common face of elements $\tau_1$ and $\tau_2$ with
unit outward normals $\overrightarrow{n_1}$ and
$\overrightarrow{n_2}$, respectively.
For $\tau\in \mathcal{T}_h$, we define the local error indicator
$\tilde{\eta}_h(\Gamma, \tau)$ by
\begin{eqnarray}\label{error-indicator}
\tilde{\eta}^2_h(\Gamma, \tau) = h_\tau^2\|\tilde{\mathcal{R}}_{\tau}(\Gamma)\|_{0,\tau}^2
+ \sum_{e\in \mathcal{E}_h,e\subset\partial \tau} h_e \|\tilde{J}_e(\Gamma)\|_{0,e}^2
\end{eqnarray}
and the oscillation $\widetilde{{\rm osc}}_h(\Gamma,\tau)$ by
\begin{eqnarray}\label{local-oscillation}
\widetilde{{\rm osc}}_h(\Gamma,\tau) =
h_\tau\|\tilde{\mathcal{R}}_{\tau}(\Gamma)-\overline{\tilde{\mathcal{R}}_{\tau}(\Gamma)}\|_{0,\tau}.
\end{eqnarray}
%where $\overline{w}$ is the $L^2$-projection of $w\in L^2(\Omega)$
%to polynomials of some degree  on $\tau$ or $e$.
%
Given  $\mathcal {T}' \subset \mathcal {T}_h$, we define the error estimator
$\tilde{\eta}_h(\Gamma,\mathcal {T}')$ and the oscillation  $\widetilde{{\rm osc}}_h(\Gamma,\mathcal {T}')$ by
\begin{eqnarray*}
\tilde{\eta}^2_h(\Gamma, \mathcal{T}') = \sum_{\tau\in \mathcal{T}'}
\tilde{\eta}^2_h(\Gamma, \tau) \quad \textnormal{and} \quad
\widetilde{{\rm osc}}^2_h(\Gamma, \mathcal{T}') = \sum_{\tau\in
\mathcal{T}'}  \widetilde{{\rm osc}}^2_h(\Gamma, \tau),
\end{eqnarray*}
respectively.
We see that a similar a posteriori error estimate to that for Poisson
equation can be expected for (\ref{model-problem}) (c.f. \cite{mekchay-nochetto05,morin-nochetto-siebert02,verfurth96}).
\begin{theorem}\label{upper-bound-theorem}
Let $\Phi \in \mathcal{H}$ be the solution of (\ref{model-weak}) and $\Phi_h \in V_h$ be the solution of (\ref{model-weak-fem}).
Then there exist constants $\tilde{C}_1$, $\tilde{C}_2$ and
$\tilde{C}_3>0$ depending only on  $c_a$ in \eqref{coercive-constant} and $\gamma^{\ast}$ in \eqref{shape-regularity} such that
\begin{eqnarray}\label{boundary-upper}
\|\Phi-\Phi_h \|^2_{1,\Omega}\leq\tilde{C}_1\tilde{\eta}^2_h(\Phi_h,\Omega),
\end{eqnarray}
\vskip -0.6cm
\begin{eqnarray}\label{boundary-lower}
\tilde{C}_2 \tilde{\eta}^2_h (\Phi_h, \Omega)\le
\|\Phi-\Phi_h\|_{1,\Omega}^2+\tilde{C}_3\widetilde{{\rm osc}}^2_h(\Phi_h,\Omega).
\end{eqnarray}
\end{theorem}

An AFE algorithm for \eqref{model-weak} is designed as follows (c.f. \cite{cascon-kreuzer-nochetto-siebert08}):
\vskip 0.1cm

\begin{algorithm}\label{algorithm-AFEM-bvp}~
\begin{enumerate}
\item Pick a given mesh $\mathcal{T}_0$, and let $k=0$.
\item Solve  \eqref{model-weak-fem} on $\mathcal{T}_k$ to get discrete solution $\Phi_k$.
\item Compute local error indictors $\tilde{\eta}_k(\Phi_k,\tau)$ for all $\tau\in \mathcal{T}_k$.
\item Construct $\mathcal{M}_k \subset \mathcal{T}_k$ by D\"{o}rfler Strategy and parameter
 $\theta$.
\item Refine $\mathcal{T}_k$ to get a new conforming mesh $\mathcal{T}_{k+1}$.
\item Let $k=k+1$ and go to 2.
\end{enumerate}
\end{algorithm}
\vskip 0.1cm

Using the similar arguments to those for scalar linear elliptic boundary
value problem (see, e.g, \cite{cascon-kreuzer-nochetto-siebert08}),
we have the following result for Algorithm \ref{algorithm-AFEM-bvp}.

\begin{theorem}\label{convergence-boundary}
If $\{\Phi_k\}_{k\in \mathbb{N}_0}$ is a sequence of finite element solutions
% corresponding to a sequence of nested finite element spaces $\{V_k\}_{k\in \mathbb{N}_0}$
produced by Algorithm \ref{algorithm-AFEM-bvp}, then there exist constants $\tilde{\gamma}>0$ and  $\tilde{\xi}\in (0,1)$
depending only on the shape regularity $\gamma^{\ast}$ and the marking parameter $\theta$,
such that for any two consecutive iterations
\begin{eqnarray*}\label{convergence-neq}
&&\|\Phi-\Phi_{k+1}\|^2_{1,\Omega} + \tilde{\gamma} \tilde{\eta}^2_{k+1}(\Phi_{k+1},\Omega)
~\leq~ \tilde{\xi}^2 \big( \|\Phi -\Phi_k\|^2_{1,\Omega} +
\tilde{\gamma} \tilde{\eta}^2_k(\Phi_k, \Omega)\big).
\end{eqnarray*}
Indeed, the constant ${\tilde\gamma}$ has the following form
\begin{eqnarray}\label{gamma-boundary}
\tilde{\gamma} = \frac{1}{(1 + \delta^{-1})\tilde{C}^2_{\ast}}
\end{eqnarray}
with $\tilde{C}_{\ast}>0$ depending on the regularity constant $\gamma^{\ast}$ and $\delta\in (0,1)$.
\end{theorem}

For the distance between two nested solutions of (\ref{model-weak-fem}), we have (c.f. \cite{cascon-kreuzer-nochetto-siebert08})
\begin{theorem}\label{localized-upper-bound}
Let $\Phi_{H} \in V_H$ and $ \Phi_{h} \in V_h$ be solutions of (\ref{model-weak-fem}) respectively.
If $\mathcal{T}_{h}$ is a refinement of $\mathcal{T}_{H}$ by marked element $\mathcal{M}_H$ and refined elements $\mathcal{R}=\mathcal{R}_{\mathcal{T}_H\rightarrow \mathcal{T}_h}$, then
\begin{eqnarray*}\label{localized-upper-bound-conc}
\|\Phi_{H} -\Phi_{h} \|^2_{1,\Omega} \leq \tilde{C}_1  \sum_{\tau \in \mathcal{R}} \tilde{\eta}^2_{H}(\Phi_{H},\tau).
\end{eqnarray*}
\end{theorem}


\small
\begin{thebibliography}{99}
\bibitem{adams75} R.A. Adams, {\em Sobolev Spaces}, Academic Press, New York,
1975.

\bibitem{agmon81} S. Agmon, {\em Lectures on the Exponential Decay
of Solutions of Second-Order Elliptic Operators}, Princeton
University Press, Princeton, 1981.

\bibitem{anantharaman09} A. Anantharaman and E. Canc\`{e}s, {\em
Existence of minimizers for Kohn-Sham models in quantum chemistry},
Ann. I. H. Poincar\'{e}-AN, {\bf 26} (2009), pp.~2425-2455.

%\bibitem{arnold-mukherjee-pouly00} D. Arnold, A. Mukherjee, and L.
%Pouly, {\em Locally adapted tetrahedral meshes using bisection},
%SIAM J. Sci. Comput., {\bf 22} (2000), pp.~431-448.

\bibitem{babuska-vogelius84} I. Babuska and M. Vogelius,
{\em Feedback and adaptive finite element solution of one-dimensional boundary value problems},
Numer. Math., {\bf 44} (1984), pp.~75-102.

\bibitem{bansch-siebert95} E. B\"{a}nsch and K. Siebert,
{\em A Posteriori Error Estimation for Nonlinear
Problems by Duality Techniques}, Albert-Ludwigs-Univ., Math. Fak.,
1995.

\bibitem{beck00} T.L. Beck, {\em Real-space mesh techniques
in density-function theory}, Rev. Mod. Phys., {\bf 72} (2000),
pp.~1041-1080.

\bibitem{becke98} A.D. Becke, {\em A new inhomogeneity parameter in density-functional theory},
J. Phys. Chem., {\bf 109} (1998), pp. 2092-2098.

\bibitem{becker-rannacher01} R. Becker  and R. Rannacher,
{\em An optimal control approach to a posteriori error estimation in
finite element methods}, Acta Numerica, {\bf 10} (2001), pp. 1-102.

%\bibitem{bachelet-hamann}G.B. Bechelet, D.R. Hamann,
%and M. Schl{\" u}ter, {\em Pseudopotentials that work: From H to
%Pu,}  Phys. Rev. B, {\bf 26} (1982), pp.~4199-4228.

\bibitem{binev-dehmen-devore04}P. Binev,  W. Dahmen,  and R. DeVore,
{\em Adaptive finite element methods with convergence rates},
Numer. Math., {\bf 97} (2004), pp. 219-268.

\bibitem{bylaska09} E.J. Bylaska, M. Holst, and J.H. Weare,
{\em Adaptive finite element method for solving the exact Kohn-Sham equation
of density functional theory}, J. Chem. Theory Comput., {\bf 5} (2009), pp 937-948.

\bibitem{cancesM10} E. Canc\`{e}s, R. Chakir, and Y. Maday, {\em
Numerical analysis of the planewave discretization of some
orbital-free and Kohn-Sham models},  M2AN, {\bf 46} (2012), pp. 341-388.

%\bibitem{carstensen09} C. Carstensen,
%{\em Convergence of adaptive finite element methods in computational
%mechanics}, Appl. Numer. Anal., {\bf 59} (2009),pp. 2119-2130.

\bibitem{cascon-kreuzer-nochetto-siebert08}
J.M. Cascon,   C. Kreuzer, R.H. Nochetto,  and  K.G. Siebert,
{\em Quasi-optimal convergence rate for an adaptive finite element method},
SIAM J. Numer. Anal., {\bf 46} (2008), pp. 2524-2550.

%\bibitem{chen-thesis} {\em ....}

\bibitem{chen-gong-he-yang-zhou10} H. Chen, X. Gong, L. He, Z. Yang, and A. Zhou,
{\em Numerical analysis of finite dimensional approximations of
Kohn-Sham equations}, Adv. Comput. Math., {\bf 38} (2013), pp. 225-256.

\bibitem{chen-he-zhou09} H. Chen, X. Gong, L. He, and A. Zhou,
{\em Adaptive finite element approximations for a class of nonlinear eigenvalue
problems in quantum physics}, Adv., Appl., Math., Mech., {\bf 3} (2011), pp. 493-518.

\bibitem{chen-he-zhou10} H. Chen, L. He, and A. Zhou, {\em Finite
element approximations of nonlinear eigenvalue problems in quantum
physics}, Comput. Methods Appl. Mech. Engrg., {\bf 200} (2011), pp.
1846-1865.

%\bibitem{chen-holst-xu08} L. Chen, M. J. Holst, and J. Xu,
%{\em The finite element approximation of the nonlinear
%Poisson-Boltzmann equation}, SIAM J. Numer. Anal., {\bf 45}(2008),
%pp. 2298-2320.

\bibitem{ciarlet78} P.G. Ciarlet, {\em The Finite Element Method
for Elliptic Problems}, North-Holland, 1978.

\bibitem{dai-thesis08} X. Dai, {\em Adaptive and Localization Based Finite Element
Discretizations for the First-Principles Electronic
Structure Calculations}, Ph.D. Thesis, Academy
of Mathematics and Systems Science, Chinese Academy of Sciences, Beijing, 2008.

\bibitem{dai-gong-yang-zhang-zhou11} X. Dai, X. Gong, Z. Yang, D, Zhang, and A. Zhou,
{\em Finite volume discretizations for eigenvalue problems with applications to electronic structure calculations},
Multiscale Model. Simul.,  {\bf 9} (2011), pp. 208-240.

\bibitem{dai-he-zhou12} X. Dai, L. He, and A. Zhou,
{\em Convergence rate and quasi-optimal complexity of adaptive finite element computations for multiple eigenvalues},  	 arXiv:1210.1846, 2012.

\bibitem{dai-shen-zhou08} X, Dai, L. Shen, and A. Zhou,
{\em A local computational scheme for higher order finite element
eigenvalue approximations},  Inter. J. Numer. Anal. Model., {\bf 5} (2008), pp. 570-589.

\bibitem{dai-xu-zhou08} X. Dai, J. Xu, and A. Zhou,
{\em Convergence and optimal complexity of adaptive finite element eigenvalue computations},
Numer. Math., {\bf 110} (2008), pp. 313-355.

\bibitem{dai-zhou08}  X. Dai and A. Zhou,
{\em Three-scale finite element discretizations for quantum eigenvalue problems},
SIAM J. Numer. Anal., {\bf 46} (2008), pp. 295-324.

%\bibitem{diening-kreuzer08} L. Diening and C. Kreuzer, {\em Linear
%convergence of adaptive finite element method for the p-Laplacian
%equation}, SIAM J. Numer. Anal., {\bf 46} (2008), pp. 614-638.

\bibitem{dorfler96} W. D{\" o}rfler,
{\em A convergent adaptive algorithm for
Poisson's equation}, SIAM J. Numer. Anal., {\bf 33} (1996), pp.
1106-1124.

%\bibitem{dreizler90} R.M. Dreizler and E.K.U. Gross,
%{\rm Density functional theory}, Springer Verlag, New York, 1990.

\bibitem{duran-padra-rodriguez03}  R.G. Dur\'{a}n, C. Padra, and
R. Rodr\'{i}guez, {\em A posteriori error estimates for the finite element
approximation of eigenvalue problems}, Math. Mod. Meth. Appl. Sci.,
{\bf 13} (2003), pp. 1219-1229.

\bibitem{edelman98} A. Edelman, T.A. Arias, and S.T. Smith,
{\em The geometry of algorithms with orthogonality constraints},
SIAM J. Matrix Anal. appl., {\bf 20} (1998), pp.~303-353.

\bibitem{fang-gao-zhou12}
J. Fang, X. Gao, and A. Zhou,
{\em A Kohn-Sham equation solver based on hexahedral finite
elements}, J. Comput. Phys., {\bf 231} (2012), pp. 3166-3180.

\bibitem{fattebert-hornung-wissink07}
J.L. Fattebert, R.D. Hornung, and A.M. Wissink,
{\em Finite element approach for density functional theory calculations on locally refined
meshes}, J. Comput. Phys., {\bf 223} (2007), pp. 759-773.

\bibitem{fournais-hoffmann-ostenhof09}
S.~Fournais, M.~Hoffmann-Ostenhof, T.~Hoffmann-Ostenhof, and T.
$\emptyset$.~S$\emptyset$rensen, {\em Analystic structure of
 many-body Coulombic wave functions}, Comm. Math. Phys.,
{\bf 289} (2009), pp.~291-310.

\bibitem{garau-morin11} E.M. Garau and P. Morin,
{\em Convergence and quasi-optimality of adaptive FEM for Steklov eigenvalue problems},
IMA J. Numer. Anal., {\bf 31} (2011), pp. 914-946.

\bibitem{garau-morin-zuppa09} E.M. Garau, P. Morin, and C. Zuppa, {\em
Convergence of adaptive finite element methods for eigenvalue
problems}, M$^3$AS, {\bf 19} (2009), pp.~721-747.

%\bibitem{garau-morin09} E.M. Garau and P. Morin, {\em Convergence
%and quasi-optimality of adaptive FEM for Steklov eigenvalue
%problems}, IMA J. Numer. Anal., {\bf 31} (2011), pp. 914-946.

\bibitem{giani-graham09} S. Giani and I. G. Graham,
{\em A convergent adaptive
method for elliptic eigenvalue problems}, SIAM J. Numer. Anal., {\bf
47} (2009), pp. 1067-1091.

\bibitem{gong-shen-zhang-zhou08}
X. Gong, L. Shen, D. Zhang, and A. Zhou,
{\em Finite element approximations for Schr{\" o}dinger
equations with applications to electronic structure computations},
J. Comput. Math., {\bf 23} (2008), pp. 310-327.

%\bibitem{he-thesis12} L. He, {\em ....}.

\bibitem{he-zhou11} L. He and A. Zhou,
{\em Convergence and complexity of adaptive
finite element methods for elliptic partial differential equations},
Inter. J. Numer. Anal. Model., {\bf 8} (2011), pp. 615-640.

\bibitem{heuveline-rannacher01}  V. Heuveline and R. Rannacher,
{\em A posteriori error control for finite element approximations of
ellipic eigenvalue problems}, Adv. Comput. Math., {\bf 15} (2001),
pp. 107-138.

\bibitem{hohenberg-kohn64} P. Hohenberg and W. Kohn,
{\em Inhomogeneous electron
gas}, Phys. Rev. B, {\bf 136} (1964), pp. 864-871.

\bibitem{kohn-sham65} W. Kohn and L.J. Sham,
{\em Self-consistent equations including exchange and correlation
effects}, Phys. Rev. A, {\bf 140} (1965), pp. 1133-1138.

\bibitem{larson00} M.G. Larson,  {\em A posteriori and a priori error
analysis for finite element approximations of self-adjoint elliptic
eigenvalue problems}, SIAM J. Numer. Anal., {\bf 38} (2000), pp.
608-625.

\bibitem{LYP88} C. Lee, W. Yang, and R.G. Parr,
{\em Development of the Colic-Salvetti correlation-energy formula
into a functional of the electron density},
Phys. Rev. B, {\bf 37} (1988), pp. 785-789.

%\bibitem{maday03} Y. Maday and G. Turinici, {\em
%Error bars and quadratically convergent methods
%for the numerical solution of the Hartree-Fock equations},
%Numer. Math., {\bf 94} (2003), pp. 739-770.

\bibitem{mao-shen-zhou06} D. Mao, L. Shen, and A. Zhou,
{\em  Adaptive finite element algorithms for eigenvalue problems based on local
averaging type a posteriori error estimates},
Adv. Comput. Math., {\bf 25} (2006), pp. 135-160.

\bibitem{martin04} R.M. Martin, {\em Electronic Structure:
Basic Theory and Practical Method}, Cambridge University Press,
Cambridge, 2004.

\bibitem{mekchay-nochetto05} K. Mekchay and  R.H. Nochetto,
{\em Convergence
of adaptive finite element methods for general second order linear
elliplic PDEs}, SIAM J. Numer. Anal., {\bf 43} (2005), pp.
1803-1827.

%\bibitem{morin-nochetto-siebert00} P. Morin,  R.H. Nochetto, and K. Siebert,
% {\em Data oscillation and convergence of adaptive FEM}, SIAM J.
%Numer. Anal., {\bf 38} (2000), pp. 466-488.

\bibitem{morin-nochetto-siebert02} P. Morin,  R.H. Nochetto,  and
K. Siebert,  {\em Convergence of adaptive finite element methods},
SIAM Review, {\bf 44} (2002), pp. 631-658.

\bibitem{morin-siebrt-veeser08} P. Morin, K.G. Siebert, and
A. Veeser, {\em A basic convergence result for conforming adaptive
finite elements}, Math. Models Methods Appl. Sci., {\bf 18} (2008),
pp.~707-737.

\bibitem{motamarri12} P. Motamarri, M.R. Nowak, K. Leiter,
J. Knap, and V. Gavini, {\em Higher-order adaptive finite-element
methods for Kohn-Sham density functional theory}, J. Comput. Phys., {\bf 253} (2013), pp. 308-343.


%\bibitem{nochetto-schmidt-siebert-veeser06} R.H. Nochetto,
%A. Schmidt, K.G. Siebert, and
%A. Veeser, {\em Pointwise a posteriori error estimates for monotone
%semi-linear equations}, Numer. Math., {\bf 104}  (2006),  pp.
%515-538.
%
%\bibitem{nochetto-siebert-veeser09} R.H. Nochetto, K.G. Siebert, and A. Veeser,
%{Theory of adaptive finite element methods: An introduction},
%Multiscale, Nonlinear and Adaptive Approximation,  Springer, 2009,
%pp. 409-542.

%\bibitem{paraview} {\em http://www.paraview.org/}

\bibitem{parr-yang94}R.G. Parr and W.T. Yang,
{\em Density-Functional Theory of Atoms and Molecules}, Oxford
University Press, New York, Clarendon Press, Oxford, 1994.

\bibitem{pask01} J.E. Pask, B.M. Klein, P.A. Sterne, and C.Y. Fong,
{\em Finite-element methods in
electronic-structure theory}, Comput. Phys. Commun., {\bf 135}
(2001), pp. 1-34.

\bibitem{pask-sterne05} J. Pask and P. Sterne, {\em Finite
element methods in ab initio electronic structure calculations},
Modelling Simul. Mater. Sci. Eng., {\bf 13} (2005), pp.~R71-R96.

\bibitem{PBE} J.P. Perdew, K. Burke, and M. Ernzerhof,  {\em Generalized
gradient approximation  made simple}, Phys. Rev.
Lett., {\bf 77} (1996), pp.~3865-3868.

%\bibitem{pickett89} W.E. Pickett, {\em Pseudopotential methods
%in condensed metter applications}, Comput. Phys. Reports, {\bf 9}
%(1989), pp.~115-198.

\bibitem{reed-simon75} M. Reed and B. Simon, {\em Methods of Modern
Mathematical Physics}, Vol. 2, Academic Press, New York, 1975.

%\bibitem{repin00} S. Repin, {\em A posteriori error estimation for nonlinear
%variational problems by duality theory}, J. Math. Sci., {\bf 99}
%(2000), pp. 927-935.

\bibitem{schneider09} R. Schneider, T. Rohwedder, A. Neelov, and
J. Blauert, {\em Direct minimization for calculating invariant
subspaces in density functional computations of the electronic
structure}, J. Comput. Math., {\bf 27} (2009), pp. 360-387.

\bibitem{shen-thesis05} L. Shen, {\em Parallel Adaptive Finite Element Algorithms for Electronic
Structure Computing based on Density Functional Theory}, Ph.D. Thesis, Academy
of Mathematics and Systems Science, Chinese Academy of Sciences, Beijing, 2005.

\bibitem{shen-zhou06} L. Shen and A. Zhou, {\em A defect correction scheme for
finite element eigenvalues with applications to quantum chemistry},  SIAM J. Sci. Comput.,
{\bf 28} (2006), pp. 321-338.


%\bibitem{schwinger11} S. Schwinger, {\em
%A Posteriori Error Analysis of Elective One-Particle
%Electronic Structure Calculations}, Phd thesis,
%Max-Planck-Institute for Mathematics in the Sciences, Leipzig, Germany, 2011.

\bibitem{stevenson07}R. Stevenson,  {\em Optimality of
a standard adaptive finite element method}, Found. Comput. Math.,
{\bf 7} (2007), pp. 245-269.

\bibitem{stevenson08} R. Stevenson, {\em The completion of locally
refined simplicial partitions created by bisection}, Math. Comput.,
{\bf 77} (2008), pp. 227-241.


\bibitem{suryanarayana-gavini-etal10} P. Suryanarayana, V. Gavini,
T. Blesgen, K. Bhattacharya, and M. Ortiz, {\em Non-periodic
finite-element formulation of Kohn-Sham density functional theory,}
J. Mech. Phys. Solids, {\bf 58} (2010), pp. 256-280.

\bibitem{torsti06} T. Torsti, T. Eirola, J. Enkovaara,
T. Hakala, P. Havu, V. Havu,
T. Hoynalanmaa, J. Ignatius, M. Lyly, I. Makkonen, T.T. Rantala, J.
Ruokolainen, K. Ruotsalainen, E. Rasanen, H. Saarikoski, and M.J.
Puska,{\em Three real-space discretization techniques in electronic
structure calculations}, Physica Status Solidi B, {\bf 243} (2006),
pp. 1016-1053.

%\bibitem{troullier90} N. Troullier and J.L. Martins,
%{\em A straightforward method for
%generating soft transferable pseudopotentials}, Solid State Comm.,
%{\bf 74} (1990), pp.~613-616.

\bibitem{tsuchida-tsukada95} E. Tsuchida and M. Tsukada, {\em
Electronic-structure calculations based on the finite-element method},
Phys. Rev. B, {\bf 52} (1995), pp.~5573-5578.

\bibitem{tsuchida-tsukada96} E. Tsuchida and M. Tsukada, {\em
Adaptive finite-element method for electronic-structure calculations},
Phys. Rev. B, {\bf 54} (1996), pp.~7602-7605.

\bibitem{tsuchida-tsukada98} E. Tsuchida and M. Tsukada, {\em
Large-scale electronic-structure calculations based on the adaptive finite-element method},
J. Phys. Soc. Jpn., {\bf 67} (1998), pp. 3844-3858.

%\bibitem{veeser02}  A. Veeser, {\em Convergent adaptive finite elements for the
%nonlinear Laplacian}, Numer. Math., {\bf 92}  (2002),  pp. 743-770.

\bibitem{verfurth96} R. Verf{\" u}rth,
{\em A Review of a Posteriori Error Estimates and
Adaptive Mesh-Refinement Techniques},
Wiley-Teubner, New York, 1996.


\bibitem{yzhang-thesis11} Z. Yang, {\em Finite Volume Discretization Based
First-Principles Electronic Structure
Calculations}, Ph.D. Thesis, Academy
of Mathematics and Systems Science, Chinese Academy of Sciences, Beijing, 2011.


\bibitem{yesrentant10} H. Yserentant, {\em Regularity and Approximability of Electronic Wave Functions},
 Lecture Notes in Mathematics, Springer-Verlag, Berlin, 2010.

\bibitem{zhang-thesis07} D. Zhang,
{\em Applications of Finite Element Methods in Electronic Structure Calculations},
Ph.D. Thesis, Fudan University, 2007.

\bibitem{zhang-shen-zhou-gong08} D. Zhang, L. Shen, A. Zhou, and X.
Gong, {\em Finite element method for solving Kohn-Sham equations
based on self-adaptive terahedral mesh},
Phy. Lett. A, {\bf 372} (2008), pp. 5071-5076.

\bibitem{zhang08} D. Zhang, A. Zhou,  and X. Gong,
{\em Parallel mesh refinement of higher order finite
elements for electronic structure calculations},
Commun. Comput. Phys., {\bf 4} (2008), pp. 1086-1105.

\bibitem{zhou04} A. Zhou,  {\em An analysis for finite dimensional approximations for the
ground state solution of Bose-Einstein condensates},
Nonlinearity, {\bf 17} (2004), pp. 541-550.

\bibitem{zhou07} A. Zhou, {\em Finite dimensional approximations for the electronic
ground state solution of a molecular system}, Math. Meth. Appl.
Sci., {\bf 30} (2007),  pp. 429-447.


\bibitem{phg} PHG, {\em http://lsec.cc.ac.cn/phg/}.

\end{thebibliography}
\end{document}